\documentclass{amsart}
\usepackage{amssymb, amsmath, latexsym, amscd, graphicx, tabularx, color}
\usepackage{rotating}
\usepackage[toc,page]{appendix}

\usepackage[all]{xy}
\usepackage[dvipdfm]{hyperref}
\usepackage[all]{hypcap}

\newtheorem{theorem}{Theorem}[section]
\newtheorem{lemma}[theorem]{Lemma}
\newtheorem{cor}[theorem]{Corollary}
\newtheorem{definition}[theorem]{Definition}
\newtheorem{proposition}[theorem]{Proposition}
\newtheorem{remark}[theorem]{Remark}
\newtheorem{example}[theorem]{Example}
\newtheorem{conj}[theorem]{Conjecture}
\newtheorem{atheorem}{Theorem A\hskip -1pt}[section]
\newtheorem{adefinition}{Definition A\hskip -1pt}[section]
\newtheorem{aproposition}{Proposition A\hskip -1pt}[section]
\newtheorem{alemma}{Lemma A.\hskip -1pt}
\newtheorem{aexample}{Example A.\hskip -1pt}

\newtheorem{bdefinition}{Definition B\hskip -1pt}[section]
\newtheorem{bproposition}{Proposition B\hskip -1pt}[section]

\newtheorem{bexample}{Example B.\hskip -1pt}

\newtheorem{cdefinition}{Definition C\hskip -1pt}[section]
\newtheorem{cproposition}{Proposition C\hskip -1pt}[section]

\newtheorem{cexample}{Example C.\hskip -1pt}

\def\pagenumber{1}

\begin{document}
\setcounter{page}{\pagenumber}
\newcommand{\T}{\mathbb{T}}
\newcommand{\R}{\mathbb{R}}
\newcommand{\Q}{\mathbb{Q}}
\newcommand{\N}{\mathbb{N}}
\newcommand{\Z}{\mathbb{Z}}
\newcommand{\tx}[1]{\quad\mbox{#1}\quad}
\parindent=0pt
\def\SRA{\hskip 2pt\hbox{$\joinrel\mathrel\circ\joinrel\to$}}
\def\tbox{\hskip 1pt\frame{\vbox{\vbox{\hbox{\boldmath$\scriptstyle\times$}}}}\hskip 2pt}
\def\circvert{\vbox{\hbox to 8.9pt{$\mid$\hskip -3.6pt $\circ$}}}
\def\IM{\hbox{\rm im}\hskip 2pt}
\def\COIM{\hbox{\rm coim}\hskip 2pt}
\def\COKER{\hbox{\rm coker}\hskip 2pt}
\def\TR{\hbox{\rm tr}\hskip 2pt}
\def\GRAD{\hbox{\rm grad}\hskip 2pt}
\def\RANK{\hbox{\rm rank}\hskip 2pt}
\def\MOD{\hbox{\rm mod}\hskip 2pt}
\def\DEN{\hbox{\rm den}\hskip 2pt}
\def\DEG{\hbox{\rm deg}\hskip 2pt}

\title[The Congruent Numbers and the Birch Swinnerton-Dyer conjecture]{THE CONGRUENT NUMBER PROBLEM \\ AND THE BIRCH SWINNERTON-DYER CONJECTURE}

\author{Agostino Pr\'astaro}
\maketitle
\vspace{-.5cm}
{\footnotesize
\begin{center}
Department SBAI - Mathematics, University of Rome La Sapienza,\\ Via A.Scarpa 16,
00161 Rome, Italy. \\
E-mail: {\tt agostino.prastaro@uniroma1.it}
\end{center}
}
\vskip 0.5cm
--
\vskip 0.5cm
\begin{abstract}
By introducing a new point of view in Algebraic Topology relating elliptic curves in $\mathbb{R}^2$ and suitable bordism groups, the congruent number problem is solved showing that the Tunnell's theorem is also sufficient. This could be considered also an indirect proof that the Birch Swinnerton-Dyer conjecture is true.
\end{abstract}

\vskip 0.5cm

\noindent {\bf AMS Subject Classification:} 11M26; 14H55; 30F30; 32J05; 33C80; 33C99; 33D90; 33E90; 55N35.

\vspace{.08in} \noindent \textbf{Keywords}:

\section[Introduction]{\bf Introduction}\label{introduction-section}

\vskip 0.5cm

The congruent number problem was a longstanding open problem in Number Theory, that more recently has been related also to the famous Birch and Swinnerton-Dyer conjecture. The vast literature and the Clay-prize give evidence the Mathematical Community's interest on this subject. For a very good introduction about it is advisable to look the paper by A. Wiles \cite{WILES2} \footnote{For complementary information see also the following Wikipedia link:

\href{http://en.wikipedia.org/wiki/Birch_and_Swinnerton-Dyer_conjecture}{\tt Birch-and-Swinnerton-Dyer-conjecture}.}

In this paper we solve the congruent problem by recasting it in the algebraic topology of suitable bordism groups. In fact, we introduce two new bordism groups, ({\em $n$-elliptic bordism groups} and {\em$n$-congruent-bordism-groups}), in the plane $\mathbb{R}^2$ where are considered elliptic curves associated to the congruent number problem. (See Definition \ref{n-elliptic-bordism-groups} and Definition \ref{n-congruent-bordism-groups}.) We show, by utilizing such bordism groups, the exactness of the following short $0$-sequence:
\begin{equation}\label{introduction-short-exact-sequence}
    \xymatrix{0\ar[r]&\mathbb{N}_{congr}\ar[r]^{i}&{}_{\square}\mathbb{N}\ar[r]^{L_\bullet}&\mathbb{Z}\ar[r]&{\rm coker}(L_\bullet)\ar[r]&0\\}
\end{equation}

Here ${}_{\square}\mathbb{N}$ is the subset of $\mathbb{N}$ of square-free numbers, and $\mathbb{N}_{congr}$ is the subset of  ${}_{\square}\mathbb{N}$ of square-free congruent numbers ({\em strong-congruent numbers}). The mapping $L_\bullet$ is defined in Lemma \ref{elleiptic-congruent-bordism-groups-and-birch-swinnerton-dyer-conjecture}. We show that ${\rm ker}(L_\bullet)=\mathbb{N}_{congr}$, namely that sequence (\ref{introduction-short-exact-sequence}) is exact. (See Theorem \ref{main-theorem}.) This gives an indirect proof that the Birch and Swinnerton-Dyer conjecture is true. (See Remark \ref{main-corollary}.) Taking into account that the set $\mathbb{P}_{yth}$ of integers interpretable as areas of Pythagorean right triangles, surjectively projects on $\mathbb{N}_{congr}$, we can obtain all possible Pythagorean triangles. Furthermore, each strong-congruent number $n\in \mathbb{N}_{congr}$ identifies an equivalence class in the set of congruent numbers. Any congruent number $q$ belonging to the equivalence class $[n]$, is obtained from $n$ by multiplying $n$ for a square $m^2$, with $m\in \mathbb{Q}$: $q=m^2n$.

The paper, after the Introduction, splits into three more sections and three appendices. 2. The congruent number problem. [This is a preparatory section, where some fundamental results are recast in the paper-style.] 3. The Birch and Swinnerton-Dyer conjecture. [Here are resumed some important results that are central for our proof. Theorem \ref{modularity-theorem} (Modularity theorem); Conjecture \ref{definition-bs-d-conjecture} (The  Birch and Swinnerton-Dyer conjecture); Proposition \ref{tunnell-theorem} (Tunnell's theorem); Theorem \ref{coates-wiles-theorem} (Coates-Wiles theorem).] 4. Elliptic and congruent bordism groups. [In this section are contained the main results. It contains the definitions of elliptic bordism groups and congruent bordism groups and their characterizations with respect to diffeomorphisms and suitable homotopies in $\mathbb{R}^2$. Theorem \ref{main-theorem} contains the solution of the congruent problem in Number Theory. Remark \ref{main-corollary} emphasizes that our solution of the Congruent number problem can be considered also an indirect proof that the Birch and Swinnerton-Dyer conjecture is true.] Appendix A. The function $L(E,s)$ and the infinitude of primes. Appendix B. Riemann surfaces and modular curves. Appendix C. Modular functions, forms and cusps.\footnote{These subjects are included in this paper in order to satisfy its expository style. We have adopted this style do not make the paper beyond any mathematical grasp, since the mathematics involved here touches sectors that can be considered far from the standard Number Theory. On the other hand experts in Algebraic Topology do not necessarily are also well introduced in Number Theory ...}

\section{\bf The congruent number problem}\label{section-congruent-number-problem}
\vskip 0.5cm

In this section we shall consider some fundamental definitions and results about the congruent number problem that will be utilized in the next sections.

\begin{definition}[Congruent number]\label{congruent-number}
A {\em congruent number} is any positive rational number $q$ such that there exists a right triangle of sides $(a,b,c)$, $a,\, b,\, c\in \mathbb{Q}$, with area $q$, namely $\frac{a\cdot b}{2}=q\in\mathbb{Q}$. (Here $c$ is the length of the hypotenuse.) We denote by $[q|a,b,c]$ a congruent number $q$, with its corresponding right triangle of sides $(a,b,c)$. Let us denote by $\mathbb{Q}_{congr}\subset\mathbb{Q}$ the subset of congruent numbers.
\end{definition}

\begin{lemma}[Congruent numbers identified via a subset of natural numbers]\label{lemma-residue-classes-congruent-numbers}

$\bullet$\hskip 2pt If $q$ is a congruent number, then also $s^2\cdot q$ is so, for any $s\in\mathbb{Q}$.

$\bullet$\hskip 2pt There is an equivalence relation $\sim$ between congruent numbers, such that each equivalence class is identified by a square-free positive integer. In other words, the set $\mathbb{N}_{congr}$, of congruent numbers, up to rational-conform equivalence, can be identified with a subset of ${}_{\square}\mathbb{N}\subset \mathbb{N}$, where ${}_{\square}\mathbb{N}$ is the set of square-free integers contained into $\mathbb{N}$: $\mathbb{Q}_{congr}/\sim\, \cong \mathbb{N}_{congr}\subset\mathbb{N}$. One has the commutative and exact diagram {\rm(\ref{commutative-exact-diagram-stron-congruent-numbers-and-rational})}.
\end{lemma}
\begin{equation}\label{commutative-exact-diagram-stron-congruent-numbers-and-rational}
     \scalebox{0.8}{$\xymatrix{\mathbb{Q}^{\times}\ar[r]^{()^2}&\mathbb{Q}^{\times}\ar[r]&\mathbb{Q}^{\times}/(\mathbb{Q}^{\times})^2\ar@{=}[d]^{\wr}\ar[r]&0\\
    &&{}_{\square}\mathbb{N}&\\
    \mathbb{Q}_{congr}\ar[uu]\ar[r]&\mathbb{Q}_{congr}/\sim\ar@{=}[r]^{\sim}&\mathbb{N}_{congr}\ar[u]\ar[r]&0\\
    0\ar[u]&&0\ar[u]&\\}$}
\end{equation}

\begin{proof}
$\bullet$\hskip 2pt In fact, if $[q|a,b,c]$ is a congruent number with associated right triangle, then also $[s^2\cdot q|s\cdot a,s\cdot b,s\cdot c]$ is a congruent number with associated right triangle. The proof is direct.

$\bullet$\hskip 2pt Therefore we can consider equivalent right triangles $(a,b,c)$ and $(\bar a,\bar b,\bar c)$, identified by congruent numbers $q$ and $\bar q$, respectively, whether they are rational-conform, namely $\bar a=s\cdot a$, $\bar b=s\cdot b$, $\bar c=s\cdot c$, $s\in\mathbb{Q}$, iff $\bar q=s^2\cdot q$. As a by-product, it follows that a congruent number identifies an equivalence class in the group $\mathbb{Q}^{\times}/(\mathbb{Q}^{\times})^{2}$. Here every residue class contains one square-free positive integer, that can be utilized to identify the class.\footnote{This is the motivation, that allows to talk about congruent numbers, simply as square-free positive integers. (Recall that a {\em square-free} integer is one divisible by no perfect square, except $1$. So a square-free integer is one $n\in\mathbb{N}$ such that in its prime decomposition $n=p_1^{\alpha_1}\cdots p_k^{\alpha_k}$ all exponent $\alpha_i$, $i=1,\cdots,k$, are $\alpha_i=1$. This is equivalent to say that the ring $\mathbb{Z}_n=\mathbb{Z}/n\mathbb{Z}$ is a product of fields, $\mathbb{Z}_n=\mathbb{Z}_{p_1}\cdots\mathbb{Z}_{p_k}$, since $\mathbb{Z}_{p_i}$, $i=1,\cdots,k$, are fields ($p_i$ is prime). For example are square-free integers $1,\, 2,\, 3,\, 5,\, 6,\, 7,\, 10,\, 11,\, 13$. Instead are not square-free integers $9=3^2$ and $20=2^2\cdot 5$.)} Therefore any equivalence class of congruent numbers, obtained by identifying conform right triangles, can be represented by a square-free integer. Let us denote by $\mathbb{N}_{congr}$ the set of such natural numbers, and by ${}_{\square}\mathbb{N}$ the set of square-free natural numbers. Then $\mathbb{N}_{congr}\subset{}_{\square}\mathbb{N}\subset \mathbb{N}$.\footnote{$\mathbb{N}_{congr}$ is a proper subset of ${}_{\square}\mathbb{N}$. For this it is enough to look at the Tab. \ref{table-examples-of-strong-congruent-numbers}. On the other hand it is well known that the square-free numbers $1,\, 2,\, 3\in{}_{\square}\mathbb{N}$ are not congruent numbers.}
\end{proof}

The characterization of square-free positive integers can be made also by means of the M\"obius function.
\begin{definition}[M\"obius function]\label{mobius-function}
The {\rm M\"obius function} $\mu(n)$ is defined in {\em(\rm\ref{formula-mobius-function})}.
\begin{equation}\label{formula-mobius-function}
  \mu(n)=\left\{\begin{array}{ll}
                  1 &,\, \hbox{\rm if $n=1$} \\
                  (-1)^k&, \, \hbox{\rm if $n=p_1\cdots p_k,\, p_i\in P$}\\
                  0&,\, \hbox{\rm if $n$ has a squared prime factor}.\\
                \end{array}\right.
\end{equation}
In {\em(\ref{formula-mobius-function})} $P$ is the set of primes.
\end{definition}
\begin{proposition}[Properties of M\"obius function]\label{properties-mobius-function}

$\bullet$\hskip 2pt $\mu(n)$ is a multiplicative function: $\mu(n_1\cdot n_2)=\mu(n_1)\cdot\mu(n_2)$, $n_1$ and $n_2$ coprime.

$\bullet$\hskip 2pt {\rm(Dirichlet series that generates the M\"obius function)} $$\sum_{1\le n\le\infty}\frac{\mu(n)}{n^s}=\frac{1}{\zeta(s)},\, s\in\mathbb{C},\, \Re(s)>1.$$

$\bullet$\hskip 2pt $n\in{}_{\square}\mathbb{N}$ iff $\mu(n)\not=0$.
\end{proposition}

\begin{definition}[Strong-congruent numbers]\label{strong-congruent-numbers}
We define {\em strong-congruent numbers} the integers belonging to $\mathbb{N}_{congr}$.
\end{definition}

\begin{proposition}[Strong-congruent numbers and Pythagorean triangles]\label{congruent-numbers-and-pythagorean-triangles}

$\bullet$\hskip 2pt A way to obtain strong-congruent numbers $n\in \mathbb{N}_{congr}\subset{}_{\square}\mathbb{N}$ is related to pass trough Pythagorean triangles, by using the following parametric formula:
\begin{equation}\label{formula-parametrization-pythagorean-numbers}
\left\{
\begin{array}{l}
  \left(s(\kappa^2-l^2), 2s\kappa l, s(\kappa^2+l^2)\right),\, \kappa>l>0,\,  \kappa-l=2r+1\, \forall r\in \{0\}\bigcup\mathbb{N}\\
  k,\, l\, {\rm coprime},\, \forall s\in \mathbb{N}\\
  \end{array}\right.
\end{equation}
 In fact, equations {\em(\ref{formula-parametrization-pythagorean-numbers})} parametrize all the Pythagorean triangles, hence the square-free part of their areas identify all the strong-congruent numbers.\footnote{The parametrization $(\kappa^2-l^2, 2\kappa l, \kappa^2+l^2)$  for {\em primitive} Pythagorean right triangles, ($s=1$), comes from the identity $(\kappa^2-l^2)^2+( 2\kappa l)^2=(\kappa^2+l^2)^2$, that interprets the Pythagorean's theorem for a right triangle of sides $\kappa^2-l^2$, $2\kappa l$ and hypothenuse $\kappa^2+l^2$. Pythagorean  right triangles are congruent triangles with congruent number $q=s^2(\kappa^2-l^2)\kappa l\in\mathbb{N}$. One can write $q=m^2\cdot n$, where $n$ is the square-free part of $q$. Then $[n|\frac{\kappa^2-l^2}{m},\frac{2\kappa l}{m},\frac{\kappa^2+l^2}{m}]$ is the congruent class of right triangles identified by the congruent square-free number $n$.} In Tab. \ref{congruent-numbers-pythagorean} are reported some examples of $n\in\mathbb{N}_{congr}$ built in such a way. Let us denote by $\mathbb{P}_{yth}$ the set of Pythagorean triangles parameterized by $\kappa$, $l$ and $s$:
\begin{equation}\label{set-pythagorean-triangle-k-l-s}
    \mathbb{P}_{yth}=\left\{
    \begin{array}{l}
    q\in\mathbb{N} \, |\, q=s^2(\kappa^2-l^2)\kappa\, l,\, \kappa,\, l\in\mathbb{N},\\
     \kappa>l>0,\, \kappa-l=2r+1,\, \forall r\in\{0\}\bigcup\mathbb{N},\, k,\, l,\, {\rm coprime},\, \forall s\in\mathbb{N}\\
    \end{array}
    \right\}.
\end{equation}
Then one has the exact commutative diagram {\rm(\ref{congruent-numbers-grail})}, {\em(congruent numbers grail)}, where
\begin{equation}\label{congruent-numbers-grail-a}
\left\{
\begin{array}{l}
  a(m)=\hbox{\rm square-free part of $m$}\\
  b=a|_{\mathbb{P}_{yth}}\\
  i=\hbox{\rm natural inclusion}\\
  j=\hbox{\rm natural inclusion}.
\end{array}
\right.
\end{equation}
\begin{equation}\label{congruent-numbers-grail}
\framebox{\parbox[c]{1.in}{\rm Congruent Numbers Grail}}\, \parbox[c]{2.5in}{\scalebox{0.8}{$
\xymatrix{\mathbb{Q}_{congr}\ar@/_2pc/[rddd]\ar@/_1pc/[dr]\ar@{^{(}->}[rrrr]&&&&\mathbb{Q}^{+}\ar@/^1pc/[dl]\ar@/^2pc/[lddd]\\
&\mathbb{P}_{yth}\ar[dd]^{b}\ar@/^0.5pc/[dr]\ar@{^{(}->}[rr]^{j}&&\mathbb{N}\ar[dd]^{a}\ar@/_0.5pc/[dl]&\\
&&0&&\\
&\mathbb{N}_{congr}\ar[dd]\ar@/^0.5pc/[dr]\ar@{^{(}->}[rr]^{i}&&{}_{\square}\mathbb{N}\ar[dd]\ar@/_0.5pc/[dl]&\\
&&0&&\\
&0\ar@{=}[rr]&&0&\\}$}}
\end{equation}
$\bullet$\hskip 2pt $\sharp(\mathbb{P}_{yth})=\aleph_0$. Furthermore one can write $\mathbb{P}_{yth}=\bigcup_{n\in \mathbb{N}_{congr}}(\mathbb{P}_{yth})_n$, where $(\mathbb{P}_{yth})_n=b^{-1}(n)$ is the fiber over $n$, the strong-congruent number $n$, and $\sharp((\mathbb{P}_{yth})_n)=\aleph_0$.

\end{proposition}
\begin{proof}
$\bullet$\hskip 2pt In fact, $a(j(q))=a(j(s^2n))=a(s^2n)=n$ and $i(b(q))=i(b(s^2n))=i(n)=n$. The surjectivity of the mapping $b$ in (\ref{congruent-numbers-grail}) is not obvious. So let assume that $n\in\mathbb{N}_{congr}$, then there exists a congruent right triangle $[n|a,b,c]=[n|\frac{r}{s},\frac{r'}{s'},\frac{r''}{s''}]$ that satisfies the conditions: $\frac{a\cdot b}{2}=\frac{1}{2}\frac{r}{s}\frac{r'}{s'}=n$, hence $rr'=2ss'n$, and $a^2+b^2=c^2$, hence $\frac{r^2s'^2+r'^2s^2}{s^2s'^2}=\frac{r''^2}{s''^2}$. From this last we can assume $s''=ss'$ and $r''=\sqrt{r^2s'^2+r'^2s^2}=\sqrt{p^2}$, where $p\in\mathbb{N}$. Then there exists a Pythagorean triangle $[q=m^2n|\bar a,\bar b,\bar c]$, with $m\in\mathbb{N}$. More precisely one has
\begin{equation}\label{pythagorean-triangle-associated-to-strobg-congruent-one}
    \left\{
    \begin{array}{l}
      \bar a=2ss'a=2ss'\frac{r}{s}=2s'r\\
      \bar b=2ss'b=2ss'\frac{r'}{s'}=2sr'\\
      \bar c=2ss'c=2ss'\frac{r''}{s''}=2ss'\frac{\sqrt{r^2s'^2+r'^2s^2}}{ss'}=2p\\
    \end{array}
    \right.
\end{equation}
In fact one has $\bar a^2+\bar b^2=4(r^2s'^2+r'^2s^2)=4p^2=\bar c^2$, and $\frac{1}{2}\frac{\bar a}{\bar b}=2ss'rr'=(2ss')^2n=m^2n$, with $m=2ss'$. Therefore the square-free part of $\frac{1}{2}\bar a\bar b$ is just $n$. Of course, since are also Pythagorean triangles the following ones $[(lm)^2n|\hat a=l\bar a,\hat b=l\bar b,\hat c=l\bar c]$, $\forall l\in\mathbb{N}$, it follows that the mapping $b:\mathbb{P}_{yth}\to \mathbb{N}_{congr}$ is surjective.
In (\ref{congruent-numbers-grail}) we have also inserted the surjective mappings $\mathbb{Q}^{\times}\to\mathbb{N}$ and $\mathbb{Q}_{congr}\to\mathbb{P}_{yth}$. In this way it is clear that all the congruent numbers can be obtained from $\mathbb{N}_{congr}$.

$\bullet$\hskip 2pt Note that $\mathbb{P}_{yth}$ is an infinite set.\footnote{Let us emphasize that even if the set of Pythagorean triangle is infinite, since from one we can generate infinite other ones with conform transformations, this property could not appear so obvious by looking the parametrization considered. In fact two different couples $(\kappa,l)$ and $(\kappa',l')$, can have equal their $(q=(\kappa^2-l^2)\kappa l)$-value in $\mathbb{P}_{yth}$. For example to the couple $(\kappa,l)=(5,2)$ there corresponds $q=210$; the same $q$-value corresponds to the couple $(\kappa,l)=(6,1)$ too.} In fact, it contains the set $\mathbb{P}_{yth}[1]=\{q\in\mathbb{P}_{yth}\, |\, q=(\kappa^2-l^2)\kappa l,\, l=1\}$, obtained from $\mathbb{P}_{yth}$ by fixing $l=1$. One can see that $\mathbb{P}_{yth}[1]$ is identified by the positive-valued curve $q(r)=2(r+1)(4r^2+8r+3)$, having $q'(r)=24 r^2+48r+22>0$, $\forall r\ge 0$. Therefore $q(r)=q(r')$ iff $r=r'$, for any $r,\, r'\in\{0\}\bigcup\mathbb{N}$.

Furthermore, from the proof of the above point we get $\sharp((\mathbb{P}_{yth})_n)=\aleph_0$.
\end{proof}

\begin{lemma}[Cardinality of $\mathbb{N}_{congr}$]\label{cardinality-square-free-congruent-numbers}
The set of strong-congruent numbers has the same cardinality of $\mathbb{N}$: $\sharp(\mathbb{N}_{congr})=\aleph_0$.
\end{lemma}

\begin{proof}
From Proposition \ref{congruent-numbers-and-pythagorean-triangles} we see that the subset $\mathbb{N}_{congr}$ of $\mathbb{N}$ contains an infinite set. This is identified by means of the square-free parts of the numbers $(\kappa^2-l^2)\kappa l$. This last set is infinite and therefore, is so the set of its square-free parts. This follows from the prime factorization of integers and from the cardinality of $P$, i.e., the set of primes: $\sharp(P)=\sharp(\mathbb{N})=\aleph_0$.\footnote{In Appendix A it is given a proof on the cardinality of $P$ that uses the zeta Riemann function. (See Theorem A1.)} Therefore it must necessarily be $\sharp(\mathbb{N}_{congr})=\aleph_0$.
\end{proof}

\begin{theorem}[First criterion for strong-congruent numbers]\label{first-criterion-for-strong-congruent-numbers}
Let $n\in{}_{\square}\mathbb{N}$ be a square-free integer. Then $n\in\mathbb{N}_{congr}$, namely $n$ is a strong-congruent number, iff there are integers $\kappa$, $l$, $\kappa>l>0$, $\kappa-l\not\cong {\rm mod}\, 2$ such that: $(\kappa^2-l^2)\kappa l=m^2\cdot n$, for some $m\in\mathbb{N}$.
\end{theorem}
\begin{proof}
This criterion follows from the congruent numbers grail (\ref{congruent-numbers-grail}) and the surjectivity of the mapping $b$ there considered.
\end{proof}

\begin{theorem}[Second criterion for strong-congruent numbers]\label{second-criterion-for-strong-congruent-numbers}
A number $n\in{}_{\square}\mathbb{N}$ is a strong-congruent number, namely $n\in\mathbb{N}_{congr}$, iff there exist three positive rational numbers $0<r<s<t$, such that the following conditions are satisfied:

{\rm(i)} $t^2-r^2=2n$;

{\rm(ii)} $t^2+r^2=2s^2$.
\end{theorem}
\begin{proof}
If $n$ is a strong-congruent number, let $[n|a,b,c]$ be its right triangle of area $n$. Set $r=\frac{a-b}{2}$, $t=\frac{a+b}{2}$. Then we get $t^2-r^2=ab=2n$ and $t^2+r^2=\frac{a^2+b^2}{2}=\frac{c^2}{2}$, hence $s=\frac{c}{2}$. Vice versa, if there exist three positive rational numbers $0<r<s<t$, satisfying above conditions (i) and (ii), then we can identify a congruent right triangle $[n|a,b,c]$, with $a=r+t$, $b=t-r$, $c=2s$. In fact, one has $a^2+b^2=4(r^2+t^2)=2s^2=c^2$ and $ab=r^2-t^2=2n$.
\end{proof}
\begin{remark}\label{warning-criterion}
Theorem \ref{first-criterion-for-strong-congruent-numbers} and Thoerem \ref{second-criterion-for-strong-congruent-numbers} do not give a way to always obtain, after a finite number of steps, an answer. In fact, the set $\mathbb{P}_{yth}$ is infinite. For example if $n=5$, after a number of step we can arrive to know that with $\kappa=5$ and $l=4$, one has $(\kappa^2-l^2)\kappa l=m^2\cdot 5$. (See Tab. \ref{congruent-numbers-pythagorean}.) Really we get $9\cdot 5\cdot 4=m^2\cdot 5$, hence $m^2=36=2^2\cdot 3^2=6^2$, namely $m=6$.This means that the strong-congruent right triangle is $[5|\frac{\kappa^2-l^2}{6},\frac{2\kappa l}{6},\frac{\kappa^2+l^2}{6}]=[5|\frac{3}{2},\frac{20}{3},\frac{41}{6}]$, that has just area $5$. But whether $n$ is not a strong-congruent number this process cannot stop ! Therefore Theorem \ref{first-criterion-for-strong-congruent-numbers} does not solve the problem to find an useful algorithm to decide whether a square-free integer is a strong-congruent number. However, since the mapping $b:\mathbb{P}_{yth}\to\mathbb{N}_{congr}$ is surjective, all the possible strong-congruent numbers can be obtained as square-free part of numbers $q\in \mathbb{P}_{yth}$.
\end{remark}

\begin{table}[t]
\caption{Examples of strong-congruent numbers $n\in\mathbb{N}_{congr}$ and triangle class $[n|\frac{a}{m},\frac{b}{m},\frac{c}{m}]$, built from Pythagorean triangles $(\kappa^2-l^2, 2\kappa l, \kappa^2+l^2)$.}
\label{congruent-numbers-pythagorean}
\begin{tabular}{|l|l|l|l|l|}
\hline
{\rm{\footnotesize $\kappa$}}&{\rm{\footnotesize $l$}}&{\rm{\footnotesize Pythagorean-triangle $(a,b,c)$}}&{\rm{\footnotesize Pythagorean-triangle area}}&{\rm{\footnotesize $[n|\frac{a}{m},\frac{b}{m},\frac{c}{m}]$}}\hfill\\
\hline\hline
{\rm{\footnotesize $2$}}&{\rm{\footnotesize $1$}}&{\rm{\footnotesize $(3,4,5)$}}&{\rm{\footnotesize  $6=2\cdot 3$}}&{\rm{\footnotesize $[6|3,4,5]$}}\hfill\\
\hline
{\rm{\footnotesize $3$}}&{\rm{\footnotesize $2$}}&{\rm{\footnotesize $(5,12,13)$}}&{\rm{\footnotesize $30=2\cdot 3\cdot 5$}}&{\rm{\footnotesize $[30|5,12,13]$}}\hfill\\
\hline
{\rm{\footnotesize $4$}}&{\rm{\footnotesize $1$}}&{\rm{\footnotesize $(15,8,17)$}}&{\rm{\footnotesize $60=2^2\cdot 3\cdot 5$}}&{\rm{\footnotesize $[15|\frac{15}{2},4,\frac{17}{2}]$}}\hfill\\
\hline
{\rm{\footnotesize $4$}}&{\rm{\footnotesize $3$}}&{\rm{\footnotesize $(7,24,25)$}}&{\rm{\footnotesize  $84=2^2\cdot 3\cdot 7$}}&{\rm{\footnotesize $[21|\frac{7}{2},12,\frac{25}{2}]$}}\hfill\\
\hline
{\rm{\footnotesize $5$}}&{\rm{\footnotesize $2$}}&{\rm{\footnotesize $(21,20,29)$}}&{\rm{\footnotesize $210=2\cdot 3\cdot 5\cdot 7$}}&{\rm{\footnotesize $[210|21,20,29]$}}\hfill\\
\hline
{\rm{\footnotesize $5$}}&{\rm{\footnotesize $4$}}&{\rm{\footnotesize $(9,40,41)$}}&{\rm{\footnotesize $180=2^2\cdot 3^2\cdot 5$}}&{\rm{\footnotesize $[5|\frac{3}{2},\frac{20}{3},\frac{41}{6}]$}}\hfill\\
\hline
{\rm{\footnotesize $6$}}&{\rm{\footnotesize $1$}}&{\rm{\footnotesize $(35,12,37)$}}&{\rm{\footnotesize $210=2\cdot 3\cdot 5\cdot 7$}}&{\rm{\footnotesize $[210|35,12,37]$}}\hfill\\
\hline
{\rm{\footnotesize $6$}}&{\rm{\footnotesize $5$}}&{\rm{\footnotesize $(11,60,61)$}}&{\rm{\footnotesize $330=2\cdot 3\cdot 5\cdot 11$}}&{\rm{\footnotesize $[330|11,60,61]$}}\hfill\\
\hline
{\rm{\footnotesize $7$}}&{\rm{\footnotesize $2$}}&{\rm{\footnotesize $(45,28,53)$}}&{\rm{\footnotesize  $630=2\cdot 3^2\cdot 5\cdot 7$}}&{\rm{\footnotesize $[70|15,\frac{28}{3},\frac{53}{3}]$}}\hfill\\
\hline
{\rm{\footnotesize $7$}}&{\rm{\footnotesize $4$}}&{\rm{\footnotesize $(33,56,65)$}}&{\rm{\footnotesize  $924=2^2\cdot 3\cdot 7\cdot 11$}}&{\rm{\footnotesize $[231|\frac{33}{2},28,\frac{65}{2}]$}}\hfill\\
\hline
{\rm{\footnotesize $8$}}&{\rm{\footnotesize $1$}}&{\rm{\footnotesize $(63,16,65)$}}&{\rm{\footnotesize  $504=2^2\cdot 3^2\cdot 2\cdot 7$}}&{\rm{\footnotesize $[14|\frac{21}{2},\frac{8}{3},\frac{65}{6}]$}}\hfill\\
\hline
{\rm{\footnotesize $8$}}&{\rm{\footnotesize $3$}}&{\rm{\footnotesize $(55,48,33)$}}&{\rm{\footnotesize  $1320=2^2\cdot 2\cdot 3\cdot 5\cdot 11$}}&{\rm{\footnotesize $[330|\frac{55}{2},24,\frac{73}{2}]$}}\hfill\\
\hline
{\rm{\footnotesize $8$}}&{\rm{\footnotesize $7$}}&{\rm{\footnotesize $(15,112,113)$}}&{\rm{\footnotesize  $840=2^2\cdot 2\cdot 3\cdot 5\cdot 7$}}&{\rm{\footnotesize $[310|\frac{15}{2},56,\frac{113}{2}]$}}\hfill\\
\hline
\multicolumn{5}{l}{\rm{\footnotesize $(\kappa^2-l^2, 2\kappa l, \kappa^2+l^2)$, $\kappa>l>0$, $\kappa\not\equiv l\, {\rm mod}\, 2$, (namely $\kappa-l=2r+1,\, r\ge 0$).}}\hfill\\
\multicolumn{5}{l}{\rm{\footnotesize $n=\frac{a\cdot b}{2m^2}\in \mathbb{N}_{congr}$, with $n$ the Pythagorean-triangle area square-free part.}}\hfill\\
\multicolumn{5}{l}{\rm{\footnotesize With respect to the formula (\ref{set-pythagorean-triangle-k-l-s}) one has taken $s=1$,}}\hfill\\
\multicolumn{5}{l}{\rm{\footnotesize namely one considers primitive Pytagorean triangles only.}}\hfill\\
\end{tabular}
\end{table}

\begin{lemma}[Congruent numbers as rational points of elliptic curves]\label{lemma-congruent-numbers-as-elliptic}

$\bullet$\hskip 2pt A positive rational number $q\in\mathbb{Q}$ is congruent iff the equation $y^2=x^3-q^2\cdot x$ in the plane $\mathbb{R}^2$, has a {\em rational point}, $P$, i.e., a point that has rational coordinates, $P=(x_P,y_P)\in\mathbb{Q}^2$, with $y_P\not=0$. This justifies the definition of congruent numbers as {\em rational elliptic points}.

$\bullet$\hskip 2pt In particular, if $[q|a,b,c]$ is a congruent right triangle class, $q\in\mathbb{Q}_{congr}$, one has that $P=(x_P,y_P)=(\frac{qb}{c-a},\frac{2q^2}{c-a})$ is a rational point on $y^2=x^3-q^2 x$. Vice versa, if $P=(x_P,y_P)$, $y_P\not=0$, is a rational point on $y^2=x^3-q^2x$, then the congruent right triangle class $[q|a,b,c]$ has $(a,b,c)=(\frac{x^2-q^2}{y},\frac{2qx}{y},\frac{x^2+q^2}{y})$.
\end{lemma}

\begin{proof}
This can be directly proved by considering that the set
\begin{equation}\label{set-a-n}
A[q]=\{(a,b,c)\in\mathbb{R}^3\, |\, a^2+b^2=c^2,\, a\cdot b=2q,\, q\in\mathbb{R}\},
\end{equation}
is in correspondence one-to-one with the set

\begin{equation}\label{set-e-n}
E[q]=\{(x,y)\in\mathbb{R}^2\, |\, y^2=x^3-q^2 x,\, y\not= 0,\, q\in\mathbb{R}\}.
\end{equation}

This is realized by the following explicit expressions of $(x,y)$ by means of $(a,b,c)$ and $q\in\mathbb{R}$:
\begin{equation}\label{relation-congruent-right-triangle-point}
x=q\frac{a+c}{b},\,  y=2q^2\frac{a+c}{b^2}
\end{equation}
and the converse expressions:
 \begin{equation}\label{relation-point-congruent-right-triangle}
 a=\frac{x^2-q^2}{y},\, b=2q\frac{x}{y},\, c=\frac{x^2+q^2}{y}.
\end{equation}
In fact, one can directly verify that the point $(x,y)\in\mathbb{R}^2$, given in (\ref{relation-congruent-right-triangle-point}), satisfies equation $y^2-x^3+q^2 x=0$, when $(a,b,c)\in A[q]$. Furthermore, by using expressions (\ref{relation-point-congruent-right-triangle}) for $a\, b,\, c$ in (\ref{relation-congruent-right-triangle-point}) one has the identity $x=x$ and $y=y$, that proves that the map $E[q]\to A[q]$ given in (\ref{relation-point-congruent-right-triangle}) is just the inverse of the one $A[q]\to E[q]$ given in (\ref{relation-congruent-right-triangle-point}).

Furthermore, $a$, $b$ and $c$ are positive iff $x$ and $y$ are positive. By restriction on $\mathbb{Q}\subset\mathbb{R}$, we get that a positive rational number $q\in\mathbb{Q}$ is congruent iff the equation $y^2=x^3-q^2 x$ has a rational point with $y\not=0$. Let us also underline that equation (\ref{real-elliptic-system}) seen in the field of real numbers, namely with $a,\, b,\, c\in\mathbb{R}$, admits always solutions under the condition $c\ge 2\sqrt{n}$, for any $q\in\mathbb{R}$.
\begin{equation}\label{real-elliptic-system}
  \left\{
  \begin{array}{ll}
   (\Gamma)&a^2+b^2-c^2=0\\
   (\Upsilon)& ab-2q=0\hskip 5pt.\\
  \end{array}
  \right.
\end{equation}
In fact the first equation in (\ref{real-elliptic-system}) represents in the plane $\mathbb{R}^2$, $(a,b)$, a circle $\Gamma$ of center $O=(0,0)$ and radius $c$. The second equation in (\ref{real-elliptic-system}) represents an equilateral hyperbola $\Upsilon$ of center $O$ and vertex $V=(\sqrt{2q},\sqrt{2q})$. Thus $\Gamma\bigcap\Upsilon\not=\varnothing$ iff $c\ge \overline{OV}=2\sqrt{q}$, for any fixed integer $q\in\mathbb{R}$. Therefore, taking into account the above considerations, we conclude that for any positive rational $q$ the set of solutions of (\ref{real-elliptic-system}) for the field $\mathbb{R}$, is an elliptic curve $E[q]$ of equation $y^2-x^3+q^2x=0$, in the plane $\mathbb{R}^2$, $(x,y)$.\footnote{It is useful to emphasize that in $\mathbb{R}^3$, where $(a,b,c)$ represents a point, equations $(\Gamma)$ and $(\Upsilon)$ in (\ref{real-elliptic-system}), identify surfaces, that we just denote with the symbols $\Gamma$ and $\Upsilon$, respectively. Therefore $E[q]=\Gamma\bigcap\Upsilon$, namely $E[q]$ is the curve intersection of such surfaces. From above calculations we know that such a curve belongs to a plane $\pi\subset \mathbb{R}^3$, where with respect to a suitable coordinate system $(x,y)$, it is represented by the equation $y^2-x^3+q^2x=0$. By considering the natural inclusions $\mathbb{R}^2\subset\mathbb{R}^2\bigcup\{\infty\}= S^2$,  $\mathbb{R}^3\subset\mathbb{R}^3\bigcup\{\infty\}= S^3=G^+_{1,4}(\mathbb{R}^4)\cong \mathbb{P}^3(\mathbb{R})$, where $G^+_{1,4}(\mathbb{R}^4)$ is the oriented Grassmann manifold of oriented $1$-dimensional subspaces of $\mathbb{R}^4$ that pass through the origin $O\in\mathbb{R}^4$, we obtain a more satisfactory representation of $E[q]$. ( See Fig. \ref{representation-elliptic-curve} and Fig. \ref{representation-compactified-elliptic-curve} for some pictures showing that such curves are represented by two connected components). (Recall that $G^+_{1,4}(\mathbb{R}^4)$ is an analytic manifold of dimension $3$ with $\pi_1(G^+_{1,4}(\mathbb{R}^4))=\mathbb{Z}_2$ and having the cell decomposition: $\infty=\mathbb{R}^0\subset\mathbb{R}\subset\mathbb{R}^2\subset\mathbb{R}^3$.)
Note that in $S^3$ the intersection curve $\Sigma^1=\Gamma\bigcap\Upsilon$ is not contained in an affine plane, but is contained in $S^2\subset S^3$. Then this can be projected by means of the stereographic projection, of $S^2$, from $\infty$, say the north pole, to the plane $\pi$, tangent $S^2$ at the south pole $O$. In such a way $\Sigma^1$ identifies a plane curve $E[n]$ on $\pi$. (See Fig. \ref{representation-compactified-elliptic-curve}(B) and Fig. \ref{representation-compactified-elliptic-curve}(C).)}  To find solutions of (\ref{real-elliptic-system}) in the sub-field $\mathbb{Q}\subset\mathbb{R}$, it is enough to find rational points of $E[q]$. In fact a rational solution of (\ref{real-elliptic-system}) is in correspondence one-to-one with the rational points of $E[q]$,  as can be directly verified by using the above transformation $(a,b,c)\leftrightarrow(x,y)$.
\end{proof}

\begin{lemma}[Congruent numbers classes as set of rational points on the same class of elliptic curves]\label{lemma-congruence-equivalent-elliptic-curve}
If $f:\mathbb{R}^2\to\mathbb{R}^2$ is a rational-conform transformation of $\mathbb{R}^2$, sending a congruent rational right triangle $[q|a,b,c]$ into a conform one $[\bar q=s^2q|\bar a=sa,\bar b=sb,\bar c=sc]$, the corresponding elliptic curves $y^2=x^3-q^2 x$ and $\bar y^2=\bar x^3-\bar q^2 \bar x$ are not related by the same conform transformation $f$, but by a different diffeomorphism $\varphi:\mathbb{R}^2\to\mathbb{R}^2$, $\varphi\not=f$, of $\mathbb{R}^2$. As a by-product even if equivalent congruent right triangle are conform each other, the corresponding elliptic curves are not conform each other, but diffeomorphic only. We can identify two elliptic curves $E[q]$ and $E[\bar q]$, $\bar q=s^2 q$, by means of the induced diffeomorphism $\varphi:E[q]\to E[\bar q]$.
\end{lemma}

\begin{proof}
In fact the conform transformation $f$ is given by $f:(x,y)\mapsto (\bar x=sx,\bar y=sy)$. This transformation has the effect to reduce or augment any right triangle into another one having the same angles, hence with sides that are parallel to the original ones. Instead the transformation  $\varphi:\mathbb{R}^2\to\mathbb{R}^2$ relating the two elliptic curves, is the following: $\varphi:(x,y)\mapsto(\bar x=s^2 x,\bar y=s^3 y)$. This can be seen by a direct calculation. In fact, $\bar x=\bar n\frac{(\bar a+\bar c)}{\bar b}=n s^2\frac{a+c}{b}=s^2x$ and $\bar y=2\bar n^2\frac{\bar a+\bar c}{\bar b^2}=2n^2s^3\frac{a+c}{b^2}=s^3y$.
Thus $\varphi$ is not a conform transformation of $\mathbb{R}^2$, $(x,y)$ ! This transformation identifies a diffeomorphism between $E[q]$ and $E[\bar q]$, with jacobian $j(\varphi)=s^5\not=0$. In fact if we substitute $(\bar x=s^2x,\bar y=s^3 y)$ in the curve $\bar y^2-\bar x^3+\bar q^2\bar x=0$ we get $s^6(y^2-x^3+q^2x=0)$. This means that the transformed point $(x,y)$ belongs to the original elliptic curve of equation $y^2-x^3+q^2x=0$. In particular if $P=(x_P,y_P)$ is a rational point of such an elliptic curve, then also its transformed $\bar P=(\bar x_P,\bar y_P)$ belongs to the transformed elliptic curve. In other words, to any conform transformation $f:\mathbb{R}^2\to\mathbb{R}^2$, corresponds a transformation of elliptic curves: $\varphi:E[q]\to E[\bar q]$. In this sense we can identify all the elliptic curves $E[\bar q]$, with the unique one $E[n]$, $n\in\mathbb{N}_{congr}$, if $\bar q=s^2 n$, where $s\in\mathbb{Q}$.
\end{proof}
\begin{lemma}[Rational points identified by a congruent right triangle]\label{rational-points-identified-by-a-congruent-right-triangle}
If $[q|a,b,c]$ is a congruent right triangle and $P=(x,y)$ is the corresponding rational point on the elliptic curve $E[q]:\, y^2=x^3-q^2x$, identified by {\em(\ref{relation-congruent-right-triangle-point})}, then we can identify also six other rational points on $E[q]$ by means of Tab. \ref{table-rational-points-identified-by-a-congruent-right-triangle}. These points are  intersection with $E[q]$ of the three straight-lines passing for $P$ and $(-q,0)$, $(0,0)$, $(q,0)$ respectively and their reflections with respect to the $x$-axis.
\end{lemma}
\begin{proof}
The proof follows directly from the fact that if $(a,b,c)$ is a solution of the set of equations (\ref{real-elliptic-system}) then we get also a set of solutions of these equations by changing sign to the parameters $a,\, b,\, c$, with the condition that $a$ and $b$ have the same sign. (See Tab. \ref{table-rational-points-identified-by-a-congruent-right-triangle}.)
\end{proof}

\begin{table}[t]
\caption{Rational points on elliptic curve $E[q]:\, y^2=x^3-q^2x$, related by symmetries.}
\label{table-rational-points-identified-by-a-congruent-right-triangle}
\begin{tabular}{|l|l|l|l|}
\hline
{\rm{\footnotesize $(a,b,c)$}}\hfill&{\rm{\footnotesize $P=(x,y)=(q\frac{a+c}{b},2q^2\frac{a+c}{b^2})$}}\hfill&{\rm{\footnotesize $(\star):\, (\frac{25}{4},\frac{75}{8})$}}\hfill
&{\rm{\footnotesize $(\star\star):\, (12,36)$}}\hfill\\
\hline
{\rm{\footnotesize $(-a,-b,-c)$}}\hfill&{\rm{\footnotesize $P_1=(x_1,y_1)=(q\frac{a+c}{b},-2n^2\frac{a+c}{b^2})$}}\hfill&{\rm{\footnotesize $(\star):\, (\frac{25}{4},-\frac{75}{8})$}}\hfill&{\rm{\footnotesize $(\star\star):\, (12,-36)$}}\hfill\\
\hline
\hfil{\rm{\footnotesize $(a,b,-c)$}}\hfill&{\rm{\footnotesize $P_2=(x_2,y_2)=(q\frac{a-c}{b},2n^2\frac{a-c}{b^2})$}}\hfill&{\rm{\footnotesize $(\star):\, (-4,-6)$}}\hfill&{\rm{\footnotesize $(\star\star):\, (-3,-9)$}}\hfill\\
\hline
{\rm{\footnotesize $(-a,-b,c)$}}\hfill&{\rm{\footnotesize $P_3=(x_3,y_3)=(-q\frac{c-a}{b},2n^2\frac{c-a}{b^2})$}}\hfill&{\rm{\footnotesize $(\star):\, (-4,6)$}}\hfill&{\rm{\footnotesize $(\star\star):\, (-3,9)$}}\hfill\\
\hline
{\rm{\footnotesize $(b,a,-c)$}}\hfill&{\rm{\footnotesize $P_4=(x_4,y_4)=(n\frac{b-c}{a},2q^2\frac{b-c}{a^2})$}}\hfill&{\rm{\footnotesize $(\star):\, (-\frac{5}{9},-\frac{100}{27})$}}\hfill&{\rm{\footnotesize $(\star\star):\, (-2,-8)$}}\hfill\\
\hline
{\rm{\footnotesize $(-b,-a,c)$}}\hfill&{\rm{\footnotesize $P_5=(x_5,y_5)=(q\frac{c-b}{-a},2n^2\frac{c-b}{a^2})$}}\hfill&{\rm{\footnotesize $(\star):\, (-\frac{5}{9},\frac{100}{27})$}}\hfill&{\rm{\footnotesize $(\star\star):\, (-2,8)$}}\hfill\\
\hline
{\rm{\footnotesize $(b,a,c)$}}\hfill&{\rm{\footnotesize $P_6=(x_6,y_6)=(n\frac{b+c}{a},2q^2\frac{b+c}{a^2})$}}\hfill&{\rm{\footnotesize $(\star):\, (15,50)$}}\hfill&{\rm{\footnotesize $(\star\star):\, (18,72)$}}\hfill\\
\hline
{\rm{\footnotesize $(-b,-a,-c)$}}\hfill&{\rm{\footnotesize $P_7=(x_7,y_7)=(n\frac{-b-c}{-a},2q^2\frac{-b-c}{a^2})$}}\hfill&{\rm{\footnotesize $(\star):\, (15,-50)$}}\hfill&{\rm{\footnotesize $(\star\star):\, (18,-72)$}}\hfill\\
\hline
\multicolumn{4}{l}{\rm{\footnotesize $(\star)$: Here is reported the corresponding example $q=5$. (See also Tab \ref{congruent-numbers-pythagorean} and Tab. \ref{table-examples-of-strong-congruent-numbers}.)}}\hfill\\
\multicolumn{4}{l}{\rm{\footnotesize $(\star\star)$: Here is reported the corresponding example $q=6$. (See also Tab \ref{congruent-numbers-pythagorean} and Tab. \ref{table-examples-of-strong-congruent-numbers}.)}}\hfill\\
\end{tabular}
\end{table}

\begin{proposition}[Mordell-Weil theorem \cite{MORDELL, WEIL1, WEIL2, SILVERMAN}]\label{abelian-groups-associated-to-congruent-numbers}
$\bullet$\hskip 2pt We can associate to each congruent number, or elliptic curve representing a congruent number, a finitely generated abelian group, $\mathbf{E}[q]$. The structure of $\mathbf{E}[q]$ is given in {\em(\ref{structure-structure-group-congruent-numbers})}.
\begin{equation}\label{structure-structure-group-congruent-numbers}
  \mathbf{E}[q]=\mathbb{Z}^r\bigoplus\mathbf{T}[q]=\underbrace{\mathbb{Z}\oplus\cdots\oplus\mathbb{Z}}_{r}\bigoplus \left(\oplus_{1\le i\le s}\mathbb{Z}_{r_i}\right),\, r,\, s<\infty.
\end{equation}
The {\em rank} of the elliptic curve $E[q]$ is the rank of $\mathbf{E}[q]$, i.e., the number $r$ representing the number of independent points of infinite order.\footnote{This means that there exists a finite sub-set of the rational points of $E[q]$, such that all the other rational points can be generated by the abelian group law.} Furthermore $\mathbf{T}[q]=\oplus_{1\le i\le r}\mathbb{Z}_{r_i}\subset\mathbf{E}[q]$ is called the {\em torsion subgroup} of $\mathbf{E}[q]$.

$\bullet$\hskip 2pt If $E[q]$ and $E[\bar q]$ are two equivalent congruent elliptic curves, ($\bar q=s^2q$, $s\in\mathbb{Q}$), then there is a canonical isomorphism $\varphi_*:\mathbf{E}[q]\cong \mathbf{E}[\bar q]$. In other words the elliptic curve $E[q]$ and $E[q']$ are {\em isogenous}. (See also Theorem \ref{modularity-theorem}.)\footnote{This theorem can be also generalized to the case where instead of $\mathbb{Q}$ one works with a number field (or algebraic number field) $K$, that is a finite field extension $K/\mathbb{Q}$ of $\mathbb{Q}$, namely $K$ is a finite dimensional $\mathbb{Q}$-vector space. (See \cite{MEREL}.) Recall that any field extension $L/K$ is called {\em algebraic} if any element $x\in L$, is {\em algebraic} over $K$, i.e., there exists $P\in K[x]$, such that $P(x)=0\in L$. A {\em transcendental} extension is an extension that is not algebraic. Transcendental extensions are of infinite degree. (The {\em degree}, $[L:K]$, of an extension $L/K$, is the dimension of the $K$-vector space $L$.) Therefore all finite extensions are algebraic. The converse is not true. An example of transcendental extension is $\mathbb{R}/\mathbb{Q}$, since $\dim_{\mathbb{Q}}\mathbb{R}=\infty$. In fact the {\em Napier's constant} (or {\em Euler's number}) $e=\mathop{\lim}\limits_{n\to\infty}(1+\frac{1}{n})^n=\sum_{0\le n\le\infty}\frac{1}{n!}$, is an irrational number and cannot be a root of some polynomial $P(x)\in\mathbb{Q}[x]$, $x\in\mathbb{R}$. Instead $\mathbb{C}/\mathbb{R}$ is an example of algebraic extension, one has $[\mathbb{C}:\mathbb{R}]=2$, with canonical basis $\{1,i\}\subset\mathbb{C}$.}
\end{proposition}

\begin{proof}
$\bullet$\hskip 2pt On an elliptic curve $E[q]$ can be defined a multiplication, respect to which it becomes an abelian group $G[E]$. On $\mathbb{R}^2$, the more general expression of an elliptic curve is of the form $y^2=x^3+ax+b$, hence a non-singular plane curve. One considers compactified such a curve by adding the $\infty$ point in the Alexandrov compactification of $\mathbb{R}^2$: $S^2=\mathbb{R}^2\bigcup\{\infty\}$. The point $\{\infty\}$ is the identity element in the natural group structure defined on $E[q]$. The set of rational points, including $\infty$, form a subgroup $\mathbf{E}[q]$ of $G[q]$.\footnote{Let $P=(x_P,y_P)$, $Q=(x_Q,y_Q)$ be two points on the elliptic curve $E$, with $x_P\not= x_Q$, then $P+Q+R=0$, identifies another point $R$ on $E$. The addition defined in this way is called the addition by means of secant (or tangent). More precisely, in the case that the line passing for $P$ and $Q$ is tangent to $E$ in $Q$, then one writes $P+Q+Q=0$, or $P+P+Q=0$ (only tangent in $P$), or $P+Q+0=0$ (only secant in $P$ and $Q$).} It can be seen that the torsion points of an elliptic curve are those with $y=0$. (See in Example \ref{examples-of-congruent-numbers}.) Therefore existence of rational points with $y\not=0$ is equivalent to say that the elliptic curve has rank positive. The structure (\ref{structure-structure-group-congruent-numbers}) follows from the fact that $\mathbf{E}[n]$ is a finite generated abelian group \cite{MORDELL, WEIL1, WEIL2}.

$\bullet$\hskip 2pt The isomorphism between structure groups $\mathbf{E}[q]$ and $\mathbf{E}[\bar q]$ is the one induced by the diffeomorphism $\varphi:\mathbb{R}^2\to\mathbb{R}^2$ relating the points of elliptic curves $E[q]$ and $E[\bar q]$, (see Lemma \ref{lemma-congruence-equivalent-elliptic-curve}).
\end{proof}
\begin{proposition}[Mazur's theorem \cite{MAZUR}]\label{mazur-theorem-elliptic-curves}
The torsion groups of elliptic curves can be only of the following types $\mathbb{Z}/N\mathbb{Z}$, $N\in\{1,2,\cdots, 10,12\}$, or $\mathbb{Z}/2\mathbb{Z}\times \mathbb{Z}/2N\mathbb{Z}$, $N\in\{1,2,3,4\}$.\footnote{An {\em elliptic curve} over $\mathbb{R}$ is a line in $\mathbb{R}^2$, $(x,y)$, defined by an equation $y^2-P(x)=0$, where $P(x)\in\mathbb{R}[x]$ is a cubic polynomial with distinct roots. By a suitable diffeomorphism of $\mathbb{R}^2$, equation $y^2-P(x)=0$ can be rewritten in the {\em Weierstrass form}: $E=E_{A,B}: y^2-4x^3+Ax+B=0$, with $A,\, B\in\mathbb{R}$, and $\Delta_{E}:=A^3-27 B^2\not=0$. [The polynomial $P(x)=4x^3-Ax-B$ has distinct roots iff its discriminant ${\rm disc}(P(x))=16\Delta_E\not=0$.] By diffeomorphisms of $\mathbb{R}^2$, $(x,y)\mapsto(s^2 x, s^3 y)$, $s\in\mathbb{R}^{\times}$, equation defining $E$ transforms into the following one: $E_{A_1,B_1}: y^2-4x^3+A_1x+B_1=0$, with $A_1=A/s^4$, $B_1=B/s^6$. One has $\Delta_{A_1,B_1}=\Delta_{A,B}
s^{-12}$. One has instead the invariant $j_E=\frac{(12 A)^3}{\Delta_E}=\frac{(12 A)^3}{A^3-27B^2}=j_{E_{A_1,B_1}}=\frac{(12 A_1)^3}{\Delta_{E_{A_1,B_1}}}=\frac{(12 A_1)^3}{A_1^3-27B_1^2}$. The {\em classical Weierstrass form} of an elliptic curve is $E_{a,b}: y^2-x^3-ax-b=0$, with $4a^3+27b^2\not=0$, is obtained by the diffeomorphism $(x,y)\mapsto(x, 2y)$, with $a=-A/4$, $b=-B/4$. Then $\Delta E_{a,b}=-16(4a^3+27 b^2)=A^3-27 B^2=\Delta_E$ and $j_{E_{a,b}}=-\frac{(12a)^3}{\Delta_{E_{a,b}}}=j_{E_{A,B}}$. Therefore one can state that elliptic curves in the planes $\mathbb{R}^2$, are in general of the type $y^2=x^3+ax+b$, with the condition $\triangle=-16(4a^3+27 b^2)\not=0$, in order to not be singular curves. ($\triangle$  is defined the {\em discriminant} of the elliptic curve.) An invariant for isomorphism classes is the {\em Klein's $j$-invariant}: $j=(-48\, a)^3/\triangle$. In Tab. \ref{klein-invariant-elliptic-complex-curves} is reported the Klein's $j$-invariant for plane elliptic curves over any field. Let us emphasize that if the characteristic of the fundamental field $K$ is neither $2$ nor $3$, then every elliptic curve over $K$ can be written in the form $y^2=x^3-px-q$, where $p$ and $q$ are elements of $K$ such that the polynomial $P(x)=x^3-px-q$ does not have any double roots. In the particular case of $y^2=x^3-q^2 x$ then its Klein's $j$-invariant is $j=1728$. So all elliptic curves of this type are diffeomorphic with the same Klein's $j$-invariant $j=1728$. If the characteristic of $K$ is $2$ or $3$, then the general form of elliptic curve is more complex. In characteristic $3$, it assumes the expression $y^2=4x^3+b_2x^2+2b_4x+b_6$, such that the polynomial $P(x)= 4x^3+b_2x^2+2b_4x+b_6$ has distinct roots. In characteristic $2$, the most general equation is $y^2+a_1xy+a_3y=x^3+a_2x^2+a_4x+a_6$, provided that the variety it defines is non-singular. (If characteristic were not an obstruction, each equation would reduce to the previous ones by a suitable change of variables.) In general one takes $x,\, y$ belonging to the {\em algebraic closure} of $K$, i.e., an algebraic extension $F/K$ that is {\em algebraically closed}, i.e., contains a root for every non-constant polynomial in $F[x]$.}
\end{proposition}

\begin{proposition}[Nagell-Lutz theorem]\label{nagell-lutz-theorem}
$\bullet$\hskip 2pt Let $y^2=x^3+ax^2+bx+c$ defines a non-singular cubic curve $C$ with integer coefficients $a$, $b$, $c$, and let $\Delta=-4a^3c+a^2b^2+18 abc-4b^3$ be the discriminant of the cubic polynomial on the right side. If $P=(x,y)$ is a rational point of finite order on $C$, for the elliptic curve group law, then:

{\rm(i)} $x$ and $y$ are integers;

{\rm(ii)} either $y=0$, in which case $P$ has order two, or else $y$ divide $\Delta$, which implies that $y^2$ divides $\Delta$.

$\bullet$\hskip 2pt {\rm(Generalized form)} For non-singular cubic curve whose Weierstrass form $y^2+a_1xy+a_3y=x^3+a_2x^2+a_4x+a_6$, has integer coefficients, any rational point $P=(x,y)$ of finite order, (namely torsion-point), must have integer coordinates, or else have order $2$ and coordinates of the form $(x=\frac{m}{4},y=\frac{n}{8})$ for $m$ and $n$ integers.
\end{proposition}
\begin{proof}
See \cite{LUTZ,SILVERMAN,SILVERMAN-TATE}.
\end{proof}

\begin{cor}\label{infinitely-many-rational-points}
A number $q\in\mathbb{Q}$ is congruent iff there exist infinitely many rational points $P=(x_P,y_P)$ on the elliptic curve $E[q]$.
\end{cor}

\begin{table}[t]
\caption{The Klein's $j$-invariant for plane elliptic curves over any field.}
\label{klein-invariant-elliptic-complex-curves}
\begin{tabular}{|l|l|}
\hline
\multicolumn{2}{|c|}{\rm{\footnotesize $y^2+a_1xy+a_3y=x^3+a_2x^2+a_4x+a_6$}}\\
\hline\hline
{\rm{\footnotesize $b_2=a_1^2+4a_2$}}&{\rm{\footnotesize $b_4=a_1a_3+2a_4$}}\hfill\\
\hline
{\rm{\footnotesize $b_6=a_3^2+4a_6$}}&{\rm{\footnotesize $b_8=a_1^2a_6-a_1a_3a_4+a_2a_3^2+4a_2a_6-a_4^2$}}\hfill\\
\hline
{\rm{\footnotesize $c_4=b_2^2-24b_4$}}&{\rm{\footnotesize $c_6=-b_2^3+36b_2b_4-216b_6$}}\hfill\\
\hline
{\rm{\footnotesize $\triangle=-b_2^2b_8+9b_2b_4b_6-8b_4^3-27b_6^2$}}&\hfil{\rm{\footnotesize $j=c_4^3/\triangle$}}\hfil\\
\hline
\multicolumn{2}{l}{\rm{\footnotesize If the field has characteristic different from $2$ or $3$, one has $j=1728 \frac{c_4^3}{c_4^3-c_6^2}$.}}\hfill\\
\multicolumn{2}{l}{\rm{\footnotesize See Tab. \ref{c-table-example-modular-functions-forms-cusps}, in Appendix C,  for more information on $j$ and $\Delta$ in the case of elliptic curves over $\mathbb{C}$.}}\hfill\\

\end{tabular}
\end{table}

\begin{example}\label{examples-of-congruent-numbers}
In the following we list some examples of congruent numbers identifying congruence equivalence classes.

$\bullet$\hskip 2pt {\rm(Fermat's theorem) [1640]}The number $1$ is not (strong-)congruent. More generally, no square number can be a congruent number. (See, e.g., \cite{COPPEL}.)

$\bullet$\hskip 2pt $n\equiv 3\, \hbox{\rm mod}\, 8$ is not a congruent number, but $2n$ is a congruent number.

$\bullet$\hskip 2pt $n\equiv 5\, \hbox{\rm mod}\, 8$ is a congruent number.

$\bullet$\hskip 2pt $n\equiv 7\, \hbox{\rm mod}\, 8$ is a congruent number and $2n$ is so.

$\bullet$\hskip 2pt In each of the congruence classes $5,\ 6,\ 7\, \hbox{\rm mod}\, 8$, there are infinitely many square-free congruent numbers with $k$ prime factors.

In Tab. \ref{table-examples-of-strong-congruent-numbers} are reported some examples of strong-congruent numbers and the sides of the corresponding congruent right triangles.

$\bullet$\hskip 2pt  If the elliptic curve $E[n]:\, y^2-x^3+n^2x=0$, has $n\in\mathbb{N}_{congr}$, then its group $\mathbf{E}[n]$ of rational points, has order greater than $4$: $|\mathbf{E}[n]|>4$. In fact its torsion points are the following:
\begin{equation}\label{torsion-points-of-e-n}
\{\infty,(x=-n,y=0), (x=0,y=0),(x=n,y=0),\}.
\end{equation}
This can be also obtained from the Nagell-Lutz theorem In fact, the discriminant of the cubic polynomial $P(x)=x^3-n^2x$ is $\Delta=-4n^6$ and by the Nagell-Lutz theorem one has $y^2|-4n^6$, so we take all possible solutions $y^2|(2n^3=2p_1^{a_1}p_2^{a_2}\cdots p_k^{a_k})$, where $2n^3=2p_1^{a_1}p_2^{a_2}\cdots p_k^{a_k}$ is the prime factorization. Therefore, we get $y\in\{\pm 2^tp_1^{r_1}p_2^{r_2}\cdots p_k^{r_k}\, |\, t=0,1,\, 0\le r_i\le a_i\}$. Since none of these values satisfies equation $y^2-x^3+n^2x=0$, it follows that all torsion points of $E[n]$ are the ones reported in {\rm(\ref{torsion-points-of-e-n})}, and $|\mathbf{T}[n]|=4$.
\end{example}
\begin{table}[t]
\caption{Examples of strong-congruent numbers $n\in \mathbb{N}_{congr}$ for square-free numbers $n\in{}_{\square}\mathbb{N}$, $1\le n\le 65$, some corresponding right triangles and non-congruent numbers $n\not\in {\rm ker}(L_\bullet)$.}
\scalebox{0.8}{$\label{table-examples-of-strong-congruent-numbers}
\begin{tabular}{|c|c|c|}
\hline
\hfil{\rm{\footnotesize $n\in {}_{\square}\mathbb{N}$}}\hfil&\hfil{\rm{\footnotesize $n\in \mathbb{N}_{congr}$}}\hfil&\hfil{\rm{\footnotesize ${[n|a,b,c]}$}}\hfil\\
\hline\hline
\hfil{\rm{\footnotesize $(\star)\, (1),\, (2),\, (3),\, [5]$}}\hfil&\hfil{\rm{\footnotesize $5$}}\hfil&\hfil{\rm{\footnotesize ${[5|\frac{3}{2},\frac{20}{3},\frac{41}{6}]}$}}\hfil\\
\hline
\hfil{\rm{\footnotesize $[6]$}}\hfil&\hfil{\rm{\footnotesize $6$}}\hfil&\hfil{\rm{\footnotesize ${[6|3,4,5]}$}}\hfil\\
\hline\hfil{\rm{\footnotesize $[7]$}}\hfil&\hfil{\rm{\footnotesize $7$}}\hfil&\hfil{\rm{\footnotesize ${[7|\frac{24}{5},\frac{35}{12},\frac{337}{60}]}$}}\hfil\\
\hline
\hfil{\rm{\footnotesize $(10),\, (11),\, [13],\, [14],\, [15],\, 17,\, (19),\, [21] $}}\hfil&\hfil{\rm{\footnotesize $21$}}\hfil&\hfil{\rm{\footnotesize ${[21|\frac{7}{2},12,\frac{25}{2}]}$}}\hfil\\
\hline
\hfil{\rm{\footnotesize $[22],\, [23],\, 26,\, 29,\, [30]$}}\hfil&\hfil{\rm{\footnotesize $30$}}\hfil&\hfil{\rm{\footnotesize ${[30|5,12,13]}$}}\hfil\\
\hline
\hfil{\rm{\footnotesize $31,\, (33),\, [34],\, 35,\, [37],\, [38],\, [39],\, [41]$}}\hfil&\hfil{\rm{\footnotesize $41$}}\hfil&\hfil{\rm{\footnotesize ${[41|\frac{40}{3},\frac{123}{20},\frac{881}{6}]}$}}\hfil\\
\hline
\hfil{\rm{\footnotesize $42,\, 43,\, [46],\, [47],\, 49,\, (51),\, 53,\, 55,\, (57),\, (58),\, (59),\, 61,\, 62,\, [65]$}}\hfil&\hfil{\rm{\footnotesize $65$}}\hfil&\hfil{\rm{\footnotesize ${[65|\frac{65}{6},12,\frac{97}{6}]}$}}\hfil\\
\hline
\multicolumn{3}{l}{\rm{\footnotesize The square-free integers between square-brackets in the column of ${}_{\square}\mathbb{N}$, denote strong-congruent numbers.}}\hfill\\
\multicolumn{3}{l}{\rm{\footnotesize The square-free integers between round-brackets in the column of ${}_{\square}\mathbb{N}$, denote non-congruent numbers,}}\hfill\\
\multicolumn{3}{l}{\rm{\footnotesize namely $n\not\in {\rm ker}( L_\bullet)\subset {}_{\square}\mathbb{N}$. (See Lemma \ref{elleiptic-congruent-bordism-groups-and-birch-swinnerton-dyer-conjecture}.)}}\hfill\\
\multicolumn{3}{l}{\rm{\footnotesize $(\star)$ Since $1$ is not a strong-congruent number, it follows that also $m^2$ is so, for any $m\in\mathbb{N}$.}}\hfill\\
\multicolumn{3}{l}{\rm{\footnotesize The proof that $5$ and $7$ are strong-congruent numbers was first given by Fibonacci \cite{LEONARDO-PISANO}.}}\hfill\\
\multicolumn{3}{l}{\rm{\footnotesize He stated also that $1$ is not a congruent number, but a first proof has been given by Fermat (See, e.g., in \cite{DICKSON}.)}}\hfill\\
\end{tabular}$}
\end{table}
\begin{theorem}[Criterion for congruent number $q\in\mathbb{Q}$]\label{criterion-congruent-number}
A number $q\in\mathbb{Q}$ is congruent if ${\rm rank}(\mathbf{E}(q))>0$.
\end{theorem}
\begin{proof}
In fact, if ${\rm rank}(\mathbf{E}(q))>0$ it means that there are rational points in the elliptic curve $y^2=x^3-q^2 x$, hence from Lemma \ref{lemma-congruent-numbers-as-elliptic} and Lemma \ref{lemma-congruence-equivalent-elliptic-curve} it follows that $q$ is representing a congruent number class.
\end{proof}

\begin{definition}[The congruent number problem]\label{definition-congruent-number-problem}
Given a positive square-free integer $n\in{}_{\square}\mathbb{N}$, there is a simple criterion to decide whether $n\in \mathbb{N}_{congr}$, (i.e., $n$ is a strong-congruent number) ?
\end{definition}
\begin{figure}[t]
 \includegraphics[height=6cm]{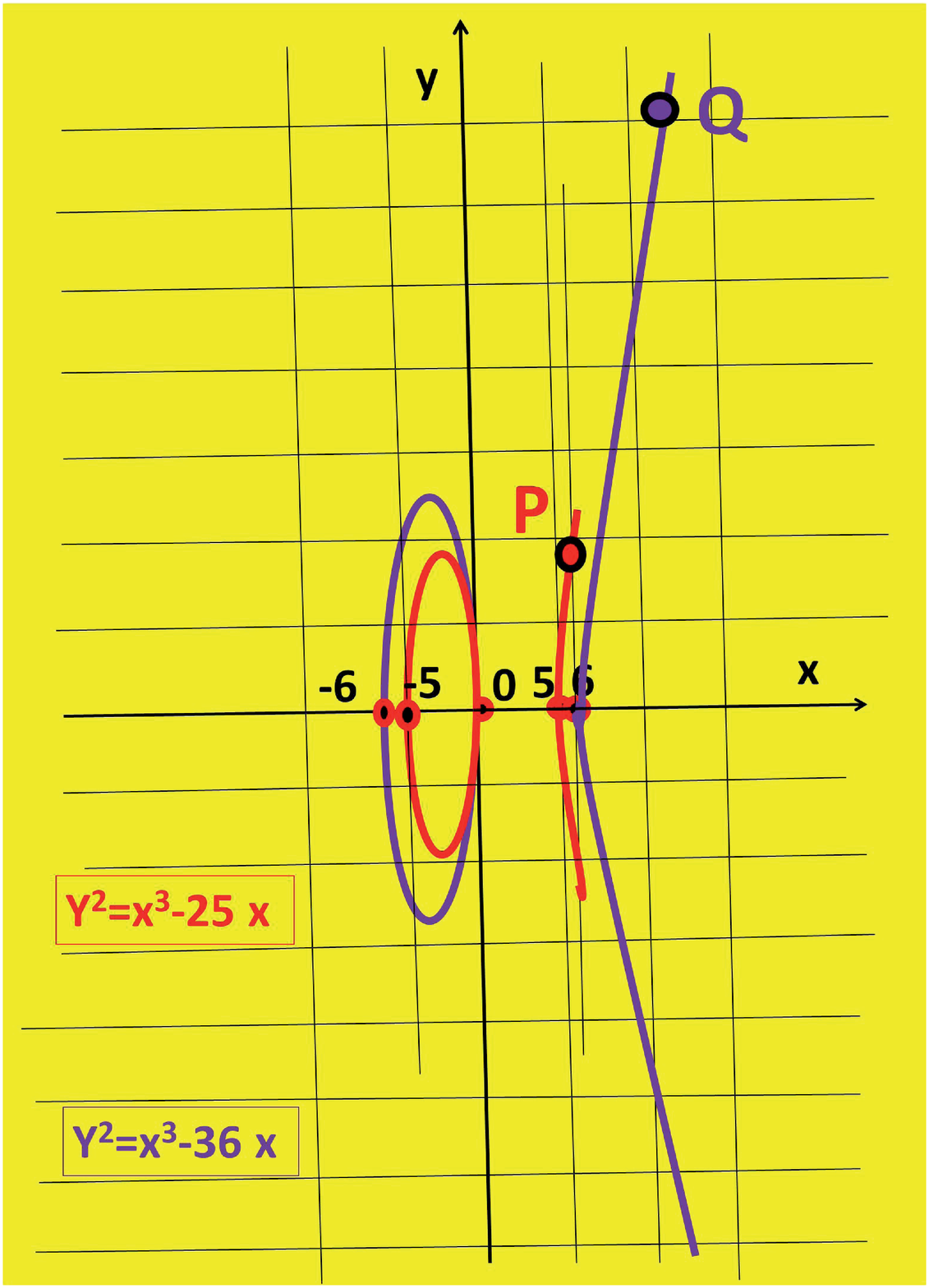}
\renewcommand{\figurename}{Fig.}
\caption{Representation of the elliptic curve (red-curve) $E[5]:\, y^2=x^3-25\, x$. $E[5]=E[5]_1\sqcup E[5]_2$ where $E[5]_1$ is the closed curve on the left-hand of the $y$-axis and and $E[5]_2$ is the other part on the right-hand of the $y$-axis. On $E[5]_1$ is reported the rational point $P=(\frac{25}{4},\frac{75}{8})$ corresponding to the congruent right triangle $(\frac{3}{2},\frac{20}3{},\frac{41}{6})$. The violet-curve represents the elliptic curve $E[6]:\, y^2=x^3-36\, x$. One has $E[5]\bigcap E[6]_1=O=(0,0)\in\mathbb{R}^2$. On $E[6]_1$ is reported the rational point $Q=(12,36)$ corresponding to the congruent right triangle $(3,4,5)$. For symmetry properties one can identify also six further rational points on $E[5]$ and $E[6]$ respectively. Some of these are also on the left-hand connected component. (For details see Tab. \ref{representation-elliptic-curve}.)}
\label{representation-elliptic-curve}
\end{figure}

\section{\bf The Birch Swinnerton-Dyer conjecture}\label{section-bs-d-conjecture}
\vskip 0.5cm
In this section we shall consider some important well-known results about congruent number problem and the related Birch Swinnerton-Dyer conjecture,(BS-D conjecture) that will be used in the next section. (For complementary information see also literature on this subject in References, and in particular look e.g., to the book by Koblitz \cite{KOBLITZ}.)

Let start with the following theorem.

\begin{theorem}[Modularity theorem or Taniyama-Shimura-Weil conjecture]\label{modularity-theorem}
Any elliptic curve $E$ over $\mathbb{Q}$ is a {\em modular curve}, i.e., there exists a surjective morphism $\varphi:X_0(N)\to E$, where $\varphi$ is a rational map with integer coefficients, and $X_0(N)$ is the {\em classical modular curve}, for some integer $N$.\footnote{$X_0(N)$ is a compact Riemann surface, defined by $X_0(N)=H^*/\Gamma_0(N)$, where $H^*=H\bigcup\mathbb{Q}\bigcup\{\infty\}$ is the extended complex upper-half plane $H\subset\mathbb{C}$, and $\Gamma_0(N)$ is the {\em congruence subgroup of level} $N$, of the {\em modular group} $SL(2;\mathbb{Z})$, defined by $\Gamma_0(N)=\scalebox{0.8}{$\left\{\left(
                      \begin{array}{cc}
                        a & b \\
                        c& d \\
                      \end{array}
                    \right)\, :\, c\equiv\, 0\,  \hbox{\rm mod}\, N\right\}$}$, for some positive $N\in\mathbb{N}$. Two elliptic curves $E$ and $E'$ are {\em isogenous} if there is a morphism of varieties, defined by a rational map between $E$ and $E'$, which is also a group homomorphism between the corresponding groups $\mathbf{E}$ and $\mathbf{E}'$ sending $\infty\in E$ to $\infty\in E'$, i.e., conserving identity elements. The isogenies with cyclic kernel of degree $N$, {\em (cyclic isogenies)}, correspond to points on $X_0(N)$: $\xymatrix{0\ar[r]&\mathbb{Z}_N\ar[r]&\mathbf{E}\ar[r]^{f}&\mathbf{E}'}$, $f\mapsto p\in X_0(N)$. When $X_0(N)$ has genus $1$, then $X_0(N)\cong E$, which will have the same $j$-invariant. For example $X_0(11)$ has $j=-2^{12}{11}^{-5}31^{3}$, and is isomorphic to the curve $y^2+y=x^3-x^2-10 x-20$. (For relations between Riemann surfaces and modular curves see also Appendix B.) This means that this elliptic curve can be parametrized by means of two functions $(x=x(z),y=y(z))$, where $x(z)$ and $y(z)$ are modular functions of weight $0$ and level $11$: in other words they are meromorphic, defined on the upper half-plane $\Im(z)>0$ and satisfy $(x(z)=x(\frac{az+b}{cz+d}),y(z)=y(\frac{az+b}{cz+d}))$, for all integers $a,\, b,\, c,\, d$ with $ad-bc=1$ and $11|c$.} This mapping is called a {\em modular parametrization of level $N$}. The {\em conductor} of $E$ is the smallest integer $N$ for which such a parametrization can be found.
\end{theorem}
\begin{proof}
This theorem has been proved in a particular case by A. Wiles \cite{WILES1}. For the general case a proof has been given in \cite{BREUIL-CONRAD-DIAMOND-TAYLOR}. For related subjects see also \cite{DARMON, KOBLITZ, IWANIEC, SERRE, SCHOENEBERG, SHIMURA, TANIYAMA, WEIL1, WEIL2, WILES1, TATE}.
\end{proof}

\begin{definition}[The Hasse-Weil-function of elliptic curve and Hasse-Weil conjecture]\label{hasse-weil-function-elliptic-curve}
$\bullet$\hskip 2pt Given an elliptic curve $E$ over $\mathbb{Q}$ of conductor N, then $E$ has good reduction at all primes $p$ not dividing $N$, it has multiplicative reduction at the primes $p$ that exactly divide $N$ (i.e. such that $p$ divides $N$, but $p^2$ does not; this is written $p\, || N$), and it has additive reduction elsewhere (i.e. at the primes where $p^2$ divides $N$).
One defines {\em Hasse-Weil-function} of $E$, the function $Z_{E,\mathbb{Q}}(s)$, given in {\em(\ref{hasse-weil-function})}.\footnote{Taking into account the functional equation for $\zeta(s)$, $\zeta(s)=f(s)\zeta(1-s)$, where $f(s)=2^s\pi^{s-1}\sin(\frac{\pi s}{2})\Gamma(1-s)$, one can rewrite $Z_{E,\mathbb{Q}}(s)$ in the form
$$Z_{E,\mathbb{Q}}(s)=f(s)\zeta^2(1-s)\prod_{p\in P, p\nmid N,a_p=p+1}(1-a_pp^{-s}+p^{1-2s})\prod_{p\in P, p||  N,a_p=\pm 1}(1-a_pp^{-s}).$$}

\begin{equation}\label{hasse-weil-function}
   Z_{E,\mathbb{Q}}(s)=\frac{\zeta(s)\, \zeta(s-1)}{L(E,s)}, \, \left\{
   \begin{array}{ll}
     L(E,s)&=\prod_{p\in P} L_p(E,s)^{-1}\\
     &\\
     L_p(E,s)&=\scalebox{0.8}{$\left\{\begin{array}{ll}
     (1-a_pp^{-s}+p^{1-2s})& \hbox{\rm if $p\nmid N$} \\
     (1-a_pp^{-s})& \hbox{\rm if $p\, || N$} \\
     1& \hbox{\rm if $p^2| N$} \\
     \end{array}
     \right.$} \end{array}\right.
\end{equation}
In {\em(\ref{hasse-weil-function})} $P$ is the set of primes, $\zeta(s)$ is the usual Riemann zeta function and $L(E,s)$ is called the {\em$L$-function} of $E/\mathbb{Q}$. Furthermore, the coefficients $a_p$ are given in {\em(\ref{hasse-weil-function-a})}.
\begin{equation}\label{hasse-weil-function-a}
a_p=\left\{\begin{array}{ll}
             p+1& \left\{\begin{array}{l}\hbox{\rm in the case of good reduction:} \\
             \hbox{\rm $p+1=$ number of points of $E$ mod $p$.} \\
             \end{array}
              \right.\\
             \pm 1& \left\{\begin{array}{l}
                            \hbox{\rm in the case of multiplicative reduction:}\\
                            \hbox{\rm $\pm$ if $E$ has split or non-split multiplicative reduction at $p$.}\\
                           \end{array}
              \right.\\
           \end{array}\right.
\end{equation}

$\bullet$\hskip 2pt {\em The Hasse�Weil conjecture} states that the Hasse�Weil zeta function should extend to a meromorphic function for all complex s, and should satisfy a functional equation similar to that of the Riemann zeta function. For elliptic curves over the rational numbers, the Hasse-Weil conjecture follows from Theorem \ref{modularity-theorem}.
\end{definition}

\begin{conj}[The Birch and Swinnerton-Dyer conjecture \cite{BIRCH-SWINNERTON-DYER}]\label{definition-bs-d-conjecture}
The rank $k$ of the abelian group $\mathbf{E}[K]$ of the elliptic curve $E$ over a number field $K$, is the order of the zero of the Hasse-Weil-function $L(E,s)$ at $s=1$: $$L(E,s)^{(r)}|_{s=1}=0,\hskip 3pt r<k,\hskip 5pt L(E,s)^{(k)}|_{s=1}\not=0.$$

Furthermore, the non-zero coefficient of the Taylor expansion of $L(E,s)$ at $s=1$, is given by more refined arithmetic data attached to $E$ over $K$ (Wiles 2006).
\end{conj}

Whether Conjecture \ref{definition-bs-d-conjecture} is true one can solve the congruent number problem in Definition \ref{definition-congruent-number-problem}. In fact, one has the following proposition.\footnote{Previous results of the Tunnell's theorem, were ones by Stephens  proving that the Conjecture \ref{definition-bs-d-conjecture} implies any postive integer $n=5,\, 6,\, 7\, {\rm mod}\, 8$ is a congruent number \cite{STEPHENS}.}

\begin{theorem}[Tunnell's theorem \cite{TUNNELL}]\label{tunnell-theorem}
For a given square-free integer $n\in{}_{\square}\mathbb{N}$, define the associated integers given in {\em(\ref{associated-integers})}.
\begin{equation}\label{associated-integers}
\left\{
\begin{array}{l}
  A_n=\sharp\{(x,y,z)\in\mathbb{Z}^3\, |\, n=2x^2+y^2+32 z^2\}\\
 B_n=\sharp\{(x,y,z)\in\mathbb{Z}^3\, |\, n=2x^2+2y^2+8 z^2\}\\
C_n=\sharp\{(x,y,z)\in\mathbb{Z}^3\, |\, n=8x^2+2y^2+64 z^2\}\\
D_n=\sharp\{(x,y,z)\in\mathbb{Z}^3\, |\, n=8x^2+2y^2+16 z^2\}.\\
\end{array}
\right.
\end{equation}
We get the implications {\em(\ref{implications})}.
\begin{equation}\label{implications}
  \framebox{$\begin{array}{cccc}
    \hbox{\rm If $n\in\mathbb{N}_{congr}$,}&n=2m+1,&\Rightarrow&2A_n=B_n.\\
\hbox{\rm If $n\in\mathbb{N}_{congr}$,}& n=2m,&\Rightarrow&2C_n=D_n.\\
\end{array}$}
\end{equation}

$\bullet$\hskip 2pt Conversely if the Birch-Swinnerton-Dyer conjecture holds for elliptic curves of the form $y^2=x^3-n^2 x$, $n\in{}_{\square}\mathbb{N}$, then the framed equalities on the right in {\em(\ref{implications})}, are sufficient to conclude that $n\in\mathbb{N}_{congr}$.
\end{theorem}

\begin{example}
With respect to Tab. \ref{table-examples-of-strong-congruent-numbers} we can verify that Tunnell's theorem works well.\footnote{Let us underline that Tunnell's theorem works only for square-free integers $n\in{}_{\square}\mathbb{N}$.} For example we get the following:

$\bullet$\hskip 2pt $(n=1)$. $A_1=B_1=2$, hence $2A_1\not=B_1$. {\em[$1$ is not a congruent number.]}

$\bullet$\hskip 2pt $(n=3)$. $A_3=B_3=4$, hence $2A_3\not=B_3$. {\em [$3$ is not a congruent number.]}

$\bullet$\hskip 2pt $(n=5)$. $A_5=B_5=0$, hence $2A_5=B_5$. {\em[$5$ is a strong-congruent number.]}

$\bullet$\hskip 2pt $(n=10)$. $C_{10}=D_{10}=4$, hence $2C_{10}\not=D_{10}$. {\em[$10$ is not a congruent number.]}

$\bullet$\hskip 2pt $(n=13)$. $A_{13}=B_{13}=0$, hence $2A_{13}=B_{13}$. {\em[$13$ is a strong-congruent number.]}

$\bullet$\hskip 2pt $(n=65)$. $A_{65}=B_{65}=0$, hence $2A_{65}=B_{65}$. {\em[$65$ is a strong-congruent number.]}
\end{example}

\begin{remark}[The relation between the congruent number problem and the BS-D conjecture]\label{the-relation-emphasized}
A way to solve the congruent number problem is to solve the Conjecture \ref{definition-bs-d-conjecture}. Therefore the Tunnell's theorem emphasizes the importance of the BS-D conjecture.
\end{remark}

\begin{theorem}[Coates-Wiles theorem \cite{COATES-WILES}]\label{coates-wiles-theorem}
$\bullet$\hskip 2pt To each elliptic curve  $E[n]:\, y^2=x^3-n^2 x$, $n\in\mathbb{N}$, there is associated a number $L(E[n])$.

$\bullet$\hskip 2pt If $E[n]$ has infinitely many rational points, then $L(E[n])=0$.

$\bullet$\hskip 2pt If $L(E[n])$ is not zero, then $n$ cannot be a strong-congruent number.

$\bullet$\hskip 2pt{\em(Tunnell's expression for $L(E[n])$.)} For any $n\in {}_{\square}\mathbb{N}$ one can write
\begin{equation}\label{tunnell-expression}
    L(E[n])=\left\{\begin{array}{l}
                     C\cdot(A_n-\frac{B_n}{2}) \mbox{ if $n$ is odd}\\ [.05in]
                     C\cdot(C_n-\frac{D_n}{2}) \mbox{ if $n$ is even}\\
                   \end{array}
\right.
\end{equation}
where $C$ is a non-zero number, and $A_n$, $B_n$, $C_n$, $D_n$ are defined in Theorem \ref{tunnell-theorem}.

$\bullet$\hskip 2pt If the BS-D conjecture is true the condition $L(E[n])=0$ is also sufficient to state that $n$ is congruent, (i.e., $n\in\mathbb{N}_{cong}$, hence $E[n]$ has infinitely many rational points).

\end{theorem}

\section{\bf Elliptic and Congruent Bordism Groups}\label{section-elliptic-bordism-groups}
\vskip 0.5cm
In this section we shall relate the congruent numbers problem and the related BS-D conjecture to suitable bordism groups and to homotopies between elliptic curve inducing isomorphisms between such bordism groups. These algebraic topologic tools will give us the way to obtain, via the Tunnell's theorem, a workable criterion that in some finite steps allows us to know if a square-free integer is a strong-congruent number. Furthermore, taking into account the realtion between Tunnell's theorem and BS-D conjecture, we get as a by-product an indirect way to consider the BS-D conjecture true.

\begin{figure}[t]
(A) \includegraphics[height=3cm]{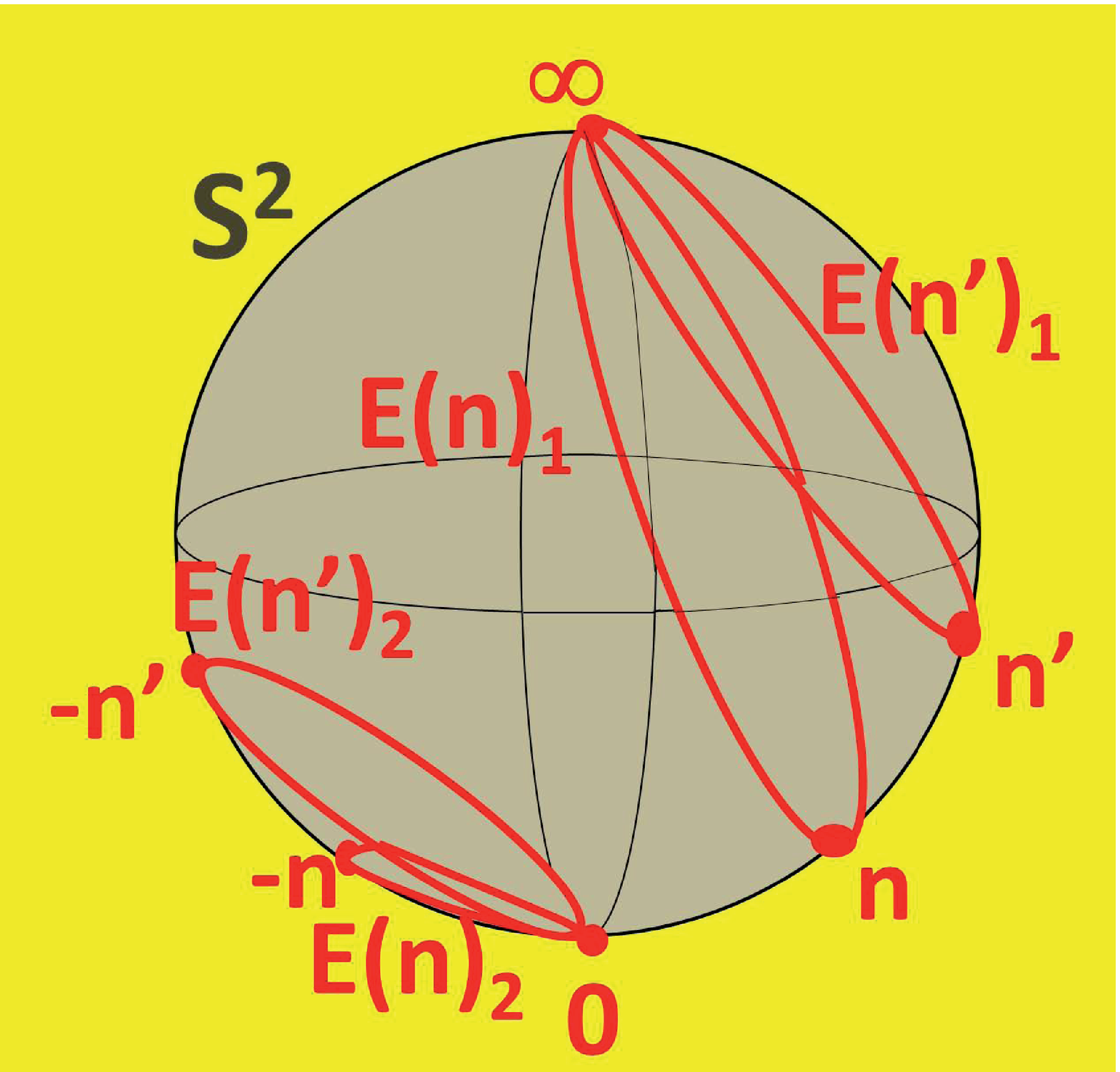} (B) \includegraphics[height=3cm]{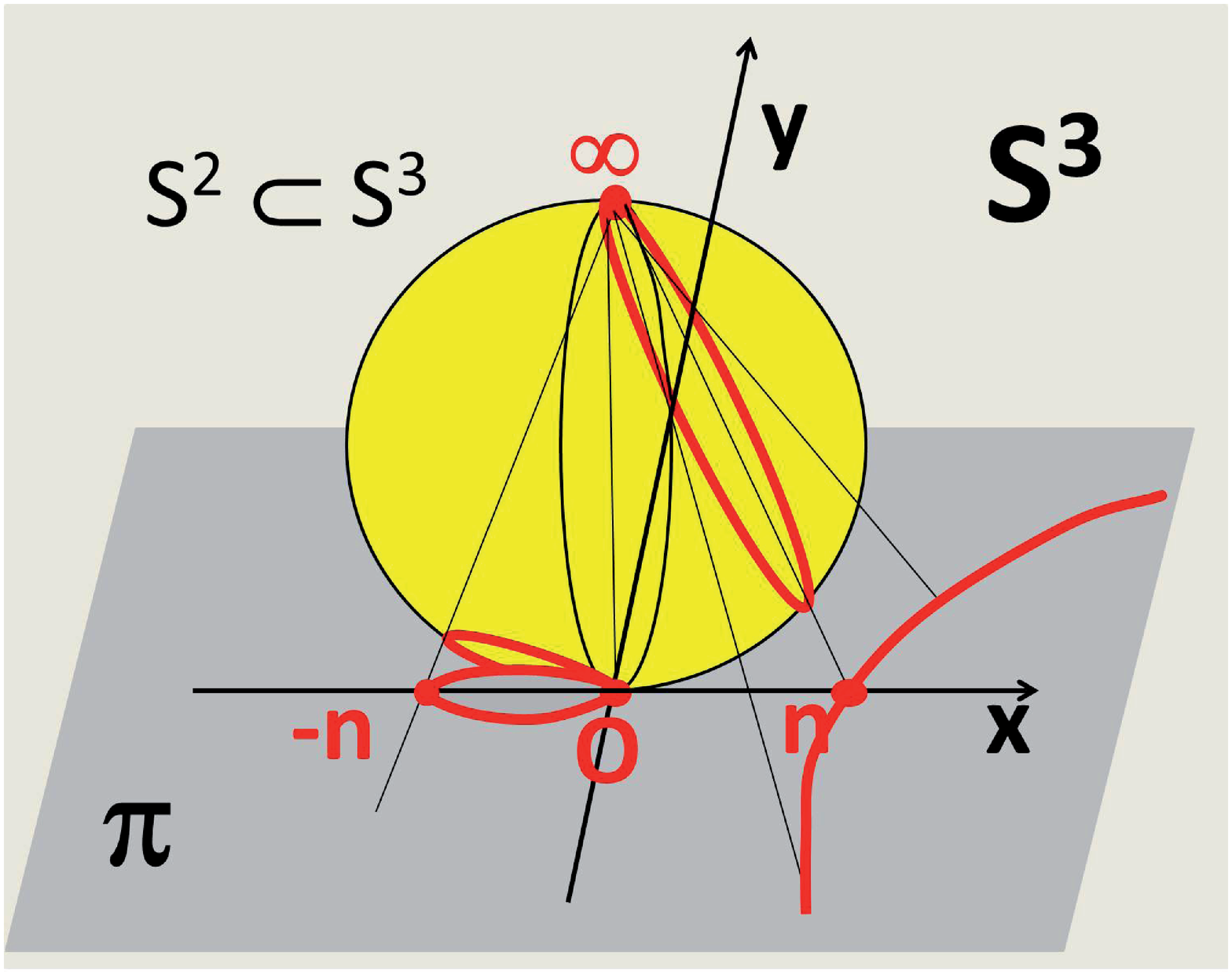} (C) \includegraphics[height=3cm]{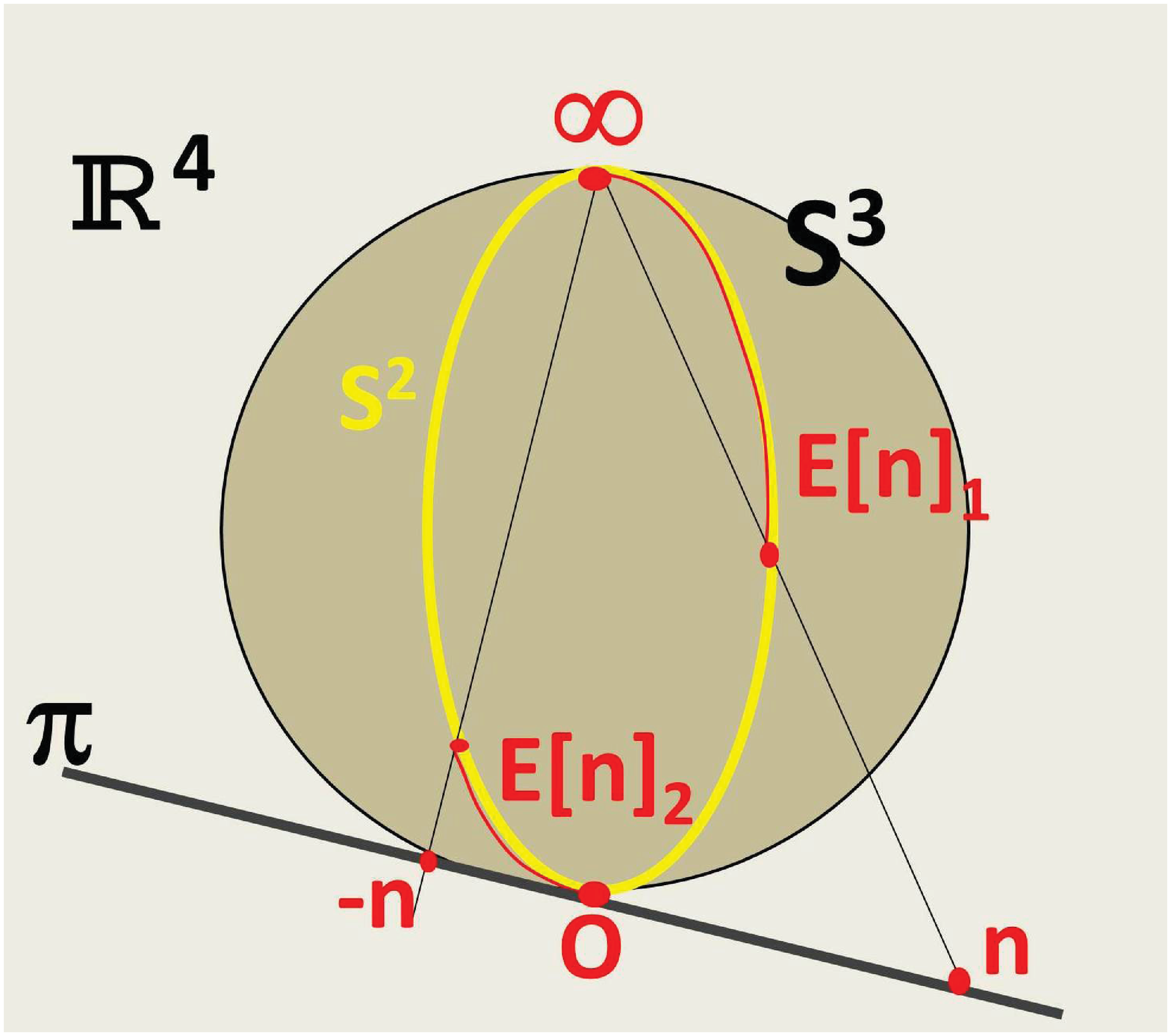}
\renewcommand{\figurename}{Fig.}
\caption{Representation of some compactified elliptic curves $E[n]:\, y^2=x^3-n^2\, x$. $E[n]=E[n]_1\sqcup E[n]_2$ where $E[n]_1$ is the closed curve on the right-hand of the (compactified) $y$-axis and and $E[n]_2$ is the other part on the left-hand of the $y$-axis. Similarly for $E[n']=E[n']_1\sqcup E[n']_2$. When $n,\, n'\in\mathbb{N}_{congr}$ we can talk about congruent elliptic curves. The points $n$ and $n'$ are placed on the (compactified) $x$-axis. $S^2=\mathbb{R}^2\bigcup\{\infty\}$. In the figure (B) is represented the relation between the compactified elliptic curve on $S^2\subset S^3\cong G^+_{1,4}(\mathbb{R}^4)\subset\mathbb{R}^4$, and its stereographic projection, of equation $y^2-x^3+n^2 x=0$, on the plane $\pi$, tangent at the south pole $O$ of $S^2$. In figure (C) is represented $S^3\subset\mathbb{R}^4$, and $S^2\subset S^3$, with the two components $E[n]_1$, $E[n]_2$ of an elliptic curve represented in read. The plane $\pi$, tangent to $S^2$, at the south pole $O$, is represented by a grey strong straight-line. There the steoreographic projection from the north pole $\infty$, identifies the two components of the plane curve, besides the point $-n$ and $n$. [Here $G^+_{1,4}(\mathbb{R}^4)$ is the oriented Grassman manifold of oriented $1$-dimensional planes in $\mathbb{R}^4$.]}
\label{representation-compactified-elliptic-curve}
\end{figure}

\begin{definition}[$n$-Elliptic bordism groups]\label{n-elliptic-bordism-groups}
Let $n\in\mathbb{N}$. We say that two points $P,\, Q\in \mathbb{R}^2$ {\em $n$-elliptic bord} if there exists an elliptic curve $E[n]:\, y^2=x^3-n^2x$, such that $P,\, Q\in E[n]$, and $P\sqcup Q=\partial \Gamma$, $\Gamma\subset E[n]$. Since the $0$-bordism group $\Omega_0(\mathbb{R}^2)\cong\mathbb{Z}_2$, it follows that if $P\in[Q]\in\Omega_0(\mathbb{R}^2)\cong\mathbb{Z}_2$, it does not necessitate that $P\sqcup Q=\partial\Gamma$, for some $\Gamma\subset E[n]$. Therefore the $n$-elliptic bordism is a new equivalence relation in $\mathbb{R}^2$ and we denotes by $\Omega_{E[n]}$ the corresponding set of equivalence classes, that we call the {\em $n$-elliptic bordism group}.
\end{definition}

\begin{theorem}\label{proposition-isomorphic-elliptic-bordism-groups}
One has the canonical isomorphism
\begin{equation}\label{isomorphic-elliptic-bordism-groups}
\Omega_{E[n]}\cong \Omega_{E[s^2n]},\, \forall s\in\mathbb{N}.
\end{equation}
In particular, two integers $n,\, \bar n\in\mathbb{N}$, belonging to the same equivalence class of congruent numbers, identify isomorphic elliptic bordism groups.
\end{theorem}
\begin{proof}
In fact the elliptic curves $E[n]$ and $E[s^2n]$ coincide up to diffeomorphisms. Furthermore, if $P=(x_P,y_P)\in E[n]$ then $\bar P=(\bar x_P,\bar y_P)=(s^2 x_P,s^3 y_P)$ is another point of $E[\bar n]$, since the diffeomorphism $\varphi:\mathbb{R}^2\to\mathbb{R}^2$, $(x,y)\mapsto(s^2 x,s^3 y)$ transforms $E[n]$ into $E[\bar n]$, $\bar n=s^2 n$. (Lemma \ref{lemma-residue-classes-congruent-numbers}.) Thus $\varphi$ induces an isomorphism between the corresponding elliptic bordism groups $\Omega_{E[n]}$ and $\Omega_{E[\bar n]}$.
\end{proof}

\begin{cor}\label{relation-elliptic-bordism-groups-congruent-integers}
In particular, if the integer $m\in\mathbb{N}$, belongs to the equivalence class of a strong-congruent number $n\in\mathbb{N}_{congr}$, hence $n$ is the square-free part of $m$, then there exists a diffeomorphism $\varphi:\mathbb{R}^2\to\mathbb{R}^2$ that induces an isomorphism $\varphi_*:\Omega_{E[n]}\cong\Omega_{E[m]}$
between the corresponding elliptic bordism groups.
\end{cor}

\begin{proposition}\label{proposition-n-elliptic-bordism-groups-structure}
Let $\Omega_{E[n]}$ be the {\em $n$-elliptic bordism group}, $n\in\mathbb{N}$. Then one has the isomorphism
\begin{equation}\label{n-elliptic-bordism-groups-structure}
\Omega_{E[n]}\cong\mathbb{Z}_2.
\end{equation}
\end{proposition}
\begin{proof}
In fact the elliptic curve $E[n]$ is made by two disjoint components: $E[n]=E[n]_1\sqcup E[n]_2$. (See Fig. \ref{representation-elliptic-curve}.) If $P$ and $Q$ belong to the same component, $P,\, Q\in E[n]_i$, $i=1.2$, then $P\sqcup Q=\partial \Gamma_i$, with $\Gamma\subset E[n]_i$. Instead, if  $P\in E[n]_1$ and  $Q\in E[n]_2$, then does not exist a curve $\Gamma\subset E[n]$, such that  $P\sqcup Q=\partial \Gamma$ with $\Gamma\subset E[n]_i$. This is enough to state that isomorphism (\ref{n-elliptic-bordism-groups-structure}) holds.
\end{proof}

\begin{definition}[$n$-Congruent bordism groups]\label{n-congruent-bordism-groups}
Let $n\in\mathbb{N}_{congr}$. We say that two points $P,\, Q\in \mathbb{R}^2$ {\em $n$-congruent bord} if there exists an elliptic curve $E[n]:\, y^2=x^3-n^2x$, such that $P,\, Q\in E[n]$, $P\sqcup Q=\partial \Gamma$, $\Gamma\subset E[n]$, and $P$ and $Q$ are rational points of $E[n]$.\footnote{Let us emphasize that $E[n]$ is an elliptic curve having infinitely many rational points. (See Corollary \ref{infinitely-many-rational-points}).} Since the $0$-bordism group $\Omega_0(\mathbb{R}^2)\cong\mathbb{Z}_2$, it follows that if $P\in[Q]\in\Omega_0(\mathbb{R}^2)\cong\mathbb{Z}_2$, it does not necessitate that $P\sqcup Q=\partial\Gamma$, for some $\Gamma\subset E[n]$. Therefore the $n$-congruent bordism is a new equivalence relation in $\mathbb{R}^2$ and we denotes by $\Omega^{\mathbb{Q}}_{E[n]}$, $n\in\mathbb{N}_{congr}$, the corresponding set of equivalence classes, that we call the {\em $n$-congruent bordism group}. {\rm[}More shortly $\Omega_{E[n]}$, is a $n$-congruent bordism group, i.e., $\Omega_{E[n]}=\Omega^{\mathbb{Q}}_{E[n]}$, iff $n$ is a strong-congruent number, i.e., $n\in\mathbb{N}_{congr}$.{\em]}
\end{definition}

\begin{proposition}\label{proposition-n-congruent-bordism-groups-structure-a}
$\bullet$\hskip 2pt Let $\Omega^{\mathbb{Q}}_{E[n]}$ and $\Omega^{\mathbb{Q}}_{E[n']}$ be two congruent bordism groups, with $n\not=n'\in\mathbb{N}_{congr}$. Then the hypothesis that two non-zero rational points $P,\, Q\in\mathbb{Q}^2$, $y_P\not=0$, $y_Q\not=0$, satisfy the condition $P\sqcup Q=0\in \Omega^{\mathbb{Q}}_{E[n]}$ excludes that could be also $P\sqcup Q=0\in \Omega^{\mathbb{Q}}_{E[n']}$.

$\bullet$\hskip 2pt One has the following isomorphism: $\Omega^{\mathbb{Q}}_{E[n]}\cong\mathbb{Z}_2$.
\end{proposition}
\begin{proof}
$\bullet$\hskip 2pt In fact the intersection at finite, of the elliptic curves $E[n]$ and $E[n']$ is only the point $O=(0,0)\in\mathbb{Q}^2$. Furthermore. the point $O$ is not considered a rational point of the elliptic curves $E[n]$ and $E[n']$. (See Lemma \ref{lemma-congruent-numbers-as-elliptic}.)

$\bullet$\hskip 2pt This proof is analogous to the one for the group $\Omega_{E[n]}$. In fact, the elliptic curve $E[n]$, $n\in\mathbb{N}_{congr}$, has the same structure. Furthermore, on the same connected component of $E[n]_i$, there exists always a rational point and its symmetric with respect to the $x$-axis.
\end{proof}
\begin{proposition}[Relation between congruent-bordisms and elliptic-bordisms]\label{relation-between-congruent-elliptic-bordism-groups}
Let $n\in\mathbb{N}$ be a congruent integer. The isomorphism $\Omega^{\mathbb{Q}}_{E[n]}\cong\mathbb{Z}_2\cong \Omega_{E[n]}$ means that any couple of points $P,\, Q\in E[n]_i$, namely belonging to the same connected component of $E[n]$, there corresponds a couple of rational points $R,\, S\in E[n]_i$.
\end{proposition}
\begin{proof}
It is useful to consider congruent elliptic curves on the compactified plane $S^2=\mathbb{R}^2\bigcup\{\infty\}$. (See Fig. \ref{representation-compactified-elliptic-curve}.) Really we have the following lemma.

\begin{lemma}\label{representation-elliptic-curves-into-projective-space}
Compactified elliptic curves that are obtained by the following ones $E: y^2-x^3-q^2 x=0$, can be represented into $\mathbb{P}^3(\mathbb{R})$.
\end{lemma}

\begin{proof}
Since $E\subset \mathbb{R}^2$, it follows that the Alexandrov compactification $\mathbb{R}\bigcup\{\infty\}\cong S^2$ identifies a natural curve in $S^2$, by means of the inclusion $\mathbb{R}^2\hookrightarrow S^2$.  Therefore the compactified elliptic curve $E^+$ is just $E^+=E\bigcup\{\infty\}$. This is just a $1$-dimensional smooth compact submanifold of $S^2$, passing for the point $\infty\in S^2$. Let us recall that $S^p$ can be identified with an oriented Grassmann manifold and a symmetric space by means of the diffeomorphisms:\footnote{See, e.g., \cite{PRAS1}.}
\begin{equation}\label{diffeomorphisms-oriented-grassman-manifolds-sphere}
S^p\cong G^+_{1,p+1}(\mathbb{R}^{p+1})\cong SO(1+p)/SO(1)\times SO(p).
\end{equation}
(Incidentally recall that these are also Einstein manifolds, i.e., the Ricci tensor is proportional to the metric tensor.) The group $SO(1+p)$ acts transitively on $G^+_{1,p+1}(\mathbb{R}^{p+1})$. $SO(1)\times SO(p)$ is the isotropy group of the point $\pi\in G^+_{1,p+1}(\mathbb{R}^{p+1})$, where $SO(1)$ acts in the oriented $1$-dimensional plane $\pi$ and $SO(p)$ acts in its orthogonal complement. Let us emphasize that forgetting the orientation, we can consider Grassman manifold $G_{1,1+p}(\mathbb{R}^{1+p})\cong \mathbb{P}^p(\mathbb{R}^{1+p})$. On the other hand we have the following exact commutative diagram:
\begin{equation}\label{fiber-bundle-structures-sphere-projective-spaces}
  \xymatrix{S^0\ar@{=}[d]\ar@{^{(}->}[r]&S^p\ar@{=}[d]^{\wr}\ar[r]&\mathbb{P}^{p}(\mathbb{R})\ar@{=}[d]^{\wr}\ar[r]&0\\
  S^0\ar@{^{(}->}[r]&G^+_{1,1+p}(\mathbb{R})\ar[r]&G_{1,1+p}(\mathbb{R})\ar[r]&0\\}
\end{equation}
Diagram (\ref{fiber-bundle-structures-sphere-projective-spaces}) emphasizes the fiber bundle structures $S^p\to \mathbb{P}^{p}(\mathbb{R})$ and $G^+_{1,1+p}(\mathbb{R})\to G_{1,1+p}(\mathbb{R})$ both with fiber $S^0$.
Therefore by considering the natural inclusion $S^{p-1}\hookrightarrow S^{p}$, such that the following diagram is commutative:
\begin{equation}\label{embedding-of-spheres}
  \xymatrix{(S^{p-1},*,\infty)\ar@{=}[d]^{\wr}\ar@{^{(}->}[r]&(S^p,*,\infty)\ar@{=}[d]^{\wr}\\
  (\mathbb{R}^{p-1},0)\bigcup\{\infty\}\ar@{^{(}->}[r]&(\mathbb{R}^{p},0)\bigcup\{\infty\}\\}
\end{equation}
one can represent $E^+\subset S^2$ into $S^3$, hence by means of the projection $S^3\to \mathbb{P}^3(\mathbb{R})$, we get the representation of $E^+$ into $\mathbb{P}^3(\mathbb{R})$.\footnote{The same result can be obtained by rewriting equations (\ref{set-a-n}) defining $E$, in projective way, namely considering instead of the coordinates $(a,b,c)\in \mathbb{R}^3$, the homogeneous coordinates $[a,b,c,d]\in \mathbb{P}^3(\mathbb{R})$. (We skip on details.)}
 \end{proof}
 In fact the point $\infty$ is a distinguished point considered the unity in the group $\mathbf{E}[n]$, hence it is assumed a rational point. Then the isomorphism $\Omega^{\mathbb{Q}}_{E[n]}\cong\mathbb{Z}_2\cong \Omega_{E[n]}$ can be justified, since to $P\bigcup Q\in 0\in\Omega_{E[n]}$, one can correspond (even if not canonically) a couple of rational points $R,\, S\in E[n]_i$, such that $R\bigcup S\in 0\in \Omega^{\mathbb{Q}}_{E[n]}$. Really since $E[n]$ contains an infinity number of rational points it follows that in the two branches of $E[n]$, bounded by $P\bigcup Q$ must be present at least one rational point, say $R$. Then for symmetry there exists also another rational point, say $S$, on the same connected component of $E[n]$. Therefore $R\bigcup S\in 0\in\Omega^{\mathbb{Q}}_{E[n]}$.
\end{proof}

\begin{theorem}[Elliptic and congruent-bordism groups and homotopies in $\mathbb{R}^2$]\label{congruent-bordism-groups-homotopies-theorem}
$\bullet$\hskip 2pt The isomorphisms of groups $\Omega_{E[n]}\cong\Omega_{E[n']}$, $n\not=n'\in\mathbb{N}$, can be induced by suitable homotopies of $\mathbb{R}^2$.

$\bullet$\hskip 2pt Let us assume that $n$ and $n'$ are congruent integers. Then the homotopy relating $E[n]$ to $E[n']$ does not necessitate identify an isogeny. When this happens, it induces a group homomorphisms $\mathbf{E}[n]\to \mathbf{E}[n']$. In particular this recurs when $n'$ and $n$ belong to the same equivalence class of congruent numbers, namely $n'=m^2 n$, with $m\in\mathbb{Q}$. In such a case the homomorphism $\mathbf{E}[n]\to \mathbf{E}[n']$ becomes an isomorphism.

$\bullet$\hskip 2pt In advance, if $n\in\mathbb{N}_{congr}$, and the homotopy induces an homorphism $\mathbf{E}[n]\to \mathbf{E}[n']$, then $n'$ cannot be a strong-congruent number too.
\end{theorem}
\begin{proof}
$\bullet$\hskip 2pt The motivation that two elliptic bordism groups $\Omega_{E[n]}\cong\Omega_{E[n']}$, $n\not=n'\in\mathbb{N}$, are isomorphic to $\mathbb{Z}_2$ follows from the fact that any two elliptic curves $E[n]:\, y^2-x^3+n^2x=0$ and $E[n']:\, y^2-x^3+n'^2x=0$, are homotopic, i.e., there exists an homotopy (flow) in $\mathbb{R}^2$, $\varphi_\lambda:\mathbb{R}^2\to\mathbb{R}^2$, $\lambda\in[0,1]\subset\mathbb{R}$, relating $E[n]$ to $E[n']$. More precisely, let us consider the following deformed elliptic curves:

\begin{equation}\label{equation-deformed-elliptic-curve}
E[n_\lambda]:\, y^2-x^3+n_\lambda^2x=0,\, \left\{\begin{array}{l}
                                                                           n_\lambda=n+\lambda(n'-n)\\
                                                                           \lambda\in[0,1]\subset\mathbb{R}.\\
                                                                         \end{array}\right.
\end{equation}
One can see that

\begin{equation}\label{equation-deformed-elliptic-curve-a}
  \scalebox{0.9}{$\left\{\begin{array}{ll}
    n_0=n&\\
    n_1=n'&\\
    E[n_0]:& y^2-x^3+n_0^2x=0\, \Rightarrow\, y^2-x^3+n^2x=0\, \Rightarrow\,  E[n_0]=E[n].\\
    E[n_1]:& y^2-x^3+n_1^2x=0\, \Rightarrow\, y^2-x^3+n'^2x=0\, \Rightarrow\, E[n_1]=E[n'].\\
    \end{array}\right.$}
\end{equation}

We can see that such deformed elliptic curves can be realized by means of the following homotopy, (flow), in $\mathbb{R}^2$:\footnote{Note that $\varphi_\lambda$ is a diffeomorphism $\mathbb{R}^2\to\mathbb{R}^2$, for any $\lambda\in[0,1]$. In fact, its jacobian $j(\varphi_\lambda)=(\frac{n}{n_\lambda})^{5/2}$, hence $j(\varphi_\lambda)\in[1,(\frac{n}{n'})^{5/2}]$, (whether $n>n'$). Since $\lambda\in[0,1]$, it follows that $j(\varphi_\lambda)$ is an irrational number, except for its boundary value $\lambda=0$.}
\begin{equation}\label{flow-homotopy}
    \varphi_\lambda:(x,y)\mapsto (x_\lambda,y_\lambda)=(\frac{n}{n_\lambda}x,(\frac{n}{n_\lambda})^{3/2}y).
\end{equation}
In fact, starting from $E[n]$, by using the flow $\varphi_\lambda$, we get $(\frac{n}{n_\lambda})^3[y^2-x^3+n_\lambda^2x]=0$. So deforming the equation $y^2-x^3+n^2x=0$ of $E[n]$, by means of the flow $\varphi_\lambda$, we get the deformed elliptic curves $E[n_\lambda]$, hence also $E[n']$ for $\lambda=1$. Thus we can state that the flow $\varphi_\lambda$ on $\mathbb{R}^2$ induces the homotopy $\psi_\lambda$ between the elliptic curves $E[n]$ and $E[n']$. This situation is resumed in the diagram (\ref{homotopic-transformation-a}).
\begin{equation}\label{homotopic-transformation-a}
\xymatrix{E[n]\ar@{.>}[d]\ar@/^2pc/[rrr]_{\psi_1}\ar@/^3pc/[rrrr]_{\psi_1}\ar[r]^{\psi_\lambda}&E[n_\lambda]\ar@{.>}[d]\ar[r]&\cdots\ar[r]&
E[n_1]\ar@{.>}[d]\ar@{=}[r]&E[n']\\
\mathbb{R}^2\ar@/_2pc/[rrr]^{\varphi_1}\ar[r]_{\varphi_\lambda}&\mathbb{R}^2\ar[r]&\cdots\ar[r]&\mathbb{R}^2&\\}
\end{equation}
There the vertical dots-lines mean that the relation between the elliptic curves on the top horizontal line is induced from the flow on $\mathbb{R}^2$, namely the plane where are embedded these curves.
Therefore (\ref{flow-homotopy}) has the effect to induce an isomorphism between the bordism groups $\Omega_{E[n]}$ and $\Omega_{E[n']}$.

$\bullet$\hskip 2pt Let us, now, assume that $n$ and $n'$ are both integer congruent numbers. Taking into account the one-to-one correspondence between $E[n]$ and $A[n]$, $\forall n\in \mathbb{N}$, (Lemma \ref{lemma-congruent-numbers-as-elliptic}), we get the following commutative diagram:
\begin{equation}\label{homotopic-transformation-b}
\xymatrix{E[n]\ar[d]\ar[r]^{\psi_\lambda}\ar@/^2pc/[rrr]^{\psi_1}&E[n_\lambda]\ar[d]\ar[r]&\cdots\ar[r]&E[n']\ar[d]\\
A[n]\ar[u]\ar[r]_{\phi_\lambda}\ar@/_2pc/[rrr]^{\phi_1}&A[n_\lambda]\ar[u]\ar[r]&\cdots\ar[r]&A[n']\ar[u]\\\\}
\end{equation}
\begin{equation}\label{homotopic-transformation-c}
\scalebox{0.6}{$\xymatrix{(x,y)\ar[d]\ar@{|->}[r]^{\psi_\lambda}&(x_\lambda,y_\lambda)\ar[d]\ar@{|->}[r]&\cdots\ar@{|->}[r]&(x_1,y_1)\ar[d]\\
(a=\frac{x^2-n^2}{y},b=2n\frac{x}{y},c=\frac{x^2+n^2}{y})\ar[u]\ar@{|->}[r]_{\phi_\lambda}&
(a_\lambda=\frac{x_\lambda^2-n_\lambda^2}{y_\lambda},b_\lambda=2n_\lambda\frac{x_\lambda}{y_\lambda},c_\lambda=\frac{x_\lambda^2+n_\lambda^2}{y_\lambda})\ar[u]\ar@{|->}[r]&\cdots\ar@{|->}[r]&
(a'=\frac{x_1^2-n'^2}{y_1},b'=2n'\frac{x_1}{y_1},c'=\frac{x_1^2+n'^2}{y_1})\ar[u]\\}$}
\end{equation}

$\bullet$\hskip 2pt If $(x,y)\in E[n]$ is a rational point, the corresponding $(x_1,y_1)\in E[n']$ does not necessitate to be a rational point too. This should happen when the map $\psi_1$ is an isogeny. In fact in such a case the corresponding to rational points are rational points too and the diffeomorphisms $\psi_1$ induce homomorphisms between the groups  $\mathbf{E}[n]$ and $\mathbf{E}[n']$.  Therefore, if  $(x,y)\in\mathbb{Q}^2$, is a rational point of $E[n]$, then for the corresponding point on $(x_1,y_1)\in E[n']$, one has $(x_1,y_1)\in \mathbb{Q}^2$ iff $(x_1,y_1)=(\frac{n}{n'}x,(\frac{n}{n'})^{3/2}y)\in \mathbb{Q}^2$. This condition is satisfied iff $\frac{n}{n'}=m^2$, for $m\in\mathbb{Q}$, or equivalently $n'=\bar m^2\cdot n$, with $\bar m=1/m$. Therefore,  $\psi_1$ is an isogeny iff the integers  $n'$ and $n$ belong to the same congruent class.\footnote{This agree with the property that two alliptic curves $E$ and $E'$ are isogenous iff there is a morphism of varieties defined by a rational map between $E$ and $E'$, which is also a group homorphism between the corresponding groups $\mathbf{E}$ and $\mathbf{E}'$, sending $\infty\in E$ to $\infty\in E'$. In fact, in the actual situation the morphism sending $E[n]$ to $E[n']$ is the diffeomorphism $\psi_1$, induced from $\varphi_1$ given in (\ref{flow-homotopy}), that in order to be a rational mapping must be $n'=\bar m^2\cdot n$, $\bar m\in\mathbb{Q}$.}

$\bullet$\hskip 2pt In particular  if $n\in\mathbb{N}_{congr}$, then must be $n'=m^2\cdot n$, hence $n'$ cannot be a strong-congruent number too.

However, from the congruent bordism point of view it is not necessary that above homotopies should be isogenies. In fact, we can use Proposition \ref{relation-between-congruent-elliptic-bordism-groups} to identify elliptic bordism with congruent-ones. Really one has the isomorphism: $\Omega^{\mathbb{Q}}_{E[n]}\cong\mathbb{Z}_2\cong\Omega^{\mathbb{Q}}_{E[n']}$.
\end{proof}

\begin{cor}[Elliptic and congruent-bordism groups and homotopies in $\mathbb{R}^2$]\label{corollary-congruent-bordism-groups-homotopies-theorem}
Let $E[n]:\, y^2-x^3+n^2x=0$, $n\in\mathbb{N}_{congr}$, be a congruent elliptic curve, and $E[n']:\, y^2-x^3+n'^2x=0$, $n\in\mathbb{N}$, another elliptic curve. Then there exists an homotopy $\psi_\lambda:E[n]\to E[n']$,\footnote{We denote this circumstance by $E[n]\simeq E[n']$, and we say that $E[n]$ is {\em homotopic} to $E[n']$.} inducing the isomorphism of groups $\Omega^{\mathbb{Q}}_{E[n]}\cong\Omega_{E[n']}$. Furthermore, if $E[n']$ is also a congruent elliptic curve, namely ${\rm rank}(E[n'])>0$, then this homotopy induces also the isomorphism $\Omega^{\mathbb{Q}}_{E[n]}\cong\Omega^{\mathbb{Q}}_{E[n']}$. Furthermore, if $n'=m^2\cdot n$, $m\in\mathbb{Q}$, then the mapping $\psi_1:E[n]\to E[n']$, induces an isogeny $(\psi_1)_*:\mathbf{E}[n]\to \mathbf{E}[n']$ that is a group isomorphism: $\mathbf{E}[n]\cong \mathbf{E}[n']$.

The commutative diagram {\em(\ref{resuming-commutative-diagram-isomorphisms-between-elliptic-bordism-groups-congruent})} summarizes above results in this section about the elliptic bordism groups and the congruent bordism groups.

\begin{equation}\label{resuming-commutative-diagram-isomorphisms-between-elliptic-bordism-groups-congruent}
    \xymatrix{\mathbb{Z}_2\ar@{=}[d]
    \ar@{=}[r]^{\thicksim}&\Omega_{E[s^2n]}\ar@{=}[d]^{\wr}_{\varphi_*}\ar@{=}[r]^{\thicksim}&\Omega^{\mathbb{Q}}_{E[s^2n]}\ar@{=}[d]^{\wr}_{\varphi_*}\\
    \mathbb{Z}_2\ar@{=}[d]^{\wr}\ar@{=}[r]^{\thicksim}&\Omega_{E[n]}\ar@{=}[d]^{\wr}_{b}\ar@{=}[r]^{\thicksim}&\Omega^{\mathbb{Q}}_{E[n]}\ar@{=}[d]^{\wr}_{c}\\
    \Omega_{E[m]}\ar@{=}[r]^{\thicksim}_{a}&\Omega_{E[n']}\ar@{=}[r]^{\thicksim}_{d}&\Omega^{\mathbb{Q}}_{E[n']}\\}
\end{equation}
There $n\in\mathbb{N}$, $n\not=n'\in\mathbb{N}$, $s\in\mathbb{Q}$, the isomorphisms $a$, $b$ and $c$ are induced by homotopies and $\varphi_*$ are induced by diffeomorphisms $\varphi:\mathbb{R}^2\to\mathbb{R}^2$. The isomorphisms $c$ and $d$ exist if $n'$ is also a congruent integer. Otherwise $\Omega^{\mathbb{Q}}_{E[n']}=\varnothing$.
\end{cor}

\begin{lemma}[Elliptic and congruent bordism groups vs. congruent number problem]\label{elleiptic-congruent-bordism-groups-and-birch-swinnerton-dyer-conjecture}
Set $$\Omega^{\mathbb{Q}}_{E[\bullet]}=\prod_{n\in\mathbb{N}_{congr}}\Omega^{\mathbb{Q}}_{E[n]}$$ and $$\Omega_{E[\bullet]}=\prod_{n\in{}_{\square}\mathbb{N}}\Omega_{E[n]}.$$ Since $\Omega^{\mathbb{Q}}_{E[\bullet]}\cong\mathbb{N}_{congr}\times\mathbb{Z}_2$ and $\Omega_{E[\bullet]}\cong{}_{\square}\mathbb{N}\times\mathbb{Z}_2$, we can consider both as trivial fiber bundles over the group $\mathbb{Z}_2$, with discrete fiber $\mathbb{N}_{congr}$ and ${}_{\square}\mathbb{N}$ respectively. In other words $\Omega^{\mathbb{Q}}_{E[\bullet]}$ and $\Omega_{E[\bullet]}$ are coverings of $\mathbb{Z}_2$. In the following we shall consider $\mathbb{N}_{congr}$ and ${}_{\square}\mathbb{N}$ as topological discrete spaces pointed at $*=5$.

Then one has the following commutative diagram (in the category of coverings):

\begin{equation}\label{commutative-diagram-bordism-related-b-s-d-conjecture}
   \scalebox{0.9}{$ \xymatrix{0\ar[r]&\Omega^{\mathbb{Q}}_{E[\bullet]}\ar@{=}[d]^{\wr}\ar[r]&
    \Omega_{E[\bullet]}\ar@{=}[d]^{\wr}\ar[r]&\Omega_{E[\bullet]}/\Omega^{\mathbb{Q}}_{E[\bullet]}\ar@{=}[d]^{\wr}\ar[r]&0&\\
    0\ar[r]&\mathbb{N}_{congr}\times\mathbb{Z}_2\ar[d]\ar[r]&{}_{\square}\mathbb{N}\times
    \mathbb{Z}_2\ar[d]\ar[r]&\framebox{$({}_{\square}\mathbb{N}/\mathbb{N}_{congr})\times\mathbb{Z}_2$}\ar[r]&0&\\
    0\ar[r]&\mathbb{N}_{congr}\ar[d]^{i}\ar[r]&{}_{\square}\mathbb{N}\ar[d]\ar[r]^{L_\bullet}&\mathbb{Z}\ar[r]&{\rm coker}(L_\bullet)\ar[r]&0\\
    &0&0&&&\\}$}
\end{equation}
with $L_\bullet$ defined in {\em(\ref{l-bullet-mapping})}.

\begin{equation}\label{l-bullet-mapping}
    L_\bullet(n)=\left\{
    \begin{array}{ll}
      2A_n-B_n& \hbox{\rm if $n$ is even}\\
      2C_n-D_n& \hbox{\rm if $n$ is odd}\\
    \end{array}
    \right.
\end{equation}
where $A_n$, $B_n$, $C_n$ and $D_n$ are defined in {\rm(\ref{associated-integers})}. The first two vertical lines are exact, and so are the first two horizontal lines in {\rm(\ref{commutative-diagram-bordism-related-b-s-d-conjecture})}. Then in order to solve the congruent problem it is enough to prove that also the bottom horizontal line in {\rm(\ref{commutative-diagram-bordism-related-b-s-d-conjecture})} is exact. Then, by using the relation between BS-D conjecture and congruent problem, we get an indirect proof that the BS-D conjecture is true for elliptic curves of the type $y^2-x^3+n^2 x=0$, {\em(weak BS-D conjecture)}.
\end{lemma}
\begin{proof}
The first part of the lemma directly follows from above results. Furthermore, note that in general one has that the bottom horizontal line in (\ref{commutative-diagram-bordism-related-b-s-d-conjecture}) is  a $0$-sequence. It is also exact at all points except at ${}_{\square}\mathbb{N}$ and one has
\begin{equation}\label{standard-set-associated-to-commutative-diagram-a}
    \left\{
    \begin{array}{l}
      {\rm coim}(L_\bullet)={}_{\square}\mathbb{N}/{\rm ker}(L_\bullet)\cong{\rm im}(L_\bullet)\\
       {\rm coker}(L_\bullet)=\mathbb{Z}/{\rm im}(L_\bullet).\\
      \end{array}
    \right.
\end{equation}
Therefore, to state that the bottom horizontal line in (\ref{commutative-diagram-bordism-related-b-s-d-conjecture}) is exact is equivalent to state that ${\rm ker}(L_\bullet)=\mathbb{N}_{congr}$ and, from (\ref{standard-set-associated-to-commutative-diagram-a}) we get also
 \begin{equation}\label{standard-set-associated-to-commutative-diagram-b}
    \left\{
    \begin{array}{l}
      {\rm coim}(L_\bullet)={}_{\square}\mathbb{N}/\mathbb{N}_{congr}\cong{\rm im}(L_\bullet)\\
       {\rm coker}(L_\bullet)=\mathbb{Z}/({}_{\square}\mathbb{N}/\mathbb{N}_{congr}).\\
      \end{array}
    \right.
  \end{equation}
  In such a case we can complete commutative diagram (\ref{commutative-diagram-bordism-related-b-s-d-conjecture}) with (\ref{commutative-diagram-bordism-related-b-s-d-conjecture-a}).

  \begin{equation}\label{commutative-diagram-bordism-related-b-s-d-conjecture-a}
   \scalebox{0.9}{$ \xymatrix{0\ar[r]&\Omega^{\mathbb{Q}}_{E[\bullet]}\ar@{=}[d]^{\wr}\ar[r]&
    \Omega_{E[\bullet]}\ar@{=}[d]^{\wr}\ar[r]&\Omega_{E[\bullet]}/\Omega^{\mathbb{Q}}_{E[\bullet]}\ar@{=}[d]^{\wr}\ar[r]&0&\\
    0\ar[r]&\mathbb{N}_{congr}\times\mathbb{Z}_2\ar[d]\ar[r]&{}_{\square}\mathbb{N}\times
    \mathbb{Z}_2\ar[d]\ar[r]&\framebox{$({}_{\square}\mathbb{N}/\mathbb{N}_{congr})\times\mathbb{Z}_2$}\ar[d]^{d}\ar[r]&0&\\
    0\ar[r]&\mathbb{N}_{congr}\ar[d]\ar[r]&{}_{\square}\mathbb{N}\ar[d]\ar[r]^{L_\bullet}&\mathbb{Z}\ar[r]^(0.4){e}&{\rm coker}(L_\bullet)\ar[r]&0\\
    &0&0&&&\\}$}
\end{equation}
In  (\ref{commutative-diagram-bordism-related-b-s-d-conjecture-a}) $d$ is defined by composition
$$\xymatrix{({}_{\square}\mathbb{N}/\mathbb{N}_{congr})\times\mathbb{Z}_2\ar@/_2pc/[rr]_{d}\ar[r]&{}_{\square}\mathbb{N}/\mathbb{N}_{congr}\ar[r]&\mathbb{Z}\\}$$

Therefore, one has $\IM(d)=\ker(e)$, hence all the commutative diagram (\ref{commutative-diagram-bordism-related-b-s-d-conjecture-a}) is exact too.

 This assures that the conditions in (\ref{implications}) are also sufficient to state that $n$ is a strong-congruent integer, namely $n\in\mathbb{N}_{congr}$. Then taking into account Theorem \ref{coates-wiles-theorem} we get an indirect proof that the Birch-Swinnerton-Dyer conjecture is true.
\end{proof}

We are ready, now, to obtain our main result.
\begin{theorem}[The congruent problem solved]\label{main-theorem}
The identification of the strong-congruent numbers with the kernel of $L_\bullet$, $\mathbb{N}_{congr}\cong{\rm ker}(L_\bullet)$, allows us to decide with a finite number of steps, whether a square-free number $b$ is a strong-congruent number. In this way we identify all the equivalence classes of congruent numbers.
\end{theorem}
\begin{proof}
Let us assume  $n'\in{\rm ker}(L_\bullet)\subset {}_{\square}\mathbb{N}$, and $n\in\mathbb{N}_{congr}$. Then, from Corollary \ref{corollary-congruent-bordism-groups-homotopies-theorem} we say that there exists a homotopy $\psi_\lambda:E[n]\to E[n']$ inducing the isomorphism of groups $\Omega^{\mathbb{Q}}_{E[n]}\cong\Omega_{E[n']}$. Then we can identify a diffeomorphism $\varphi:\mathbb{R}^2\to\mathbb{R}^2$, $(x,y)\mapsto (x',y')$, sending the rational point $(x,y)\in E[n]$, to a point $(x_1,y_1)\in E[n']$, such that $(x_1=x\frac{n}{n'},y_1=(\frac{n}{n'})^{3/2}y)$. Therefore, to the point $(x,y)$ on $E[n]$ corresponds the strong-congruent right triangle $(a=\frac{x^2-n^2}{y},b=2n\frac{x}{y},c=\frac{x^2+n^2}{y})$, and to the corresponding point on $E[n']$ it is associated the right triangle
$(a'=\frac{x_1^2-n'^2}{y_1},b'=2n'\frac{x_1}{y_1},c'=\frac{x_1^2+n'^2}{y_1})$. The direct expression of these sides in term of $(x,y)$, is given in (\ref{relations-between-sides-points-on-elliptic-curves}).
\begin{equation}\label{relations-between-sides-points-on-elliptic-curves}
    \left\{a'=\frac{x^2n^2-n'^4}{y}\, (n^3n')^{-\frac{1}{2}},\,
   b'=2\frac{x}{y}\, (\frac{n'^3}{n})^{\frac{1}{2}},\,
   c'=\frac{x^2n^2+n'^4}{y} \, (n'n^3)^{-\frac{1}{2}}\right\}.
\end{equation}
The point $(x_1,y_1)$ is rational iff $n'=\bar m^2\cdot n$, with $\bar m\in \mathbb{Q}$. In such a case
$a'$, $b'$ and $c'$ are rational numbers too. But  $n'=m^2\cdot n$ contradicts the assumption that $n'\in{}_{\square}\mathbb{N}$. Therefore, if the map $\psi_1:E[n] \to E[n']$, induces an isogeny  $(\psi_1)_*:\mathbf{E}[n] \to \mathbf{E}[n']$, must necessarily be $n'\in\mathbb{N}_{congr}$, hence $n'=n$.\footnote{This means that for two elliptic curves $E_n$ and $E_{n'}$, such that $n,\, n'\in\mathbb{N}_{congr}$, namely $n$ and $n'$ are both strong-congruent numbers, cannot exist an isogeny $\mathbf{E}[n]\to \mathbf{E}[n']$, according with Theorem \ref{congruent-bordism-groups-homotopies-theorem}.} On the other hand, we have also the following lemma.

\begin{lemma}\label{main-lemma}
$\ker(L_\bullet)=\mathbb{N}_{congr}$.
 \end{lemma}
 \begin{proof}
Set $\widetilde{\Omega}_{E[\bullet]}=\prod_{n\in\ker(L_\bullet)}\Omega_{E[n]}$. Since $\widetilde{\Omega}_{E[\bullet]}\cong\ker(L_\bullet)\times\mathbb{Z}_2$ one can consider $\widetilde{\Omega}_{E[\bullet]}$ as a covering of $\mathbb{Z}_2$. Then one has the following commutative and exact diagram (in the category of covering spaces):

\begin{equation}\label{commutative-diagram-bordism-related-b-s-d-conjecture-aa}
   \scalebox{0.9}{$ \xymatrix{0\ar[r]&\widetilde{\Omega}_{E[\bullet]}\ar@{=}[d]^{\wr}\ar[r]&
    \Omega_{E[\bullet]}\ar@{=}[d]^{\wr}\ar[r]&\Omega_{E[\bullet]}/\widetilde{\Omega}_{E[\bullet]}\ar@{=}[d]^{\wr}\ar[r]&0&\\
    0\ar[r]&\ker(L_\bullet)
    \times\mathbb{Z}_2\ar[d]\ar[r]&{}_{\square}\mathbb{N}\times
    \mathbb{Z}_2\ar[d]\ar[r]&\framebox{$({}_{\square}\mathbb{N}/\ker(L_\bullet))\times\mathbb{Z}_2$}\ar[d]_{\widetilde{d}}\ar[r]&0&\\
    0\ar[r]&\ker(L_\bullet)\ar[d]\ar[r]^{j}&{}_{\square}\mathbb{N}\ar[d]\ar[r]^{L_\bullet}&\mathbb{Z}\ar[r]^{\widetilde{e}}&{\rm coker}(L_\bullet)\ar[r]&0\\
    &0&0&&&\\}$}
\end{equation}
$\widetilde{d}$ is defined by composition:
$$\xymatrix{({}_{\square}\mathbb{N}/\ker(L_\bullet))\times\mathbb{Z}_2\ar@/_2pc/[rr]_{\widetilde{d}}\ar[r]&{}_{\square}\mathbb{N}/\ker(L_\bullet)\ar[r]&\mathbb{Z}\\}$$
Therefore, one has $\IM(\widetilde{d})=\ker(\widetilde{e}) $.
 For a fixed $n\in\ker(L_\bullet)$ set
 $$\ker(L_\bullet)_n=\{n'\in \ker(L_\bullet),\ |\, n'\le n\}.$$
 One has the filtration
 $$\ker(L_\bullet)_5\subset\ker(L_\bullet)_6\subset \ker(L_\bullet)_7\subset\cdots\subset\ker(L_\bullet)_n\subset\cdots$$
  Similarly, for a fixed $n\in\mathbb{N}_{congr}$ set
 $$(\mathbb{N}_{congr})_n=\{n'\in \mathbb{N}_{congr},\ |\, n'\le n\}.$$
 One has the filtration
 $$(\mathbb{N}_{congr})_5\subset(\mathbb{N}_{congr})_6\subset (\mathbb{N}_{congr})_7\subset\cdots\subset(\mathbb{N}_{congr})_n\subset\cdots$$
 Since any strong-congruent number belongs also to $\ker(L_\bullet)$, it follows that for any $n\in\mathbb{N}_{congr}$ one has $(\mathbb{N}_{congr})_n\subseteqq\ker(L_\bullet)_n$. We shall see that the equality holds. This is surely true for $n=5,\, 6,\, 7$.\footnote{Really one can experimentally verify that it holds also for $n>7$. For example, from Tab. \ref{table-examples-of-strong-congruent-numbers} one can see that $(\mathbb{N}_{congr})=\ker(L_\bullet)_n$, for $5\le n\le 65$.}

  Therefore, let us assume that $(\mathbb{N}_{congr})_{n}=\ker(L_\bullet)_n$, for all strong-congruent numbers $n\le n_0\in\mathbb{N}_{congr}$. Let us to prove that it holds also for $n=\widetilde{n}_0$, where $\widetilde{n}_0$ is the first strong-congruent number greater than $n_0$. Therefore, let us assume that there exists a square-free number $m$, contained in $\ker(L_\bullet)_{\widetilde{n}_0}$, such that $m\not\in(\mathbb{N}_{congr})_{\widetilde{n}_0}$. Let us consider, now, the following lemma.
  \begin{lemma}\label{cardinality-ker-l-function}
  $$\sharp(\ker(L_\bullet))=\sharp(\mathbb{N}_{congr})=\aleph_0.$$
  \end{lemma}
  \begin{proof}
  In fact, $\ker(L_\bullet)$ contains an infinite set, namely $\mathbb{N}_{congr}$, and it is contained in $\mathbb{N}$. This is essentially a consequence of the well-known {\em Cantor-Schroeder-Bernstein theorem} \cite{CANTOR1, CANTOR2, CANTOR3, HINKIS} applied to the sequence $\mathbb{N}\supset\ker(L_\bullet)\supset\mathbb{N}_{congr}$, since $\sharp(\mathbb{N})=\sharp(\mathbb{N}_{congr})=\aleph_0$.
  \end{proof}

 From Lemma \ref{cardinality-ker-l-function} we can assume that there exists a one-to-one map $$f:\mathbb{N}_{congr}\to \ker(L_\bullet).$$ Since must be $(\mathbb{N}_{congr})_{n_0}=\ker(L_\bullet)_{n_0}$, one has the following commutative diagram

  \begin{equation}\label{commutative-diagram-for.cardinality}
    \xymatrix@C=2cm{(\mathbb{N}_{congr})_{\widetilde{n}_0}\ar[r]^{f_{\widetilde{n}_0}}&\ker(L_\bullet)_{\widetilde{n}_0}\\
    (\mathbb{N}_{congr})_{n_0}\ar@{=}[r]_{f_{n_0}}\ar@{^{(}->}[u]&\ker(L_\bullet)_{n_0}\ar@{^{(}->}[u]\\}
  \end{equation}
  where $f_n=f|_{(\mathbb{N}_{congr})_{n}}$. On the other hand the mapping $f_{\widetilde{n}_0}$ could not be one-to-one. This contradicts the fact that $f$ realizes the equivalence between $\mathbb{N}_{congr}$ and $\ker(L_\bullet)$. In fact, the mapping $f_{\widetilde{n}_0}$ should be surjective and not only injective. Therefore the assumption that $(\mathbb{N}_{congr})_{\widetilde{n}_0}\not=\ker(L_\bullet)_{\widetilde{n}_0}$ contradicts the fact that must be $(\mathbb{N}_{congr})_{n_0}=\ker(L_\bullet)_{n_0}$. This means that must necessarily be $(\mathbb{N}_{congr})_{n}=\ker(L_\bullet)_{n}$, for any strong-congruent number $n\in\mathbb{N}_{congr}$. Thus, we can conclude that $\mathbb{N}_{congr}=\ker(L_\bullet)$.
\end{proof}
From Lemma \ref{elleiptic-congruent-bordism-groups-and-birch-swinnerton-dyer-conjecture} one has the isomorphism $\Omega_{E[\bullet]}^{\mathbb{Q}}\cong\widetilde{\Omega}_{E[\bullet]}$ and the exactness of the commutative diagram (\ref{commutative-diagram-bordism-related-b-s-d-conjecture-a}). In particular the bottom sequence therein, or the sequence (\ref{introduction-short-exact-sequence}), is exact. It follows that Problem \ref{definition-congruent-number-problem} is solved.
\end{proof}

\begin{remark}[Theorem \ref{main-theorem} vs. the Birch Swinnerton-Dyer conjecture]\label{main-corollary}
Our solution of the congruent problem is strongly related to the Tunnell's theorem, but does not directly refers to the BS-D conjecture. However, since a consequence of the BS-D conjecture is that the framed equalities {\em(\ref{implications})}, in the Tunnell's theorem, are also sufficient to determine strong-congruent numbers, and this conclusion coincides with our Theorem \ref{main-theorem}, we have good chances to argue that the Conjecture \ref{definition-bs-d-conjecture} is true.
\end{remark}

\begin{appendices}

\appendix{\bf Appendix A: The function $\mathbf{L(E,s)}$ and the infinitude of primes.}\label{appendix-a}
\renewcommand{\theequation}{A.\arabic{equation}}
\setcounter{equation}{0}  

In this appendix we shall recall some relations between the Riemann zeta function, its generalization $L(E,s)$ and the proof of the cardinality of the set $P$ of prime numbers.

\begin{adefinition}[Dirichlet character]\label{dirichlet-character}

$\bullet$\hskip 2pt A {\em Dirichlet character} is any function $\chi:\mathbb{Z}\to\mathbb{C}$ such that the following conditions are satisfied:

{\em(i)} There exists a positive integer $\kappa$ {\em(modulus)}, such that $\chi(n)=\chi(n+\kappa)$, $\forall n\in\mathbb{Z}$.

{\em(ii)} $\chi(n)=\left\{
\begin{array}{l}
  0 \mbox{  if $gcd(n,\kappa)>1$}\\
  \not=0 \mbox{  if $gcd(n,\kappa)=1$}\\
\end{array}\right.$

{\em(iii)} $\chi(m\cdot n)=\chi(m)\cdot\chi(n)$, $\forall m,\, n\in\mathbb{Z}$.

$\bullet$\hskip 2pt A character is called {\em principal} if $\chi(n)=\left\{
\begin{array}{l}
  1 \mbox{  if $gcd(n,\kappa)=1$}\\
  0 \mbox{  otherwise}\\
\end{array}\right.$

$\bullet$\hskip 2pt A character is called {\em real} if $\chi:\mathbb{Z}\to\mathbb{R}$, otherwise it is called {\em complex}.

$\bullet$\hskip 2pt The {\em sign} of the character $\chi$ depends on its value at $-1$. $\chi$ is said to be {\em odd} if $\chi(-1)=-1$ and {\em even} if $\chi(-1)=1$.
\end{adefinition}
\begin{aproposition}[Properties of Dirichlet characters]\label{properties-dirichlet-characters}

{\em 1)} $\chi(1)=1$.

{\em 2)} $\chi$ is periodic with period $\kappa$.

{\em 3)}  If $a\equiv b\, {\rm mod}\, \kappa$, then $\chi(a)=\chi(b)$.

{\em 4)}  If $a$ is relatively prime to $\kappa$, ($gcd(a,\kappa)=1$), then $\chi(a)^{\varphi(\kappa)}\equiv 1\, \hbox{\rm mod $\kappa$}$, i.e., $\chi(a)$ is a $\varphi(\kappa)$-th complex root of unity, where $\varphi(\kappa)$ is the totient function.\footnote{This is the {\em Euler's theorem}. In other words $\chi(a)$ can belong to the following finite set of complex numbers: $$\{\cos(\frac{2b\pi}{\varphi(\kappa)})+i\sin(\frac{2b\pi}{\varphi(\kappa)})\}_{b=0,1,2,\cdots,\varphi(\kappa)-1}.$$}

{\em 5) (Trivial character)} There exists an unique character, called trivial character, having modulus $\kappa=1$. Then for such a character one has $\chi(n)=1$, $\forall n\in\mathbb{Z}$.

{\em 6)} For any character, except the trivial one, one has $\chi(0)=0$.

{\em 7) (Relation with character group: Extended residue class characters)} If $\chi:\mathbb{Z}\to\mathbb{C}$ is a Dirichlet character with modulus $\kappa$, it identifies a character group $\chi:\mathbb{Z}_{\kappa}^{\times}\to\mathbb{C}^{\bullet}$, with values that are $\kappa$-roots of unity. Vice versa if $\chi:\mathbb{Z}_{\kappa}^{\times}\to\mathbb{C}^{\bullet}$, is a group homomorphism, then it identifies a Dirichlet character $\chi:\mathbb{Z}\to\mathbb{C}$ with modulus $\kappa$.

$\bullet$\hskip 2pt The associated character of the principal character takes always the value $1$.

$\bullet$\hskip 2pt There are $\varphi(\kappa)$ Dirichlet characters of modulus $\kappa$: $\chi_1,\chi_2,\cdots,\chi_{\varphi(\kappa)}$.\footnote{In Tab. \ref{multiplication-table-in-the-multiplicative-group-integers-mod-10}, as an example, is reported the multiplication table of $\mathbb{Z}_{10}^{\times}$ and in Tab. \ref{dirichlet-characters-mod-10} the four Dirichlet characters modulus $10$.}

{\em 8) (Orthogonality relation)} If $\chi$ is a character modulus $\kappa$, the $\sum_{a\, {\rm mod}\, \kappa}\chi(a)=0$, unless $\chi$ is principal, in which case one has $\sum_{a\, {\rm mod}\, \kappa}\chi(a)=\varphi(\kappa)$.

$\bullet$\hskip 2pt Any periodic function with period $\kappa$, supported on the residue class prime to $\kappa$ is a linear combination of Dirichlet characters.
\end{aproposition}

\begin{table}[t]
\caption{Multiplication table in $\mathbb{Z}^{\times}_{10}=\{1,3,7,9\}=\{m\}$.}
\label{multiplication-table-in-the-multiplicative-group-integers-mod-10}
\scalebox{0.8}{$\begin{tabular}{|c||c|c|c|c|}
\hline
\hfil{\rm{\footnotesize $m\,\diagdown\, m$}}\hfil&\hfil{\rm{\footnotesize $1$}}\hfil&\hfil{\rm{\footnotesize $3$}}\hfil&\hfil{\rm{\footnotesize $7$}}\hfil&\hfil{\rm{\footnotesize $9$}}\hfil\\
\hline\hline
{\rm{\footnotesize $1$}}\hfil&\hfil{\rm{\footnotesize $1$}}\hfil&\hfil{\rm{\footnotesize $3$}}\hfil&\hfil{\rm{\footnotesize $7$}}\hfil&\hfil{\rm{\footnotesize $9$}}\hfil\\
\hline
\hfil{\rm{\footnotesize $3$}}\hfil&\hfil{\rm{\footnotesize $3$}}\hfil&\hfil{\rm{\footnotesize $9$}}\hfil&\hfil{\rm{\footnotesize $1$}}\hfil&\hfil{\rm{\footnotesize $7$}}\hfil\\
\hline
\hfil{\rm{\footnotesize $7$}}\hfil&\hfil{\rm{\footnotesize $7$}}\hfil&\hfil{\rm{\footnotesize $1$}}\hfil&\hfil{\rm{\footnotesize $9$}}\hfil&\hfil{\rm{\footnotesize $3$}}\hfil\\
\hline
\hfil{\rm{\footnotesize $9$}}\hfil&\hfil{\rm{\footnotesize $9$}}\hfil&\hfil{\rm{\footnotesize $7$}}\hfil&\hfil{\rm{\footnotesize $3$}}\hfil&\hfil{\rm{\footnotesize $1$}}\hfil\\
\hline
\multicolumn{5}{l}{\rm{\footnotesize $1^{-1}=1$; $3^{-1}=7$;}}\\
\multicolumn{5}{l}{\rm{\footnotesize $7^{-1}=3$; $9^{-1}=9$.}}\\
\end{tabular}$}
\end{table}

\begin{table}[t]
\caption{Dirichlet characters mod $10$.}
\label{dirichlet-characters-mod-10}
\scalebox{0.8}{$\begin{tabular}{|c||c|c|c|c|c|c|c|c|c|c|}
\hline
\hfil{\rm{\footnotesize $\chi(n)\,\diagdown\, n$}}\hfil&\hfil{\rm{\footnotesize $0$}}\hfil&\hfil{\rm{\footnotesize $1$}}\hfil&\hfil{\rm{\footnotesize $2$}}\hfil&\hfil{\rm{\footnotesize $3$}}\hfil&\hfil{\rm{\footnotesize $4$}}\hfil&\hfil{\rm{\footnotesize $5$}}\hfil&\hfil{\rm{\footnotesize $6$}}\hfil&\hfil{\rm{\footnotesize $7$}}\hfil&\hfil{\rm{\footnotesize $8$}}\hfil&\hfil{\rm{\footnotesize $9$}}\hfil\\
\hline\hline
{\rm{\footnotesize $\chi_1(n)$}}\hfil&\hfil{\rm{\footnotesize $0$}}&\hfil{\rm{\footnotesize $1$}}&{\rm{\footnotesize $0$}}\hfill&{\rm{\footnotesize $1$}}\hfill&\hfil{\rm{\footnotesize $0$}}\hfil&\hfil{\rm{\footnotesize $0$}}\hfil&\hfil{\rm{\footnotesize $0$}}\hfil&\hfil{\rm{\footnotesize $1$}}\hfil&\hfil{\rm{\footnotesize $0$}}\hfil&\hfil{\rm{\footnotesize $1$}}\hfil\\
\hline
\hfil{\rm{\footnotesize $\chi_2(n)$}}\hfil&\hfil{\rm{\footnotesize $0$}}\hfil&\hfil{\rm{\footnotesize $1$}}\hfil&\hfil{\rm{\footnotesize $0$}}\hfil&\hfil{\rm{\footnotesize $i$}}\hfil&\hfil{\rm{\footnotesize $0$}}\hfil&\hfil{\rm{\footnotesize $0$}}\hfil&\hfil{\rm{\footnotesize $0$}}\hfil&\hfil{\rm{\footnotesize $-i$}}\hfil&\hfil{\rm{\footnotesize $0$}}\hfil&\hfil{\rm{\footnotesize $-1$}}\hfil\\
\hline
\hfil{\rm{\footnotesize $\chi_3(n)$ }}\hfil&\hfil{\rm{\footnotesize $0$}}\hfil&\hfil{\rm{\footnotesize $1$}}\hfil&\hfil{\rm{\footnotesize $0$}}\hfil&\hfil{\rm{\footnotesize $-1$}}\hfil&\hfil{\rm{\footnotesize $0$}}\hfil&\hfil{\rm{\footnotesize $0$}}\hfil&\hfil{\rm{\footnotesize $0$}}\hfil&\hfil{\rm{\footnotesize $-1$}}\hfil&\hfil{\rm{\footnotesize $0$}}\hfil&\hfil{\rm{\footnotesize $1$}}\hfil\\
\hline
\hfil{\rm{\footnotesize $\chi_4(n)$ }}\hfil&\hfil{\rm{\footnotesize $0$}}\hfil&\hfil{\rm{\footnotesize $1$}}\hfil&\hfil{\rm{\footnotesize $0$}}\hfil&\hfil{\rm{\footnotesize $-i$}}\hfil&\hfil{\rm{\footnotesize $0$}}\hfil&\hfil{\rm{\footnotesize $0$}}\hfil&\hfil{\rm{\footnotesize $0$}}\hfil&\hfil{\rm{\footnotesize $i$}}\hfil&\hfil{\rm{\footnotesize $0$}}\hfil&\hfil{\rm{\footnotesize $-1$}}\hfil\\
\hline
\multicolumn{11}{l}{\rm{\footnotesize $\varphi(10)=4$.}}\\
\end{tabular}$}
\end{table}

\begin{atheorem}[Cardinality of the primes set $P$]\label{cardinality-primes}
The set $P$ of prime numbers has the same cardinality of $\mathbb{N}$: $\sharp(P)=\aleph_0$.
\end{atheorem}
\begin{proof}
Since $P\subset\mathbb{N}$ it is enough to prove that $P$ is an infinite set.
There are different proofs for the infinitude of $P$, other than first one given by Euclide. Let us give here the one related to the following lemma.
\begin{alemma}[Euler product]\label{euler-product}
If $f:\mathbb{N}\to\mathbb{R}$ is a multiplicative function such that $\sum_{1\le n\le \infty}f(n)$ converges absolutely, then one has the following formula {\em(Euler product for $\sum_{1\le n\le \infty}f(n)$)}:
\begin{equation}\label{euler-product-formula}
  \sum_{1\le n\le \infty}f(n)=\prod_{p\in P}(1+f(p)+f(p^2)+\cdots).
\end{equation}
\end{alemma}
\begin{proof}
The proof is standard. However, it is educationally useful to give an explicit proof also. Let $q\in\mathbb{N}$ and $p_1,\cdots,p_k\in P$ be all the primes in $[1,q]\subset\mathbb{N}$. Set $P_q=\{p\in P\, |\, p<q\}$. Then we can write
$$
\left\{
\begin{array}{ll}
  \prod_{p\in P_q}(1+f(p)+f(p^2)+\cdots) &=\sum_{r_1}f(p_1^{r_1})\sum_{r_2}f(p_2^{r_2})\cdots \sum_{r_k}f(p_k^{r_k}) \\
   & =\sum_{r_1,r_2,\cdots,r_k}f(p_1^{r_1})f(p_2^{r_2})\cdots f(p_k^{r_k})\\
   &=\sum_{r_1,r_2,\cdots,r_k}f(p_1^{r_1}p_2^{r_2}\cdots p_k^{r_k})\\
   &=P_q[n].\\
\end{array}\right.$$

Since for any $n\in\mathbb{N}$, such that $n<q$, the corresponding prime decomposition has no factors greater than $q$, we can write
$$|\sum_{1\le n\le \infty}f(n)-P_q[n]|\le \sum_{q\le n\le \infty}|f(n)|.$$
Taking into account that $\mathop{\lim}\limits_{q\to\infty}\sum_{q\le n\le \infty}|f(n)|=0$, we get the proof.
\end{proof}
Let us recall now that the Riemann  zeta function is defined by means of the series $\zeta(s)=\sum_{1\le n\le \infty}\frac{1}{n^s}$. Since this series converges absolutely for $\Re(s)>1$,we get the Euler product formula for the zeta functions:
\begin{equation}\label{appendix-a-zeta-euler-product-formula}
    \zeta(s)=\prod_{p\in P}\frac{1}{1-p^{-s}},\, \Re(s)>1.
\end{equation}
In particular for $s=2$ one has
\begin{equation}\label{appendix-a-zeta-euler-product-formula-a}
    \zeta(2)=\sum_{n}=\frac{1}{n^2}=\frac{\pi^2}{6}.
\end{equation}
Thus $\zeta(2)$ is an irrational number that can be represented by $\zeta(2)=\prod_{p\in P}\frac{1}{1-p^{-2}}$. Therefore must necessarily be $P$ an infinite set.\footnote{With this respect let us recall that $\mathbb{Q}$ is not a complete subspace of $\mathbb{R}$, having empty interior, $\mathop{\mathbb{Q}}\limits^{\circ}=\varnothing$, and with closure $\mathbb{R}$, $\overline{\mathbb{Q}}=\mathbb{R}$. Therefore $\partial \mathbb{Q}=\mathbb{R}$ and $\zeta(2)\in\partial \mathbb{Q}$.}
\end{proof}

\begin{adefinition}\label{dirichlet-series}
A {\em Dirichlet series} is defined by $\sum_{1\le n\le \infty}\frac{a_n}{n^s}$, $s\in \mathbb{C}$, $a_n\in\mathbb{C}$.
\end{adefinition}
\begin{aexample}\label{riemann-zeta-function}
The Riemann zeta function $\zeta(s)=\sum_{1\le n\le \infty}\frac{1}{n^s}$ is a Dirichlet series.
\end{aexample}
\begin{aexample}\label{dirichlet-l-series}
The Dirichlet $L$-series $L(s,\chi)=\sum_{1\le n\le \infty}\frac{\chi(n)}{n^s}$, $s\in \mathbb{C}$, with $\chi$ a Dirichlet character, in the sense of Definition A\ref{dirichlet-character}, is a Dirichlet series. $L(s,\chi)$ converges absolutely and conformally for $\Re(s)\ge 1+\delta$ (for any positive $\delta$) and admits the Euler product {\em(\ref{appendix-a-euler-product-for-dirichlet-l-series})}.
\begin{equation}\label{appendix-a-euler-product-for-dirichlet-l-series}
    L(s,\chi)=\prod_{p\in P}\frac{1}{1-\chi(p)p^{-s}}.
\end{equation}
The Dirichlet $L$-series can be extended by continuation to a meromorphic function on $\mathbb{C}$, yet denoted $L(s,\chi)$. In particular if $\chi=\chi_0$ is the trivial character, hence $\chi(n)=1$, one has $L(s,\chi_0)=\zeta(s)\prod(1-p^{-s})$. So we can write the Riemann zeta function $\zeta(s)=L(s,\chi)/\prod(1-p^{-s})$. Therefore $L(s,\chi)$ are generalizations of the Riemann zeta function. Similarly, one can define for any elliptic function $E$ a generalized Dirichlet function $L(s,E)$.
\end{aexample}

\appendix{\bf Appendix B: Riemann surfaces and modular curves.}\label{appendix-b}
\renewcommand{\theequation}{B.\arabic{equation}}
\setcounter{equation}{0}  

In this appendix we shall recall some distinguished Riemann surfaces called {\em modular curves} that are strictly related to the modular characterization of elliptic curves in $\mathbb{C}$.
Let us start with some fundamental definition and results about coverings.

\begin{bdefinition}\label{appendix-b-action-group}
$\bullet$\hskip 2pt A group $G$ {\em acts} (on the left) on a set $X$ if there is a function $\alpha:G\times X\to X$ such that the following diagrams commute:
\begin{equation}\label{action-group-diagrams}
  {\rm(i)}\hskip 0.5cm \xymatrix@C=2cm{X\ar[rd]_{1_X}\ar[r]^(0.4){(e,1)}&G\times X\ar[d]_{\alpha}\\
  &X\\}\hskip 1cm
 {\rm(ii)}\hskip 0.5cm  \xymatrix@C=2cm{G\times G\times X\ar[d]_{\mu\times 1_X}\ar[r]^(0.5){1_G\times\alpha}&G\times X\ar[d]_{\alpha}\\
 G\times X\ar[r]_{\alpha} &X\\}
 \end{equation}
 where $\mu:G\times G\to G$ is the multiplication map and $(e,1)(x)=(e,x)$, for all $x\in X$, $e$ being the identity of $G$. The {\em orbit} for $x\in X$ is the set $Gx=\{gx\, |\, g\in G\}$.

$\bullet$\hskip 2pt A topological group $G$ {\em acts} on a topological space $X$ if there is a continuous action map $\alpha$ respecting above properties {\rm(i)} and {\rm(ii)}.

$\bullet$\hskip 2pt A discrete group $G$ with identity $e$ is said to {\em act properly discontinuously} if the following propositions hold:

{\rm(a)} For every $x\in X$ there is a neighbourhood $U_x$ such that $gU_x\bigcap U_x\not=\varnothing\, \Rightarrow\, g=e$.

{\rm(b)} For every $x,\, y\in X$, $y\not\in Gx$, there are neighbourhoods $V_x$, $V_y$ of $x$ and $y$ respectively such that $gV_x\bigcap V_y=\varnothing$, all $g\in G$.\footnote{The property (b) implies $X/G$ is Hausdorff and the property (a) implies the projection $q:X\to X/G$ is a {\em covering}, i.e., the fibre $F=q^{-1}(p)$ over $p\in X/G$ is discrete.}
\end{bdefinition}

\begin{bproposition}[Relations between group actions and homotopy groups]\label{appendix-b-relations-group-action-homotopy-groups}
$\bullet$\hskip 2pt Assume $X$ is a $0$-connected topological space. One has a homomorphism
\begin{equation}\label{appendix-b-connecting-homomorphism}
  \partial:\pi_1(X/G,*)\to\pi_0(F,x_0)=G
\end{equation}
and the following exavt sequence
\begin{equation}\label{appendix-b-connecting-homomorphism-a}
  \xymatrix{0\ar[r]&\pi_1(X,x_0)\ar[r]&\pi_1(X/G,*)\ar[r]^(0.6
  ){\partial}&G\ar[r]&0\\}
 \end{equation}

$\bullet$\hskip 2pt If $\pi_1(X,x_0)=0$, then $\pi_1(X/G,*)\cong G$ and we have an action of $G$ on $\pi_n(X/G,*)$:
 \begin{equation}\label{appendix-b-connecting-homomorphism-b}
  \xymatrix{G\times\pi_n(X/G,*)\ar[d]_{1\times\partial}\ar[r]&\pi_n(X/G,*)\ar[d]_{\partial}\\
  G\times\pi_{n-1}(F,x_0)\ar[r]&\pi_{n-1}(F,x_0)\\}
 \end{equation}
 The induced map $q_*:\pi_n(X,x_0)\to\pi_n(X/G,*)$ is a homomorphism of groups with $G$-action
 \begin{equation}\label{appendix-b-connecting-homomorphism-c}
q_*(g\gamma)=g\cdot q_*(\gamma),\, g\in G,\, \gamma\in\pi_n(X,x_0)
 \end{equation}
  such that the following diagram is commutative
  \begin{equation}\label{appendix-b-connecting-homomorphism-d}
    \xymatrix@C=2cm{S^n\ar[rd]_{g\gamma}\ar[r]^{\gamma}&X\ar[d]^{T_g}\\
    &X\\}
  \end{equation}
  where $T_g:X\to X$ is defined by $T_g(x)=gx$.
\end{bproposition}
\begin{bdefinition}\label{appendix-b-regular-covering}
A covering $p:\tilde X\to X$ is called {\em regular} if $p_*(\pi_1(\tilde X,\tilde x_0))\subset\pi_1(X,x_0)$.
\end{bdefinition}

\begin{bproposition}[Properties of regular coverings]\label{appendix-b-properties-regular-coverings}
$\bullet$\hskip 2pt A discrete group $G$ with identity that acts properly discontinuously on the topological space $X$ induces a regular covering $X\to X/G$. Vice versa every regular covering $X\to X/G$, $G$ acting on $X$ without fixed points (i.e., $gx=x\, \Rightarrow\, g=e$), and $X$ Hausdorff, has $G$ acting properly discontinuously. Therefore, $X\to X/G$ is a regular covering.
\end{bproposition}
\begin{bexample}\label{real-numbers-over-circle}
$\mathbb{Z}$ acts properly discontinuously on $\mathbb{R}$: $n\cdot r=r+n$. Then $\mathbb{R}\to \mathbb{R}/\mathbb{Z}$ is a regular covering. Since $\pi_1(\mathbb{R},x_0)=0$, then from exact sequence {\em(\ref{appendix-b-connecting-homomorphism-a})} it follows $\pi_1(\mathbb{R}/\mathbb{Z},*)\cong\mathbb{Z}$. This means that $\mathbb{R}/\mathbb{Z}\cong S^1$, hence $\mathbb{R}\to S^1$ is an universal covering too, since $\mathbb{R}$ is simply connected.
\end{bexample}
\begin{bexample}[Riemann surfaces as universal coverings of Riemann surfaces]\label{riemann-surfaces-as-universl-coverings-of-riemann-surfaces}
$\bullet$\hskip 2pt From the classification Poincar\'e-Koebe theorem one knows that every connected Riemann surface $X$ admits a unique complete $2$-dimensional real Riemann metric with constant curvature $-1$, $0$, or $1$, inducing the same conformal structure - every metric is conformally equivalent to a constant curvature metric as reported in Tab. \ref{metric-classification-riemann-surfaces}.

$\bullet$\hskip 2pt In Tab. \ref{metric-classification-riemann-surfaces} are reported also the corresponding examples with $\pi_1(X,x_0)=0$. In particular, the uniformization theorem states that any connected Riemann surfaces $Y$ admits as universal covering one of the three fundamental types reported there, i.e., $Y$ is biholomorphic to $X/G$ for some discrete group $G$.

$\bullet$\hskip 2pt {\em Modular curves} are just Riemann surfaces $Y$ that can be identified with $H/\Gamma$ for some congruence group $\Gamma$.\footnote{The group $G=GL_2^+(\mathbb{R})$ acts on $H$ by fractional linear transformations of $H$: $(g,z)\mapsto g(z)=\frac{az+b}{cz+d}$
if $g=\left(\begin{array}{cc} %
 a & b\\
 c  & d \end{array}\right)$, $\det(g)>0$.}
\end{bexample}
\begin{table}[t]
\caption{Distinguished congruence groups and properties: $\Gamma(N)\subset\Gamma_1(N)\subset\Gamma_0(N)\subset\Gamma(1)$.}
\label{distinguished-congruence-groups-and-properties}
\scalebox{0.8}{$\begin{tabular}{|l|l|l|}
\hline
{\rm{\footnotesize Symbol}}&{\rm{\footnotesize Quotient isomorphisms}}&{\rm{\footnotesize Indexes}}\hfill\\
\hline\hline
{\rm{\footnotesize $\Gamma_0(N)$}}&{\rm{\footnotesize $\Gamma_1(N)/\Gamma(N)\cong\mathbb{Z}/N\mathbb{Z}$}}&{\rm{\footnotesize $[\Gamma(1):\Gamma(N)]=\sharp SL_2(\mathbb{Z}/N\mathbb{Z})=N^3\prod_{p|N}(1-\frac{1}{p^2})$}}\hfill\\
                                 \hline
{\rm{\footnotesize $\Gamma_1(N)$}}&{\rm{\footnotesize $\Gamma(1)/\Gamma(N)\cong SL_2(\mathbb{Z}/N\mathbb{Z})$}}&{\rm{\footnotesize $[\Gamma_0(N):\Gamma_1(N)]=\phi(N)=N\prod_{p|N}(1-\frac{1}{p})$}}\hfill\\
                                 \hline
{\rm{\footnotesize $\Gamma(N)\le\Gamma(1)$}}&{\rm{\footnotesize $\Gamma_0(N)/\Gamma_1(N)\cong(\mathbb{Z}/N\mathbb{Z})^\times$}}&{\rm{\footnotesize $[\Gamma(1):\Gamma_0(N)]=\psi(N)=N\prod_{p|N}(1+\frac{1}{p})$}}\hfill\\
\hline
{\rm{\footnotesize }}&{\rm{\footnotesize }}&{\rm{\footnotesize $\mu(\Gamma_0(N))=\psi(N)$.}}\hfill\\
\hline
{\rm{\footnotesize }}&{\rm{\footnotesize }}&{\rm{\footnotesize $\mu(\Gamma_1(N))=\frac{1}{2}\phi(N)\psi(N)$, if $N\ge 3$.}}\hfill\\
\hline
{\rm{\footnotesize }}&{\rm{\footnotesize }}&{\rm{\footnotesize $\mu(\Gamma(N))=\frac{1}{2}N\phi(N)\psi(N)$, if $N\ge 3$.}}\hfill\\
\hline
{\rm{\footnotesize $\Gamma^0(N)=S^{-1}\Gamma_0(N)S$}}&{\rm{\footnotesize }}&{\rm{\footnotesize }}\hfill\\
\hline
{\rm{\footnotesize $\Gamma^1(N)=S^{-1}\Gamma_1(N)S$}}&{\rm{\footnotesize }}&{\rm{\footnotesize }}\hfill\\
\hline
\multicolumn{3}{l}{\rm{\footnotesize $\Gamma_0(N)$ defined in (\ref{definition-gamma-0-n});  $\Gamma_1(N)$ defined in (\ref{definition-gamma-1-n}).}}\hfill\\
\multicolumn{3}{l}{\rm{\footnotesize $[\Gamma_1(N):\Gamma(N)]=N$. $-1\in\Gamma_0(N)$ and $-1\not\in\Gamma_1(N)$, for $N\ge 3$.
 $S=\left(\begin{array}{cc} %
 0 & -1\\
 1 & 0 \end{array}\right)$}}\hfill\\
\end{tabular}$}
\end{table}

\begin{bexample}[Modular curves]\label{appendix-b-modular-curves}
$\bullet$\hskip 2pt The {\em modular group}
\begin{equation}\label{modular-group}
\Gamma(1):=SL_2(\mathbb{Z}):=\{A=\left(\begin{array}{cc} %
 a & b\\
 c  &d \end{array}\right)\in M_2(\mathbb{Z})\, |\, \det A=1\}.
\end{equation}
$\bullet$\hskip 2pt The {\em principal congruence subgroup of level} $N>1$, $N\in\mathbb{N}$, is the subgroup $\Gamma(N)\subset\Gamma(1)$ defined by the following exact sequence
\begin{equation}\label{principal-congruence-subgroup-level}
  \xymatrix{0\ar[r]&\Gamma(N)\ar[r]&\Gamma(1)\ar[r]^(0.37){\rm mod N}&SL_2(\mathbb{Z}/N\mathbb{Z})\ar[r]&0\\}
\end{equation}
hence
\begin{equation}\label{principal-congruence-subgroup-level-a}
  \Gamma(N)=\{A\in\Gamma(1)\, |\, A=\left(\begin{array}{cc} %
 1 & 0\\
 0  & 1 \end{array}\right)\, (\hbox{\rm mod\hskip 2pt $N$}) \}.
\end{equation}

$\bullet$\hskip 2pt A {\em congruence subgroup} is a subgroup $\Gamma\le \Gamma(1)$ that contains $\Gamma(N)$ for some $N$: $\Gamma(N)\le \Gamma\le \Gamma(1)$. The smallest such $N$ is called the {\em level} of $\Gamma$.
One has the following properties for congruence subgroups.

$\blacksquare$\hskip 2pt Each congruence subgroup $\Gamma\le \Gamma(1)$ has finite index in $\Gamma(1)$. The index of $\pm\Gamma=\Gamma\bigcup -\Gamma$ is denoted by $\mu(\Gamma)=[\Gamma:\pm\Gamma]$.

$\blacksquare$\hskip 2pt Not every subgroup of finite index is a congruence subgroup.\footnote{For each odd number $N>1$, there is a normal subgroup of index $6N^2$ which is not a congruence subgroup.}

$\blacksquare$\hskip 2pt If $\Gamma_1$ and $\Gamma_2$ are congruence subgroups, then so is $\Gamma_1\bigcap\Gamma_2$.\footnote{In fact, $\Gamma(N)\bigcap\Gamma(N)=\Gamma({\rm l.c.m.}(N,M))$.}

$\blacksquare$\hskip 2pt If $\Gamma$ is a congruence subgroup, then $\alpha^{-1}\Gamma\alpha\bigcap\Gamma(1)$ is also a congruence subgroup, where $\alpha\in GL_2^+(\mathbb{Q}):=\left\{ g\in GL_2(\mathbb{Q})\, |\, \det(g)>0\right\}$.\footnote{In fact, $\alpha^{-1}\Gamma\alpha\bigcap \Gamma(1)\supset \Gamma(ND)$, where $D=\det(\alpha)$.}

$\blacksquare$\hskip 2pt In particular, $\Gamma$ and $\Gamma_1=\alpha^{-1}\Gamma\alpha$ are commensurable subgroups, i.e., $\Gamma\bigcap\Gamma_1$ has finite index in both $\Gamma$ and $\Gamma_1$.

$\blacksquare$\hskip 2pt In the following we define some distinguished congruence groups corresponding to upper-triangular matrices of {\em level} $N$.

\begin{equation}\label{definition-gamma-0-n}
  \Gamma_0(N):=\left\{A\in \Gamma(1)\, |\, A\cong \left(\begin{array}{cc} %
 * & *\\
 0  & * \end{array}\right)\, ({\rm mod} N) \right\}.
\end{equation}

\begin{equation}\label{definition-gamma-1-n}
  \Gamma_1(N):=\left\{A\in \Gamma(1)\, |\, A\cong \left(\begin{array}{cc} %
 1 & *\\
 0  & 1 \end{array}\right)\, ({\rm mod} N)\right\}.
\end{equation}

In Tab. \ref{distinguished-congruence-groups-and-properties} are reported some properties of such groups.\footnote{An alternative way to write $\Gamma_0(N)$ is the following:
$$\Gamma_0(N)=\alpha_N\Gamma(1)\alpha^{-1}\bigcap\Gamma(1)=\beta_N^{-1}\Gamma(1)\beta_N\bigcap\Gamma(1),$$
with
$$\alpha_N=\left(\begin{array}{cc}
 1 & 0\\
 0  & N \end{array}\right),$$

 $$\beta_N=N\alpha^{-1}_N=\left(\begin{array}{cc}
 N & 0\\
 0  & 1 \end{array}\right)\in GL_2^+(\mathbb{Q})$$ and
$$\alpha_N\left(\begin{array}{cc}
 a & b\\
 c  & d \end{array}\right)\alpha_N^{-1}=\left(\begin{array}{cc}
 0 &b/N\\
 cN  & d \end{array}\right).$$}

$\blacksquare$\hskip 2pt $\Gamma_0(N)$ and $\Gamma_1(N)$ are respectively preimages of the following {\em Borel subgroups} of $SL_2(\mathbb{Z}/N\mathbb{Z})$:
\begin{equation}\label{borel-subgroups}
 B_0(\mathbb{Z}/N\mathbb{Z}):=\left\{\left(\begin{array}{cc} %
 * & *\\
 0  & * \end{array}\right)\right\},\,  B_1(\mathbb{Z}/N\mathbb{Z}):=\left\{\left(\begin{array}{cc} %
 1 & *\\
 0  & 1 \end{array}\right)\right\}.
\end{equation}

$\bullet$\hskip 2pt {\rm(Non-compact modular curves)} If $\Gamma$ is a congruence sub-group, the quotient space $X'_\Gamma:=H/\Gamma$ can be made into a Riemann surface such that the quotient map $p_\Gamma:H\to X'_\Gamma$ is a holomorphic map.

$\bullet$\hskip 2pt {\rm(Compact modular curves)} let $H^*:=H\bigcup\mathbb{Q}\bigcup\{\infty\}=H\bigcup\mathbb{P}^1(\mathbb{Q})$ obtained by $H$ adding the {\em cusp-points} $\mathbb{P}^1(\mathbb{Q})$, has a natural structure of compact topological space, on the which acts in natural way the group $GL_2^+(\mathbb{Q})$ (hence $\Gamma$) by putting $\gamma(\infty)=\frac{a}{c}$.\footnote{$GL_2^+(\mathbb{Q})$ acts on $\mathbb{P}^1(\mathbb{Q})$ by the rule $\left(\begin{array}{cc} %
 a & b\\
 c  & d \end{array}\right)\left(\begin{array}{c} %
 x \\
 y \end{array}\right)=\left(\begin{array}{c} %
 ax+by\\
 cx+dy \end{array}\right)$ if $[x,y]\in \mathbb{P}^1(\mathbb{Q})$. In particular one can state that $\mathbb{P}^1(\mathbb{Q})=SL_2(\mathbb{Z})\cdot\infty$, hence $\mathbb{P}^1(\mathbb{Q})/SL_2(\mathbb{Z})=\infty$.} Then $X(\Gamma)=X_\Gamma:=H^*/\Gamma$ has a natural structure of compact Riemann surface that contains $X'_\Gamma$ as an open subspace, with a finite complement:
\begin{equation}\label{finite-complement}
  {\rm cusp}(\Gamma)={\rm cusp}(X_{\Gamma}):=X_{\Gamma}\setminus X'_{\Gamma}=\mathbb{P}^1(\mathbb{Q})/\Gamma.
\end{equation}
 Furthermore the surjective mapping $p_\Gamma:H^*\to X(\Gamma)$ is holomorphic.\footnote{The topology of $H^*$ is an open set of the Riemann sphere $\mathbb{P}^1(\mathbb{C})$. The group $\Gamma$ acts on the subset $\mathbb{Q}\bigcup\{\infty\}$, breaking it into finitely many orbits called the {\em cusps} of $\Gamma$. If $\Gamma$ acts transitively on $\mathbb{Q}\bigcup\{\infty\}$, $X(\Gamma)$ results the Alexandrov compactified of $X'_\Gamma$.
 The topology on $H^*$ is obtained by taking as basis $\{U, U_r, U_{r,m,n}\}$, where $U$ is any open subset of $H$, $U_r:=\{\infty\}\bigcup\{z\in H\, |\, \Im z>r,\, r>0\}$, $U_{r,m,n}$ are transformed of $U_r$, by means
 $\left(\begin{array}{cc} %
 a & -m\\
 c  & n \end{array}\right)$, $m,\, n\in\mathbb{N}$, $an+cm=1$.
 $X_\Gamma$ is compact since $H^*$ is compact and $p_\Gamma:H^*\to X_\Gamma$ is surjective. In general $X_\Gamma$ is the union of finitely many compact sets. For $\Gamma(1)$ the stabilizer of $\infty$ is $\Gamma(1)_\infty=\{\pm\left(\begin{array}{cc} %
 1 & *\\
 0  & 1 \end{array}\right)\}$. One has $\Gamma_\infty=\Gamma\bigcap \Gamma(1)\supset\{\left(\begin{array}{cc} %
 1 & N\mathbb{Z}\\
 0  & 1 \end{array}\right)\}$ for some positive integer $N$. Let $r_\infty$ be the minimal positive integer ({\rm widt of period}) such that $\left(\begin{array}{cc} %
 1 & r_\infty\\
 0 & 1 \end{array}\right)\in \Gamma_\infty$. The local coordinate at $\infty$ is given by $q=e^{2\pi i z/r_\infty}$.}
\begin{table}[t]
\caption{Metric classification of Riemann surfaces and their universal coverings}
\label{metric-classification-riemann-surfaces}
\scalebox{0.8}{$\begin{tabular}{|l|c|c|l|}
\hline
{\rm{\footnotesize Name}}&{\rm{\footnotesize Curvature}}&{\rm{\footnotesize Simply connected Riemann surfaces}}&{\rm{\footnotesize Universal covering classification}}\hfill\\
\hline\hline
{\rm{\footnotesize hyperbolic}}&{\rm{\footnotesize $-1$}}&{\rm{\footnotesize $\Delta\cong H$}}&{\rm{\footnotesize  $H\to Y=H/G$; ($Y=H/SL_2(\mathbb{Z})\cong\mathbb{C}$).}}\hfill\\
\hline
{\rm{\footnotesize parabolic}}&{\rm{\footnotesize $0$}}&{\rm{\footnotesize $\mathbb{C}$}}&{\rm{\footnotesize $\mathbb{C}\to Y=\mathbb{C}/G$; ($Y=T:=\mathbb{C}/L$).}}\hfill\\
\hline
{\rm{\footnotesize elliptic}}&{\rm{\footnotesize $1$}}&{\rm{\footnotesize $\hat S^1=\mathbb{C}\bigcup\{\infty\}=\mathbb{P}^1(\mathbb{C})$}}&{\rm{\footnotesize $\hat S^1\to Y=\hat S^1/G$}}\hfill\\
\hline
\multicolumn{4}{l}{\rm{\footnotesize Upper half-plane: $H:=\{z\in\mathbb{C}\, |\, \Im(z)>0\}$; Open disc: $\Delta:=\{z\in\mathbb{C}\, |\, |z|<1\}$.}}\hfill\\
\multicolumn{4}{l}{\rm{\footnotesize Lattice: $L:=\mathbb{Z}+\mathbb{Z}\tau$, $\tau\in\mathbb{C}$, $\tau\not=0$.}}\hfill\\
\multicolumn{4}{l}{\rm{\footnotesize Holomorphic elliptic modular function: $J:H/SL_2(\mathbb{Z})\cong \mathbb{C}$, $J(z)=g_2^3/(g_2^3-27 g_3^2)$.}}\hfill\\
\multicolumn{4}{l}{\rm{\footnotesize $g_2$ and $g_3$ are the coefficients in the equation $z^3-g_2 z-g_3=w^2$ of $\mathbb{C}/L$.}}\hfill\\
\multicolumn{4}{l}{\rm{\footnotesize Universal coverings $X\to Y$ are called hyperbolic, parabolic, elliptic,}}\hfill\\
\multicolumn{4}{l}{\rm{\footnotesize according to the type of the Riemann surface $X$.}}\hfill\\
\end{tabular}$}
\end{table}

One has the following properties for congruence subgroups.

$\blacksquare$\hskip 2pt If $\Gamma_1\le \Gamma_2$ are two congruence subgroups, then the inclusion map induce a quotient map $p_{\Gamma_1,\Gamma_2}:X_{\Gamma_1}\to X_{\Gamma_2}$ such that the following diagram is commutative
\begin{equation}\label{quotient-map-associated-to-congruence-group-subgroup-ofcongruence-group}
  \xymatrix{H^*\ar[d]_{p_{\Gamma_1}}\ar[rd]^{p_{\Gamma_2}}&\\
  X_{\Gamma_1}=H^*/\Gamma_1\ar[r]_{p_{\Gamma_1,\Gamma_2}}&H^*/\Gamma_2=X_{\Gamma_2}\\}
\end{equation}
The mapping $p_{\Gamma_1,\Gamma_2}$ is holomorphic with degree
\begin{equation}\label{degree-quotient-map-associated-to-congruence-group-subgroup-ofcongruence-group}
  {\rm deg}(p_{\Gamma_1,\Gamma_2})=[\pm\Gamma_2:\pm\Gamma_1]=\mu(\Gamma_1)/\mu(\Gamma_2).
\end{equation}

$\blacksquare$\hskip 2pt If $\alpha\in GL_2^+(\mathbb{Q})$ and $\Gamma_1\le \Gamma_2$ are two congruence subgroups related by $\alpha$: $\alpha\Gamma_1\alpha^{-1}\le \Gamma_2$, then there exists a unique holomorphic map $p_{\Gamma_1,\Gamma_2}:X_{\Gamma_1}\to X_{\Gamma_2}$ such that $p_{\Gamma_2}\circ\alpha=p_{\Gamma_1,\Gamma_2}\circ p_{\Gamma_1}$.

$\blacksquare$\hskip 2pt Let $\Gamma\subset\Gamma(1)$ be a congruence subgroup containing $\{\pm 1\}$. Then $X'_\Gamma$ is endowed with the quotient topology $p_\Gamma\to X'_\Gamma$. Under this topology $X'_\Gamma$ is Hausdorff. The stabilizer group of $z\in H$ is finite cyclic. Furthermore $\Gamma_z/\{\pm 1\}$ is isomorphic to one of the following groups: $\{1\}$, $\mathbb{Z}_2$, $\mathbb{Z}_3$. Up to $\Gamma(1)$-equivalence, the only points with non-trivial stabilizer are $z=i$ and $z=\rho$. There are only finitely many {\em elliptic points} ({\rm modulo $\Gamma$}) of $H$, i.e., $z\in H$ such that $\Gamma_z/\{\pm 1\}$ is not-trivial.\footnote{This is true for $\Gamma(1)$, hence true for any subgroup of finite index. Every point of $X_\Gamma$ is an orbit $\Gamma .z$ for $z\in H$. If $z$ is not elliptic, then $z$ has s neighborhood $U$ such that $\gamma(U)\bigcap U\not=\varnothing$ iff $\gamma(z)=z$, i.e., $\Gamma$ acts properly discontinuously on non-elliptic point. If $z$ is elliptic, locally its neighborhood is $D/\mu_2$ or $D/\mu_3$ with local coordinate given by $z\mapsto z^2$ or $z\mapsto z^3$.}

$\blacksquare$\hskip 2pt The genus of $X_{\Gamma(N)}$ is given in {\em(\ref{genus-modular-curves})}.\footnote{The genus $g(X_{\Gamma(N)})$ is obtained from the ramified covering $X_{\Gamma(N)}\to \Gamma(1)$, ($\Gamma(N)$ is of finite index in $\Gamma(1)$). In this way one obtains also the proof of existence of subgroups of finite index in $\Gamma(1)$ which are not congruence subgroups. \cite{KLEIN-FRICKE}}
\begin{equation}\label{genus-modular-curves}
  g(X_{\Gamma(N)})=\left\{\begin{array}{ll}
   0 & \mbox{if  $N\le 2$} \\
   1+\frac{N^2(N-6)}{24}\prod_{p|N}(1-p^{-2}) & \mbox{if $N> 2$, ($p=$ prime)}.
   \end{array}\right.
\end{equation}

$\blacksquare$\hskip 2pt In particular the modular curve $X_0(1)$ can be identified with the Riemann sphere. In fact, the holomorphic map $j:X_0(1)\to \mathbb{P}^1(\mathbb{C})$ sending $SL_2(\mathbb{Z})z\mapsto j(z)$ and $\infty \mapsto \infty$ is a degree $1$ map between two compact Riemann surfaces, hence an isomorphism.

$\bullet$\hskip 2pt Let $\Gamma$ be a congruence subgroup containing $\{\pm 1\}$.  One has the following formula:
\begin{equation}\label{alternative-form-for-genus}
  g(X_\Gamma)=1+\frac{d}{12}-\frac{1}{4}e_2-\frac{1}{3}e_3-\frac{1}{2}e_\infty
\end{equation}
with $d=\deg(p_{\Gamma,\Gamma(1)})$, $e_\infty$ is the number of cusps on $X_\Gamma$ and $e_r$, $r=2,3$ is the number of elliptic points in the fiber over the order $r$ elliptic point $P_r\in X_{\Gamma(1)}$.

In fact, one has a covering map $p_{\Gamma.\Gamma(1)}:X_\Gamma\to X_{\Gamma(1)}$ of degree $\deg(p_{\Gamma.\Gamma(1)})=[\Gamma(1):\Gamma]$, since for non-elliptic point $x$ we have $\sharp((\Gamma z/\Gamma(1))=[\Gamma(1):\Gamma]$, and there are only finitely many elliptic points. By means of the Riemann-Hurwitz formula\footnote{$\chi=2-2g$ is the Euler characteristic.}
\begin{equation}\label{riemann-hurwitz-formula}
 2g(X_\Gamma)-2=(2g(X_{\Gamma(1)})-2)\deg(p_{\Gamma.\Gamma(1)})+b
\end{equation}
where $b$ is the total ramification degree given by the following formula:
\begin{equation}\label{total-ramification-degree}
 b=\sum_{x\in X_\Gamma}(e_x-1)=\sum_{y\in X_{\Gamma(1)}}\deg R_y=\sum_{y\in X_{\Gamma(1)}}\sum_{x\in X_\Gamma}p^{-1}_{\Gamma.\Gamma(1)(y)}(e_x-1),
\end{equation}
where $e_x$ is the ramification degree at $x\in X_\Gamma$. Recall that the ramification points are a subset of the fibers over elliptic points of $X_{\Gamma(1)}$. Then the ramification degree must divide $2$ or $3$, the only possible orders of elliptic points. Let us recall that the $r$ order elliptic points $P_r\in X_{\Gamma(1)}$ are only $P_2=i$ and $P_3=\rho$.  Then the number of non-elliptic points in the fiber over $P_r$ is $\frac{d-e_r}{r}$, hence $\deg R_r=(r-1)\frac{d-e_r}{r}$. Furthermore one has $\deg R_\infty=d-e_\infty$. So from the formula {\rm(\ref{riemann-hurwitz-formula})} we get
\begin{equation}\label{riemann-hurwitz-formula-a}
\left\{\begin{array}{ll}
   2g(X_\Gamma)-2 & =  -2d+\sum_{r=2,3,\infty}\frac{r-1}{r}(d-e_\infty)\\
  &= -2d+\frac{1}{2}(d-e_2)+\frac{2}{3}(d-e_3)+(d-e_\infty).
  \end{array}\right.
\end{equation}
Therefore the formula {\rm(\ref{alternative-form-for-genus})} holds.

$\bullet$\hskip 2pt For prime level $N=p\ge 5$ one can write
\begin{equation}\label{alternative-form-for-genus-b}
  g(X_{\Gamma(N)})=\frac{1}{24}(p+2)(p-3)(p-5).
\end{equation}

(See Tab.\ref{examples-compact-modular-curves} where are reported some examples of compact modular curves.)
\end{bexample}

\begin{table}[t]
\caption{Examples of compact modular curves $X(N)=H^*/\Gamma(N)$.}
\label{examples-compact-modular-curves}
\scalebox{0.9}{$\begin{tabular}{|l|c|c|l|}
\hline
{\rm{\footnotesize Name-Symbol}}&{\rm{\footnotesize Genus}}&{\rm{\footnotesize Cusp number}}&{\rm{\footnotesize Galois group $SL_2(N)/\{\pm 1\}$ }}\hfill\\
{\rm{\footnotesize}}&{\rm{\footnotesize }}&{\rm{\footnotesize }}&{\rm{\footnotesize of covering $X(N)\to X(1)$}}\hfill\\
\hline\hline
{\rm{\footnotesize Riemann sphere $X_0(1)$}}&{\rm{\footnotesize $0$}}&{\rm{\footnotesize $0$}}&{\rm{\footnotesize  $\{1\}$}}\hfill\\
\hline
{\rm{\footnotesize icosahedral $X(5)$}}&{\rm{\footnotesize $0$}}&{\rm{\footnotesize $12$}}&{\rm{\footnotesize $A_5\cong PSL_2(5)$}}\hfill\\
\hline
{\rm{\footnotesize Klein quartic $X(7)$}}&{\rm{\footnotesize $3$}}&{\rm{\footnotesize $24$}}&{\rm{\footnotesize $PSL_2(7)$}}\hfill\\
\hline
{\rm{\footnotesize $X(11)$}}&{\rm{\footnotesize $26$ }}&{\rm{\footnotesize }}&{\rm{\footnotesize }}\hfill\\
\hline
{\rm{\footnotesize classical modular curve $X_0(N)$}}&{\rm{\footnotesize }}&{\rm{\footnotesize }}&{\rm{\footnotesize }}\hfill\\
\hline
{\rm{\footnotesize $X_1(N)$}}&{\rm{\footnotesize $0$, for $N=1,\cdots,10,12$}}&{\rm{\footnotesize }}&{\rm{\footnotesize }}\hfill\\
\hline
\multicolumn{4}{l}{\rm{\footnotesize $SL_2(N)/\{\pm 1\}\cong PSL_2(N)$ for $N$ prime.}}\hfill\\
\end{tabular}$}
\end{table}

\appendix{\bf Appendix C: Modular functions, forms and cusps.}\label{appendix-b}
\renewcommand{\theequation}{C.\arabic{equation}}
\setcounter{equation}{0}  

In this appendix we summarize some fundamental definitions and results concerning modular functions, modular forms and cusp forms that are used in the paper.

\begin{cdefinition}\label{c-weakly-modular-function}
Let $H=\{z\in\mathbb{C}\, |\, \Im(z)>0\}$. $H$ can be identified with the open disk $\Delta$. Set
$$G=GL_2^+(\mathbb{R})=\{g\in\left(
                               \begin{array}{cc}
                                 a & b \\
                                 c & d \\
                               \end{array}
                             \right),\, a,\,b,\,c,\, d\in\mathbb{R}\}.$$
Let us consider the action $G\times H\to H$ by means of fractional linear transformations of $H$: $(g,z)\mapsto g(z)=\frac{az+b}{cz+d}$. Set $\Gamma\subset \Gamma(1)=SL_2(\mathbb{Z})\subset G$ a subgroup of the modular group $\Gamma(1)$. A function $f:H\to\mathbb{C}$ is called {\em weakly modular function of weight $k\in\mathbb{Z}$} on $\Gamma$, if are satisfied the following conditions:

{\em (i)} $f$ is meromorphic;

{\em (ii)} $f(g(z))=j(g,z)^k\, f(z)$ for all $g\in\Gamma$, with $j(g,z)=cz+d$.
\end{cdefinition}

\begin{cproposition}\label{c-weakly-modular-function-properties}
Every meromorphic function $f:H\to \mathbb{C}$ is weakly modular of weight $k$ for some $\Gamma\le GL_2^+(\mathbb{R})$.
\end{cproposition}

\begin{cdefinition}\label{c-modular-function-form-cusp}

$\bullet$\hskip 2pt  A function $f:H\to\mathbb{C}$ is called {\em modular function of weight $k\in\mathbb{Z}$} on $\Gamma$, if are satisfied the following conditions:

{\em (i)} $f$ is a weakly modular function of weight $k$;

{\em (ii)} $f$ is meromorphic at $\infty$.

The set of modular functions of weight $k$ is denoted by $\mathbf{A}_k=\mathbf{A}_k[\Gamma]$.

$\bullet$\hskip 2pt  A {\em modular form of weight $k\in\mathbb{Z}$} on $\Gamma$, is a modular function of weight $k$ that is holomorphic on $H$ and at $\infty$.

The set of modular forms of weight $k$ is denoted by $\mathbf{M}_k=\mathbf{M}_k[\Gamma]$.

$\bullet$\hskip 2pt  A {\em cusp form of weight $k\in\mathbb{Z}$} on $\Gamma$, is a modular form of weight $k$ that vanishes at $\infty$.

The set of cusp forms of weight $k$ is denoted by $\mathbf{S}_k=\mathbf{S}_k[\Gamma]$.
\end{cdefinition}

\begin{cproposition}\label{c-modular-function-form-cusp-properties}
$\bullet$\hskip 2pt  $\mathbf{A}_k$, $\mathbf{M}_k$ and $\mathbf{S}_k$ are $\mathbb{C}$-vector spaces and one has:

{\em (i)} $\mathbf{A}=\bigoplus_{k}\mathbf{A}_k$ is a graded field:

{\em (ii)} $\mathbf{M}=\bigoplus_{k}\mathbf{M}_k$ is a graded ring:

{\em (iii)} $\mathbf{S}=\bigoplus_{k}\mathbf{S}_k$ is a graded ideal of $\mathbf{M}$.

The functions in $\mathbf{A}$, $\mathbf{M}$ and $\mathbf{S}$ do not satisfy the transformation properties {\em(ii)} in Definition C\ref{c-weakly-modular-function}.

$\bullet$\hskip 2pt  $\mathbf{M}=\mathbb{C}[E_4,E_6]$. (For the definition of $E_4$ and $E_6$ see Tab. \ref{c-table-example-modular-functions-forms-cusps}.)

$\bullet$\hskip 2pt  $\mathbf{S}=\triangle\cdot\mathbf{M}$. (For the definition of $\triangle$ see Tab. \ref{c-table-example-modular-functions-forms-cusps}.)

$\bullet$\hskip 2pt  If $k$ is an even integer then $\mathbf{A}_k$ is a one-dimensional $\mathbf{A}_0$-vector space generated by $(E_6/E_4)^{k/2}$.

$\bullet$\hskip 2pt  If $k$ is an odd integer then $\mathbf{A}_k=\{0\}$.

$\bullet$\hskip 2pt  $\mathbf{A}_0=\mathbb{C}(j)$, the space of rational functions generated by the $j$-invariant.

$\bullet$\hskip 2pt  $\mathbf{A}=\mathbb{C}(E_4,E_6)$ is the quotient field of $\mathbf{M}$.
\end{cproposition}

\begin{table}[t]
\caption{Examples of modular functions, modular forms and cusp forms.}
\label{c-table-example-modular-functions-forms-cusps}
\scalebox{0.9}{$\begin{tabular}{|l|l|l|}
\hline
\hfil{\rm{\footnotesize Name}}\hfil&\hfil{\rm{\footnotesize Definition}}\hfil&\hfil{\rm{\footnotesize Properties}}\hfil\\
\hline\hline
{\rm{\footnotesize Eisenstein series}}&{\rm{\footnotesize $G_k(z)=\sum'_{m,\,n\in\mathbb{Z}}\frac{1}{(mz+n)^k}$}}&{\rm{\footnotesize $\bullet$ converges absolutely for $k\ge 3$}}\hfill\\
{\rm{\footnotesize }}&{\rm{\footnotesize }}&{\rm{\footnotesize $\bullet$ $G_k\in\mathbf{M}_k$, for $k\ge 3$}}\hfill\\
{\rm{\footnotesize }}&{\rm{\footnotesize }}&{\rm{\footnotesize $\bullet$ $G_k=0$ for $k\equiv\, 1\, \hbox{\rm mod $2$}$}}\hfill\\
{\rm{\footnotesize }}&{\rm{\footnotesize }}&{\rm{\footnotesize $\bullet$ $G_k(z)=\zeta(k)\sum_{m,n; gcd(m,n)=1}\frac{1}{(mz+n)^k}$}}\hfill\\
{\rm{\footnotesize }}&{\rm{\footnotesize }}&{\rm{\footnotesize $\bullet$ $G_k(z)=2\zeta(k)E_k(z)$ for $k\equiv\, 0\, \hbox{\rm mod $2$}$}}\hfill\\
\hline
{\rm{\footnotesize Poincar\'e series}}&{\rm{\footnotesize $P_{m,k}(z)=\sum_{\gamma\in\Gamma_\infty\setminus\Gamma}\frac{1}{j(\gamma,z)^k}e^{2\pi im\gamma(z)}$}}&{\rm{\footnotesize $P_{0,k}(z)=G_k(z)$}}\hfill\\
{\rm{\footnotesize }}&{\rm{\footnotesize }}&{\rm{\footnotesize $P_{m,k}(z)\in\mathbf{S}_k$, $m>0,\, k\ge 3$.}}\hfill\\
\hline
{\rm{\footnotesize Discriminant form}}&{\rm{\footnotesize $\triangle=g_2^3-27 g_3^2$}}&{\rm{\footnotesize $\triangle\in \mathbf{S}_{12}$}}\hfill\\
{\rm{\footnotesize }}&{\rm{\footnotesize $g_2=60G_4=\frac{4\pi^4}{3}E_4$}}&{\rm{\footnotesize $\triangle(z)=(2\pi)^{12}\sum_{n\ge 1}\tau(n)\, e^{2n\pi iz}$}}\hfill\\
{\rm{\footnotesize }}&{\rm{\footnotesize $g_3=140G_6=\frac{8\pi^4}{27}E_6$}}&{\rm{\footnotesize $\tau(n)$ Ramanujan function}}\hfill\\
\hline
{\rm{\footnotesize $j$-invariant}}&{\rm{\footnotesize $j(z)=1728\frac{g_2^3}{\triangle}$}}&{\rm{\footnotesize $j(z)\in\mathbf{A}_{0}$}}\hfill\\
{\rm{\footnotesize }}&{\rm{\footnotesize }}&{\rm{\footnotesize $j(z)=1728\frac{E_4^3}{E^3_4-E^2_6}$}}\hfill\\
\hline
\multicolumn{3}{l}{\rm{\footnotesize $\sum'$ denotes that the term $(m,n)=(0,0)$ has been omitted.}}\hfill\\
\multicolumn{3}{l}{\rm{\footnotesize $E_k(z)=1+c_k\sum_{1\le n\le \infty}\sigma_{k-1}(n)e^{2n\pi iz}$}}\hfill\\
\multicolumn{3}{l}{\rm{\footnotesize $\sigma_{k-1}(n)=\sum_{n|d}d^{k-1}$; $c_k=-\frac{2k}{B_k}\mathop{=}\limits^{Euler}\frac{(2\pi i)^k}{(k-1)!\zeta(k)}$.}}\hfill\\
\multicolumn{3}{l}{\rm{\footnotesize $B_k=k^{th}$ Bernoulli number. $\sum_{0\le k\le\infty}B_k\frac{z^k}{k!}=\frac{z}{e^z-1}$}}\hfill\\
\multicolumn{3}{l}{\rm{\footnotesize $\Gamma_\infty=\{\left(
                                                        \begin{array}{cc}
                                                          1 & n \\
                                                          0 & 1 \\
                                                        \end{array}\right)\, |\, n\in\mathbb{Z}\}\le \Gamma(1)$}}\hfill\\
\multicolumn{3}{l}{\rm{\footnotesize $g_2$ is a modular form of weight $4$.}}\hfill\\
\multicolumn{3}{l}{\rm{\footnotesize $j$ is a surjective meromorphic function $H\to\mathbb{C}$, invariant under $SL_2(\mathbb{Z})$-action.}}\hfill\\
\multicolumn{3}{l}{\rm{\footnotesize $j$ gives a bijection between isomorphism classes of elliptic curves over $\mathbb{C}$ and complex numbers.}}\hfill\\
\end{tabular}$}
\end{table}

\begin{cexample}\label{c-example-modular-functions-forms-cusps}
In Tab. \ref{c-table-example-modular-functions-forms-cusps} are reported some examples of modular functions, modular functions and cusp forms that are useful for a direct understanding of this paper.
\end{cexample}
\end{appendices}


\begin{thebibliography}{2020}

\bibitem{BIRCH-SWINNERTON-DYER} B. J. Birch and H. P. F. Swinnerton-Dyer, \textit{Notes on elliptic curves. II}, J. Reine Angewandte Math. \textbf{218}(1965), 79�-108.


\bibitem{BREUIL-CONRAD-DIAMOND-TAYLOR} C. Breuil, B. Conrad, F. Diamond, R. Taylor, \textit{On the modularity of elliptic curves over $\mathbb{Q}$: wild $3$-adic exercises}, Journal of the American Mathematical Society \textbf{14 (4)}(2001), 843�-939, doi:10.1090/S0894-0347-01-00370-8.

\bibitem{CANTOR1} G. Cantor, \textit{Beiträge zur Begründung der transfiniten Mengenlehre. I}. Mathematische Annalen \textbf{46}(1895), 481-–512. doi:10.1007/bf02124929.

\bibitem{CANTOR2} G. Cantor, \textit{Beiträge zur Begründung der transfiniten Mengenlehre. II}. Mathematische Annalen \textbf{49}(1897), 207-–246. doi:10.1007/bf01444205.

\bibitem{CANTOR3} G. Cantor, \textit{Mitteilungen zur Lehre vom Transfiniten}, Zeitschrift für Philosophie und philosophische Kritik \textbf{91}(1887), 81–-125.

\bibitem{COATES-WILES}  J. Coates and A. Wiles, \textit{On the conjecture of Birch and Swinnerton-Dyer}, Invent. Math. \textbf{39}(1977), 223�-251.


\bibitem{COPPEL}  W. A. Coppel, \textit{Number Theory: An Introduction to Mathematics. Part B.} Springer-Verlag, New York,
2006.

\bibitem{DARMON}  H. Darmon, \textit{Shimura�Taniyama conjecture}, in Hazewinkel, M., Encyclopedia of Mathematics, Springer, 2001, ISBN 978-1-55608-010-4.

 \bibitem{DICKSON}  L. E. Dickson, \textit{History of Theory of Numbers} Springer-Verlag, New York,
1952.

\bibitem{FENG}  K. Feng, \textit{Non-congruent Numbers, Odd graphs and the B-S-D Conjecture}  Acta Arithmetica: \textbf{LXXV(1)}, 1996.

\bibitem{HAZEWINKEL} M. Hazewinkel, Encyclopedia of Mathematics, Springer, 2001, ISBN 978-1-55608-010-4.

\bibitem{HINKIS} A. Hinkis, \textit{Proofs of the Cantor-Bernstein theorem. A mathematical excursion}, Science Networks. Historical Studies \textbf{45}, Heidelberg: Birkhäuser/Springer, (2013). doi:10.1007/978-3-0348-0224-6, ISBN 978-3-0348-0223-9.


\bibitem{KLEIN-FRICKE} F. Klein and R. Fricke, \textit{Vorlesungen \"uber die Theorie der elliptischen Modularfunktionen}, vol. 1-2, Teubner (1890--1892).

\bibitem{KOBLITZ} N. Koblitz, \textit{Introduction to Elliptic Curves and Modular Forms}, 2nd ed., Graduate Texts in Mathematics, vol. 97, Springer-Verlag, New York, 1994.

\bibitem{IWANIEC} H. Iwaniec, \textit{Topics in Classical Automorphic Forms}, Amer. Math. Soc., Providence, 1997.

\bibitem{LUTZ} E. Lutz, \textit{Sur l'\'{e}quation $y^2 = x^3 - Ax - B$ dans les corps $p$-adiques}, J. Reine Angew. Math. \textbf{177}(1937), 237�-247.

\bibitem{MAZUR} B. C. Mazur, \textit{Modular curves and Eisenstein ideal}, Publ. Math. Inst. Hautes Etudes Sci.\textbf{47}(1977), 33�-186.


\bibitem{MEREL} L. Merel, \textit{Bornes pour la torsion des courbes elliptiques sur les corps de nombres}, Inventiones Mathematicae \textbf{124(1�3)}(1996), 437-�449.

\bibitem{MORDELL} L. J. Mordell, \textit{On the rational solutions of the indeterminate equations of the third and fourth degrees}, Proc. Cambride. Phil. Soc. \textbf{21}(1922-23), 179--192.

\bibitem{LEONARDO-PISANO} L. Pisano, \textit{Liber Quadratorum}, 1225. Republished by Prince Boncompagni in the year 1856.


\bibitem{PRAS1} A. Pr\'astaro, \textit{Geometry of PDEs and Mechanics}, World Scientific Publishing, River Edge, NJ, 1996, 760 pp. ISBN 9810225202.

\bibitem{SERRE} J.-P. Serre, \textit{A Course in Arithmetic}, Springer-Verlag, New York, 1973.


\bibitem{SHIMURA} G. Shimura, \textit{Yutaka Taniyama and his time. Very personal recollections}, Bull. London Math. Soc. \textbf{21 (2)} (1989), 186�-196. doi:10.1112/blms/21.2.186.

\bibitem{SCHOENEBERG} B. Schoeneberg, \textit{EllipticModular Functions},  Springer-Verlag, New York, 1974.


\bibitem{SILVERMAN} J. H. Silverman, \textit{The arithmetic of elliptic curves}, Springer, 1986. ISBN 0-387-96203-4.

\bibitem{SILVERMAN-TATE} J. H. Silverman and J. Tate, \textit{Rational Points on Elliptic Curves}, Springer, 1994. ISBN 0-387-97825-9.

\bibitem{STEPHENS} N. M. Stephens, \textit{Congruence properties of congruent numbers}, Bull. London Math. Soc. \textbf{7} (1975), 182�-184. doi:10.1112/blms/21.2.186.

\bibitem{TANIYAMA} Y. Taniyama, \textit{Problem $12$}, Sugaku (in Japanese) \textbf{7}(1956), 269;  (English translation in Shimura 1989, p. 194).

\bibitem{TATE} J. Tate, \textit{The arithmetic of elliptic curves}, Invent. Math. \textbf{23}(1974), 179--206.

\bibitem{TUNNELL} J. B. Tunnell, \textit{A classical diophantine problem and modular forms of weight $3/2$}, Invent.
Math. \textbf{72(2)} (1983), 323�-334.

\bibitem{WEIL1} A. Weil, \textit{L'arithm\'etique sur les courbes alg\'ebriques}, Acta Math \textbf{52}(1929), 281--315.

\bibitem{WEIL2} A. Weil, \textit{\"Uber die Bestimmung Dirichletscher Reihen durch Funktionalgleichungen}, Mathematische Annalen \textbf{168}(1967), 149�-156. doi:10.1007/BF01361551, ISSN 0025-5831; \textit{Number Theorey: An Approach Through History. Birkh\"auser, 1984.}

\bibitem{WILES1} A. Wiles, \textit{Modular elliptic curves and Fermat's last theorem}, Ann. Math. \textbf{142 (3)}(1995), 443�-551.

\bibitem{WILES2} A. Wiles, \textit{The Birch and Swinnerton-Dyer Conjecture}. Official Problem Description, at the Clay Mathematics Institute.

    \href{http://www.claymath.org/millenium-problems/birch-and-swinnerton-dyer-conjecture}{http://www.claymath.org/millenium-problems/birch-and-swinnerton-dyer-conjecture}.
\end{thebibliography}
\end{document}